\documentclass[12pt,reqno,a4paper]{article}

\usepackage[outer=2.5cm,top=2.5cm,bottom=2.5cm,inner=2.0cm]{geometry}

\usepackage{hyperref} 
\usepackage{amsmath}
\usepackage{amsthm}
\usepackage{amssymb}
\usepackage{amsfonts}
\usepackage{dsfont}
\usepackage{stmaryrd}
\usepackage{commath}
\usepackage{mathrsfs}
\usepackage{enumitem}
\usepackage{graphicx}
\usepackage{tikz}
\usepackage{algorithm}
\usepackage{algorithmic}
\usepackage{pgfplots}
\usepackage{subfig}
\usepackage[numbers]{natbib}

\usepackage{todonotes}
\usepackage{placeins}

\font\msbm=msbm10

\numberwithin{equation}{section}

\theoremstyle{plain}
\newtheorem{theorem}{Theorem}[section]

\newtheorem{corollary}[theorem]{Corollary}
\newtheorem{proposition}[theorem]{Proposition}

\theoremstyle{definition}

\newtheorem{example}[theorem]{Example}
\newtheorem*{example*}{Example}
\newtheorem{remark}[theorem]{Remark}
\def\mathbb#1{\hbox{\msbm{#1}}}

\newcommand{\field}[1]{\ensuremath{\mathds{#1}}}

\newcommand{\re}{\field R}\newcommand{\N}{\field N}
\newcommand{\Ept}{\field E}\newcommand{\Prob}{\field P}

\newcommand{\C}{\field{C}}
\newcommand{\tor}{\field T} %

\newcommand{\Z}{\field Z}

\newcommand{\R}{\field R}
\newcommand{\T}{\field T}

\newcommand{\Id}{\ensuremath{\mathrm{Id}}}

\newcommand{\supp}{{\rm  supp \, }}

\newcommand{\bk}{\mathbf{k}}
\newcommand{\bj}{\mathbf{j}}
\newcommand{\bx}{\mathbf{x}}
\newcommand{\by}{\mathbf{y}}

\newcommand{\bL}{\mathbf{L}}

\newcommand{\bH}{\mathbf{H}}
\newcommand{\bX}{\mathbf{X}}

\newcommand{\bU}{\mathbf{U}}
\newcommand{\bV}{\mathbf{V}}

\newcommand{\<}{\langle}
\renewcommand{\>}{\rangle}

\newcommand{\be}{\begin{equation}}
\newcommand{\ee}{\end{equation}}
\newcommand{\beq}{\begin{eqnarray}}
\newcommand{\beqq}{\begin{eqnarray*}}
\newcommand{\eeq}{\end{eqnarray}}
\newcommand{\eeqq}{\end{eqnarray*}}

\newcommand{\Hsm}{\ensuremath{H^s_{\textnormal{mix}}}}

\newcommand{\IntEx}[1]{\operatorname{Int}_{#1}}

\newcommand{\cP}{\mathcal{P}}

\newcommand{\eps}{\varepsilon}

\newcommand{\diag}{\mathrm{diag}}

\newcommand{\btL}{\mathbf{\tilde{L}}}

\begin{document}

\title{Worst-case recovery guarantees for least squares approximation using random samples}
\author{Lutz K\"ammerer\footnote{Corresponding author: kaemmerer@mathematik.tu-chemnitz.de}, Tino Ullrich, Toni Volkmer
 \\\\
TU Chemnitz, Faculty of Mathematics, 09107 Chemnitz, Germany}

\maketitle

\begin{abstract}
We construct a least squares approximation method for the recovery of complex-valued functions from a reproducing kernel Hilbert space on $D \subset \R^d$. The nodes are drawn at random for the whole class of functions and the error is measured in $L_2(D,\varrho_D)$. We prove worst-case recovery guarantees by explicitly controlling all the involved constants. This leads to new preasymptotic recovery bounds with high probability for the error of {\em Hyperbolic Fourier Regression} on multivariate data. In addition, we further investigate its counterpart {\em Hyperbolic Wavelet Regression} also based on least-squares to recover non-periodic functions from random samples. Finally, we reconsider the analysis of a cubature method based on plain random points with optimal weights and reveal near-optimal worst-case error bounds with high probability. It turns out that this simple method can compete with the quasi-Monte Carlo methods in the literature which are based on lattices and digital nets. 

\small
\medskip
\noindent {\textit{Keywords and phrases}} : 
least squares approximation, random sampling, quadrature, sampling recovery, hyperbolic wavelet regression

\medskip

\small%
\noindent {\textit{2010 AMS Mathematics Subject Classification}} : 
41A10, %
41A25, %
41A55, %
41A60, %
41A63, %
42A10, %
65D32, %
65Txx, %
68Q25, %
68W40, %
94A20  %

\end{abstract}

\section{Introduction} %
We consider the problem of learning complex-valued multivariate functions on a domain $D \subset \R^d$ from function samples on the set of nodes $\bX :=\{\bx^1,...,\bx^n\} \subset D$. The functions are modeled as elements from some reproducing kernel Hilbert space $H(K)$ with kernel $K\colon D\times D\to \C$. The nodes $\bX$ are drawn independently according to a tailored probability measure depending on the spectral properties of the embedding of $H(K)$ in $L_2(D,\varrho_D)$, where the error is measured. Our main focus in this paper is on worst-case recovery guarantees. In fact, we aim at recovering all $f \in H(K)$ simultaneously from sampled values at the sampling nodes in $\bX$ with high probability. To be more precise, we implement algorithms and provide error bounds to control the worst-case error
$$
    \sup\limits_{\|f\|_{H(K)}\leq 1} \|f-S_{\bX}^mf\|_{L_2(D,\varrho_D)}\,,
$$
where $S_{\bX}^m$ is the fixed recovery operator.
In contrast to that, the problem of reconstructing an individual function from random samples has been considered by several authors in the literature, e.g.\ Smale, Zhou \cite{SmZh04}, Bohn \cite{Bo17, Bo18_1}, Bohn, Griebel \cite{BoGr17}, Cohen, Davenport, Leviatan~\cite{CoDaLe13}, Chkifa, Migliorati, Nobile, Tempone \cite{ChkCoMiNoTe15}, Cohen, Migliorati \cite{CoMi17}, and many others to mention just a few.  

Let us emphasize, that we do not develop a Monte Carlo method here. It is rather the use of ``random information'' which gained substantial interest  in the Information Based Complexity (IBC) community and in the field of Compressed Sensing, see the recent survey \cite{HiKrNoPrUl19} and \cite{FoRa13}. We construct a recovery operator $S_{\bX}^m$ which computes a best least squares fit $S_{\bX}^mf$ to the given data $(f(\bx^1),...,f(\bx^n))^\top$ from the finite-dimensional space spanned by the first $m-1$ singular vectors of the embedding
\begin{equation}\label{f000}
    \Id : H(K) \to L_2(D,\varrho_D)\,.
\end{equation}
The right singular vectors $e_1(\cdot), e_2(\cdot),...$ of this embedding are arranged according to their importance, i.e., with respect to the non-increasing rearrangement of the singular values $\sigma_1 \geq \sigma_2 \geq \cdots >0$. 

The investigations in this paper are inspired by the recent results by Krieg and M.~Ullrich~\cite{KrUl19}, which triggered substantial progress in the field. See also the discussion below in Remark~\ref{DavMar} and \cite[Sect.\ 7]{NSU20}. In this paper we extend and improve the results from \cite{KrUl19} in several directions. In particular, we investigate and implement a least squares regression algorithm under weaker conditions and give practically useful parameter choices which lead to a controlled failure probability and explicit error bounds. 

A typical error bound relates the worst-case recovery error to the sequence of singular numbers $(\sigma_k)_{k\in\N}$ of the embedding \eqref{f000} which represent the approximation numbers or linear widths.
One main contribution of this paper is the following general bound, where all constants are determined precisely under mild conditions. Recall that  $(e_k(\cdot))_{k\in \N}$ denotes the sequence of right singular vectors of the embedding \eqref{f000}, i.e., the eigenfunctions of $\Id^*\circ\Id:H(K) \to H(K)$, and $\sigma_1 \geq \sigma_2 \geq \cdots >0$ the corresponding singular numbers. 

\begin{theorem}[cf.\ Cor.\ \ref{cor10}]\label{thm:intro} Let $H(K)$ be a separable reproducing kernel Hilbert space of complex-valued functions on a subset $D \subset \R^d$ such that the positive semidefinite kernel $K:D\times D\to \C$ satisfies $\sup_{\bx\in D}K(\bx,\bx)<\infty$. Let further $\varrho_D$ denote a probability measure on $D$. Furthermore, for $n\in \N$ and $\delta\in(0,1/3)$, we define $m\in \N$ such that
$$
N(m):=\sup_{\bx \in D}\sum\limits_{k=1}^{m-1}\sigma_k^{-2}\,|e_k(\bx)|^2 \leq \frac{n}{48(\sqrt{2}\log(2n)-\log\delta)}
$$
holds. Then the random reconstruction operator $S^m_{\bX}$ (see Algorithm \ref{algo1} below), which uses samples on the $n$ i.i.d. (according to $\varrho_D$) drawn nodes in $\bX$, satisfies
\begin{equation*}
  \Prob\left(\sup\limits_{\|f\|_{H(K)}\leq 1} \|f-S^m_\bX f\|^2_{L_2(D,\varrho_D)}
  \leq \frac{29}{\delta}\,\max\Big\{\sigma^2_m,\frac{\log(8n)}{n}T(m) \Big\}\right) \geq 1-3\delta\,,
\end{equation*}
where $T(m) := \sup\limits_{\bx \in D}\sum\limits_{k=m}^{\infty}|e_k(\bx)|^2$.
\end{theorem}
The occurrence of the fundamental quantity $N(m)$ is certainly not a surprise. It is also known as spectral function (see \cite{Gr19} and the references therein) and in case of orthogonal polynomials it is the inverse of the infimum of the Christoffel function (cf.\ e.g.~\cite{DuXu2001}). It represents a well-known ingredient for inequalities related to sampling and discretization, see for instance Gr\"ochenig and Bass \cite{GrBa13}, Gr\"ochenig \cite{Gr99, Gr19}, Temlyakov \cite{Te18,Te19}, and Temlyakov et al.~\cite{Te19b}. If, for instance, $N(m) \in \mathcal{O}(m)$ holds, we achieve near optimal error bounds with respect to the number of used sampling values $n$. Note that by a straight-forward computation we also
have $T(m) \leq 2\sum_{k\geq m/2}^{\infty} \sigma_k^2 N(4k)/k\,.$ Hence, if $N(m) \in \mathcal{O}(m)$, the bound in Theorem \ref{thm:intro} is upper bounded by 
\begin{equation}\label{1.3}
	\sup\limits_{\|f\|_{H(K)}\leq 1} \|f-S^m_\bX f\|^2_{L_2(D,\varrho_D)} \leq \frac{C_{K,\varrho_D}}{\delta\cdot m}\sum\limits_{k\geq m/2}^{\infty}
	\sigma_k^2 
\end{equation}
with $m:=n/(c_1\log(n)+c_2\log(\delta^{-1}))$ and a constant $C_{K,\varrho_D}>0$ depending on the measure $\varrho_D$ and the kernel $K$.  

In the general case we do not necessarily have $N(m) \in \mathcal{O}(m)$. Here, a technique called ``importance sampling'', see e.g.\ \cite{Hi10,RaWa12,RaWa16,NeWa16}, turns out to be useful, see Algorithm \ref{algo1:reweighted} below. As proposed in \cite{CoMi17} and specified in full detail in~\cite{KrUl19}, one may sample from a reweighted distribution using a specific density $\varrho_m$, defined in \eqref{density1} below, which is different for any $m$ and depends on the spectral properties of the embedding \eqref{f000}. It determines the important ``area'' to sample. In other words, we incorporate additional knowledge about the spectral properties of our embedding. The underlying recovery operator is still constructive and the determined error bounds hold with high probability. Computing this envelope density (and sample from it) has been studied in \cite[Sect.\, 5]{CoMi17}. A refinement of this technique together with Theorem \ref{thm:intro} above leads to the following precise bounds under even weaker conditions.

\begin{theorem}[cf.\ Thm.\ \ref{sampling_numbers}]\label{thm:intro3}
Let $H(K)$ be a separable reproducing kernel Hilbert space of complex-valued functions on a subset $D \subset \R^d$.
Let further $\varrho_D$ denote a non-trivial $\sigma$-finite measure on $D$ and we assume that the positive semidefinite kernel $K:D\times D\to \C$ satisfies 
$$
\int_{D} K(\bx,\bx) \,\varrho_D(\mathrm{d}\bx) < \infty\,.
$$
Furthermore, for $n\in \N$ and $\delta\in(0,1/3)$ we fix
$$
m := \left\lfloor\frac{n}{96(\sqrt{2}\log(2n)-\log\delta)}\right\rfloor.
$$
Then the random reconstruction operator $\widetilde{S}^m_\bX$ (see Algorithm \ref{algo1:reweighted} below), which uses $n$ samples drawn according to a probability measure depending on $\varrho_D, m$ and $K$ satisfies
\begin{equation}
  \Prob\left(\sup\limits_{\|f\|_{H(K)}\leq 1} \|f-\widetilde{S}^m_\bX f\|^2_{L_2(D,\varrho_D)}
  \leq \frac{50}{\delta}\,\max\Big\{\sigma^2_m,\frac{\log(8n)}{n}\sum_{j=m}^\infty\sigma_j^2 \Big\}\right) \geq 1-3\delta\,.\label{eq:intro_prob_general}
\end{equation}
\end{theorem}

By the same reasoning as above we may replace the error bound in \eqref{eq:intro_prob_general} by \eqref{1.3} but this time with a universal (and precisely determined) constant $C>0$. The result refines the bound in \cite{KrUl19} as we give precise constants here in the general situation. 

A further application of our least squares method is in the field of numerical integration. Oettershagen \cite{Oe17}, Belhadji, Bardenet, Chainais \cite{BeBaCh19}, Groechenig \cite{Gr19}, Migliorati, Nobile \cite{MiNo18}, and many others used least squares optimization to design quadrature rules. This results in (complex-valued) weights $\mathbf{q}:=(q_1,...,q_n)^\top$ in a cubature formula, i.e.,
$$
    \widetilde{Q}_{\bX}^mf = \mathbf{q}^\top\cdot \mathbf{f} = \sum\limits_{j = 1}^n q_j\, f(\bx^j) := \int_D \widetilde{S}^ m_{\bX}f \, \mathrm{d}\mu_D\,,
$$
where
$\mathbf{f}:=(f(\bx^j))_{j=1}^n$ and 
$\mu_D$ is the measure for which we want to compute the integral. In our setting, the integration nodes $\bX = \{\bx^1,...,\bx^n\}$ are determined once in advance for the whole class. Clearly, the bounds from Theorem \ref{thm:intro}, \ref{thm:intro3} can be literally transferred to control the worst-case integration error 
$$
	\sup\limits_{\|f\|_{H(K)}\leq 1} \Big|\int_D f(\bx)\,d\bx - \widetilde{Q}^m_{\bX}f\Big|\,,
$$
see Theorems \ref{thm:int_error}, \ref{thm:int_error2}.  

As the main example we consider the recovery of functions from Sobolev spaces with mixed smoothness (also known as tensor product Sobolev spaces or hyperbolic cross spaces). This problem has been investigated by many authors in the last 30 years, see \cite{DuTeUl19} and the references therein. The above general bound on the worst case errors can for instance be used for any non-periodic embedding
$$
    \Id:H^s_{\text{mix}}([0,1]^d) \to L_2([0,1]^d)\,,\quad s>1/2\,.
$$
The spaces $H^s_{\text{mix}}([0,1]^d)$ can be represented in various ways as a reproducing kernel Hilbert space satisfying the requirements of the above theorems, see the concrete collection of examples in \cite[Section 7.4]{BeTh04}. Applying Theorem \ref{thm:intro3} and plugging in well-known upper bounds on the singular numbers we improve on the asymptotic sampling bounds in \cite[Sect.\ 5]{DuTeUl19}, see also Dinh D\~ung~\cite{Du11} and Byrenheid \cite{By18} for the non-periodic situation. In addition, using refined preasymptotic estimates for the $(\sigma_j)_{j\in\N}$ in the non-periodic situation (see \cite[Section 4.3]{Kr18}) yields reasonable bounds for sampling numbers in case of small $n$. 

Let us emphasize that the result by Krieg, Ullrich \cite{KrUl19} can be considered as a major progress for the research on the complexity of this problem. They disproved Conjecture 5.6.2. in \cite{DuTeUl19} for $p=2$ and $1/2 < s <(d-1)/2$. Indeed, the celebrated sparse grid points are outperformed in a certain range for $s$. This represents a first step towards the solution of \cite[Outstanding Open Problem 1.4]{DuTeUl19}. As a consequence of the recent contributions by Nagel, Sch\"afer, T. Ullrich \cite{NSU20} and Temlyakov \cite{Te20_2} based on the groundbreaking solution of the Kadison Singer problem by Marcus, Spielman and Srivastava \cite{MaSpSr15} it is now evident that sparse grid methods are not optimal in the full range of parameters (except maybe in $d=2)$. Still it is worth mentioning that the sparse grids represent the best known deterministic construction what concerns the asymptotic order. Indeed, the guarantees are deterministic and only slightly worse compared to random nodes in the asymptotic regime. However, regarding preasymptotics the random constructions provide substantial advantages. 

In this paper, we use the simple least squares algorithm from~\cite{CoDaLe19, CoMi17, Bo17, KrUl19}, and we show that using random points makes it also possible to obtain explicit worst-case recovery bounds also for small $n$. In the periodic setting the analysis benefits from the fact that the underlying eigenvector system is a bounded orthonormal system (BOS), see \eqref{f20}, which implies in particular $N(m) \in \mathcal{O}(m)$. In case of the complex Fourier system we have a BOS constant $B=1$ and obtain dimension-free constants. This allows for a priori estimates on the number of required samples and arithmetic operations in order to ensure accuracy $\varepsilon>0$ with our concrete algorithm. In particular, we incorporate recent preasymptotic bounds for the singular values $(\sigma_k)_{k\in\N}$, see \cite{KuSiUl15, KSU3} and \cite{Kr18}. We obtain with probability larger than $1-3\delta$ for the periodic mixed smoothness space $H^s_{\text{mix}}(\T^d)$ with $2s>1+\log_2 d$ equipped with the $\|\cdot\|_{\#}$-norm (see \eqref{equ:weight_pound} below) the worst-case bound
$$
\sup\limits_{\|f\|_{\#}\leq 1}\|f-S^{m,\#}_\bX f\|^2_{L_2(\T^d)}
  \leq \frac{10s}{\delta(2s-1-\log_2d)}\left(\frac{16}{3m}\right)^{\frac{2s}{1+\log_2d}}\,.
$$
The number of samples $n$ scales similarly as after \eqref{1.3} with precisely determined absolute constants $c_1,c_2 < 70$, see Corollary \ref{cor:hsmix_preasymp} below. 

We also demonstrate in Section \ref{HWR} that the BOS assumption is not necessary for getting a practical algorithm. Similar as in \cite{Bo17} we use an algorithm called ``Hyperbolic Wavelet Regression'' and show that it recovers non-periodic functions belonging to a Sobolev space with mixed smoothness $H^s_{\text{mix}}$ from the $n$ nodes $\bX$ drawn according to the uniform measure with a rate similar as for the periodic case. The proposed approach achieves rates which improve on the bounds in \cite{Bo17} and are only worse by $\log^s n$ in comparison to the optimal rates achieved by hyperbolic wavelet best approximation studied in \cite{DeVKoTe98} and \cite{SU09}.

Finally, by our refined analysis we were able to settle a conjecture in \cite{Oe17} that the worst-case integration error for $H^s_{\text{mix}}(\T^d)$ is bounded in order by $n^{-s}\log(n)^{ds}$ with high probability. This conjecture was based on the outcome of several numerical experiments (described in~\cite{Oe17}) where the worst-case error has been simulated using the RKHS Riesz representer. It is remarkable in two respects. First, it is possible to benefit from higher order smoothness although we use plain random points. And second, this simple method can compete with most of the quasi-Monte Carlo methods based on lattices and digital nets studied in the literature, see \cite[pp.\ 195, 247]{DiKuSl13}. Moreover, if $s<(d-1)/2$ we get a better asymptotic rate than sparse grid integration which is shown to be of order $n^{-s}\log(n)^{(d-1)(s+1/2)}$, see \cite{DuUl15}.  

We practically verify our theoretical findings with several numerical experiments. There, we compare the recovery error for the least squares regression method $S_{\bX}^m$ to the optimal error given by the projection on the eigenvector space. We also study a non-periodic regime, where we randomly sample points according to the Chebyshev measure. Algorithmically, the coefficients $\mathbf{c}:=(c_k)_{k=1}^{m-1}\in\C^{m-1}$ of the approximation $S_{\bX}^m:=\sum_{k=1}^{m-1} c_k\,\sigma_k^{-1}\,e_{k}$ can be obtained by computing the least squares solution of the (over-determined) linear system of equations
$\bL_m \, \mathbf{c}=(f(\bx^j))_{j=1}^n$, where $\bL_m := \left(\sigma_k^{-1}\,e_{k}(\bx^j)\right)_{j=1;\,k=1}^{n;\,m-1}\in\C^{n\times(m-1)}$.
In order to solve this linear system of equations, one can apply a standard conjugate gradient type iterative algorithm, e.g., LSQR \cite{PaSa82}. The corresponding arithmetic costs are bounded from above by $C\,R\,m\,n < C\, R\,n^2$, where $C>0$ is an absolute constant and
$R$ the number of iterations which is rather small due to the well-conditioned least squares matrices. 

\paragraph{Outline.} In Section \ref{setting} we describe the setting in which we want to perform the worst-case analysis. There we use the framework of reproducing kernel Hilbert spaces of complex-valued functions. Section \ref{sect_lsqr} is devoted to the least squares algorithm, where the worst-case analysis is given in Sections \ref{worst_case_analysis}, \ref{mixed1}, and \ref{HWR}. In the first place, we present the general results in Section~\ref{worst_case_analysis}.
Section~\ref{mixed1} considers the particular case of hyperbolic Fourier regression. In Section~\ref{HWR}
we investigate the particular case of non-periodic functions with a bounded mixed derivative and their recovery via hyperbolic wavelet regression. The main tools from probability theory, like concentration inequalities for spectral norms and Rudelson's lemma, are provided in Section~\ref{sect_prob}. The analysis of the recovery of individual functions (Monte Carlo) is given in Section \ref{sect_individual}. Consequences for optimally weighted numerical integration based on plain random points are given in Section \ref{Numint} and Section 8.2. Finally, the numerical experiments are shown and discussed in Section~\ref{sec:num}.

\paragraph{Notation.} As usual $\N$ denotes the natural numbers, $\N_0:=\N\cup\{0\}$, $\Z$ denotes the integers,
$\R$ the real numbers, and $\C$ the complex numbers. If not indicated otherwise the symbol $\log$ denotes the natural logarithm. $\C^n$ denotes the complex $n$-space, whereas $\C^{m\times n}$ denotes the set of all $m\times n$-matrices $\bL$ with complex entries. The spectral norm (i.e. the largest singular value) of matrices $\bL$ is denoted by $\|\bL\|$ or $\|\bL\|_{2\to 2}$. Vectors and matrices are usually typesetted bold with $\bx,\by\in \R^n$ or $\C^n$. For $0<p\leq \infty$ and $\bx\in \C^n$ we denote $\|\bx\|_p := (\sum_{i=1}^n
|x_i|^p)^{1/p}$ with the usual modification in the case $p=\infty$. If $T:X\to Y$ is a continuous operator we write
$\|T\colon X\to Y\|$ for its operator (quasi-)norm. For two sequences $(a_n)_{n=1}^{\infty},(b_n)_{n=1}^{\infty}\subset \R$  we write $a_n \lesssim b_n$ if there exists a constant $c>0$ such that $a_n \leq
c\,b_n$ for all $n$. We will write $a_n \asymp b_n$ if $a_n \lesssim b_n$ and
$b_n \lesssim a_n$. We write $f \in \mathcal{O}(g)$ for non-negative functions $f,g$ if there is a constant $c>0$ such that $f \leq cg$. $D$ denotes a subset of $\R^d$ and $\ell_{\infty}(D)$ the set of bounded functions on $D$ with $\|\cdot\|_{\ell_{\infty}(D)}$ the supremum norm.
\par

\section{Reproducing kernel Hilbert spaces}
\label{setting}
We will work in the framework of reproducing kernel Hilbert spaces. The relevant theoretical background can be found in \cite[Chapt.\ 1]{BeTh04} and \cite[Chapt.\ 4]{StChr08}. Let $L_2(D,\varrho_D)$ be the space of complex-valued square-integrable functions with respect to~$\varrho_D$. Here $D \subset \R^d$ is an arbitrary subset and $\varrho_D$ a measure on $D$. 
We further consider a reproducing kernel Hilbert space $H(K)$ with a Hermitian positive definite kernel $K(\bx,\by)$ on $D \times D$. The crucial property is the identity  
\begin{equation}
		f(\bx) = \langle f, K(\cdot,\bx) \rangle_{H(K)} \label{eq:eval_inner_product}
\end{equation}
for all $\bx \in D$. It ensures that point evaluations are continuous functionals on $H(K)$. We will use the notation from \cite[Chapt.\ 4]{StChr08}. In the framework of this paper, the finite trace of the kernel 
\begin{equation}\label{integrab}
	\int_{D} K(\bx,\bx) \, \varrho_D(\mathrm{d}\bx) < \infty
\end{equation}
or its boundedness  
\begin{equation} \label{CK000}
		\|K\|_{\infty} := \sup\limits_{\bx \in D} \sqrt{K(\bx,\bx)} < \infty
\end{equation}
is assumed. The boundedness of $K$ implies that $H(K)$ is continuously embedded into $\ell_\infty(D)$, i.e., 
\begin{equation*}
		\|f\|_{\ell_{\infty}(D)} \leq  \|K\|_{\infty}\cdot\|f\|_{H(K)}\,.
\end{equation*}
Note that we do not need the measure $\varrho_D$ for this embedding. 

The embedding operator 
$$\Id:H(K) \to L_2(D,\varrho_D)$$ 
is compact under the integrability condition \eqref{integrab}, which we always assume from now on. We additionally assume that $H(K)$ is at least infinite dimensional. However, we do not assume the separability of $H(K)$ here. Due to the compactness of $\Id$ the operator $\Id^{\ast}\circ \Id$ provides an at most countable system of strictly positive eigenvalues $(\lambda_j)_{j\in \N}$. These eigenvalues are summable as a consequence of \eqref{integrab} and \eqref{K(x,x)}, \eqref{C_K} below, such that the singular numbers $(\sigma_j)_{j\in \N}$ belong to $\ell_2$. Indeed, let $\Id^\ast$ be defined in the usual way as 
$$
	\langle \Id(f), g\rangle_{L_2} = \langle f, \Id^{\ast}(g) \rangle_{H(K)} \,.
$$
Then $W_{\varrho_D} := \Id^\ast\circ \Id : H(K) \to H(K)$ is non-negative definite, self-adjoint and compact. Let $(\lambda_j,e_j)_{j\in \N}$ denote the eigenpairs of $W_{\varrho_D}$, where $(e_j)_{j \in \N} \subset H(K)$ is an orthonormal system of eigenvectors, and $(\lambda_j)_{j \in \N}$ the corresponding positive eigenvalues. In fact, $W_{\varrho_D}e_j = \lambda_je_j$ and $\langle e_j, e_k \rangle_{H(K)} = \delta_{j,k}$\,. The sequence of positive eigenvalues are arranged in non-increasing order, i.e.,
$$
		\lambda_1 \geq \lambda_2 \geq \lambda_3 \geq \cdots > 0\,.
$$
Note that we have by Bessel's inequality 
\begin{equation}\label{eig_vectors} 
	\|f\|_{H(K)}^2 \geq \sum\limits_{k=1}^{\infty} |\langle f,e_k\rangle_{H(K)} |^2\,.
\end{equation}
Let us point out, that by \eqref{eig_vectors}, the function
\begin{equation} \label{K(x,x)}
	\bx \mapsto \sum\limits_{k=1}^{\infty}|e_k(\bx)|^2 \leq K(\bx,\bx)
\end{equation}
exists pointwise in $\C$. This implies 
\begin{equation}\label{C_K}
		\sum\limits_{k=1}^{\infty} \lambda_k \leq \int_D K(\bx,\bx) \, \varrho_D(\mathrm{d}\bx) 
\end{equation}
and we get by $\lambda_j = \sigma_j^2$ that $\Id:H(K) \to L_2(D,\varrho_D)$ is a Hilbert-Schmidt operator if $\eqref{integrab}$ holds. We will restrict to the situation where we have equality in \eqref{C_K}. This can be achieved by posing additional assumptions, namely that $H(K)$ is separable and $\varrho_D$ is $\sigma$-finite, see \cite[Thm.\ 4.27]{StChr08}. 
It further holds that 
$$
	\langle e_j, e_k \rangle_{L_2} = \langle \Id(e_j), \Id(e_k) \rangle_{L_2} = \langle We_j,e_k \rangle_{H(K)} = \lambda_j 
	\langle e_j,e_k\rangle_{H(K)} = \lambda_j\delta_{j,k}\,.
$$
Hence, $(e_j)_{j\in \N}$ is also orthogonal in $L_2(D,\varrho_D)$ and $\|e_j\|_2 = \sqrt{\lambda_j} =: \sigma_j$. We define the orthonormal system $(\eta_j)_{j\in \N} := (\lambda_j^{-1/2}e_j)_{j\in \N}$ in $L_2(D,\varrho_D)$\,. 

For our subsequent analysis the quantity
\begin{equation}\label{f1b}
	N(m) := \sup\limits_{\bx \in D}\sum\limits_{k=1}^{m-1}|\eta_k(\bx)|^2\
\end{equation}
plays a fundamental role. We often need the related quantity $T(m) := \sup_{\bx \in D}\sum_{k=m}^{\infty}|e_k(\bx)|^2$ which can be estimated by $T(m) \leq 2\sum_{k\geq m/2}^{\infty} \sigma_k^2 N(4k)/k$. The first one is sometimes called ``spectral function'', see \cite{Gr19} and the references therein. Clearly, by \eqref{C_K} $N(m)$ and $T(m)$ are well-defined if the kernel is bounded, i.e., if \eqref{CK000} is assumed. In fact, $T(m)$ is bounded by $\|K\|_{\infty}$. It may happen that the system $(\eta_k)_{k\in\N}$ is a uniformly $\ell_\infty(D)$ bounded orthonormal system (BOS), i.e., we have for all $k\in \N$ 
\begin{equation}
\label{f20}
\|\eta_k\|_{\ell_{\infty}(D)} \leq B\,.
\end{equation}
Let us call $B$ the BOS constant of the system. In this case we have $N(m) \leq (m-1)B^2 \in \mathcal{O}(m)$ and $T(m) \leq B^2\sum_{k=m}^{\infty}\lambda_k$\,.

\begin{remark}\label{injective} We would like to point out an issue concerning the embedding operator $\Id:H(K) \to L_2(D,\varrho_D)$ defined above. As discussed in \cite[p.\ 127]{StChr08} this embedding operator is in general {\em not} injective as it maps a function to an equivalence class. As a consequence, the system of eigenvectors $(e_j)_{j\in \N}$ may not be a basis in $H(K)$ (note that $H(K)$ may not even be separable). However, there are conditions which ensure that the orthonormal system $(e_j)_{j\in \N}$ is an orthonormal basis in $H(K)$, see \cite[4.5]{StChr08} and \cite[Ex.\ 4.6, p.\ 163]{StChr08}, which is related to Mercer's theorem \cite[Thm.\ 4.49]{StChr08}. Indeed, if we additionally assume that the kernel $K(\cdot,\cdot)$ is bounded and continuous on $D \times D$ (for a domain $D \subset \R^d$), then $H(K)$ is separable and consists of continuous functions, see \cite[Thms.\ 16, 17]{BeTh04}. If we finally assume that the measure $\varrho_D$ is a Borel measure with full support then $(e_j)_{j\in \N}$ is a complete orthonormal system in $H(K)$. In this case we have the pointwise identity  
\begin{equation}\label{ptwise} %
	K(\bx,\by)=\sum\limits_{j=1}^{\infty} \overline{e_j(\by)}\,e_j(\bx)\,,\quad \bx,\by \in D\,,
\end{equation}
as well as equality signs in \eqref{eig_vectors}, \eqref{K(x,x)} and \eqref{C_K}, see for instance \cite[Cor.\ 4]{BeTh04}. Let us emphasize, that a Mercer kernel $K(\cdot,\cdot)$, which is a continuous kernel on $D\times D$ with a compact domain $D \subset \R^d$ satisfies all these conditions, see \cite[Thm.\ 4.49]{StChr08}. In this case, we even have \eqref{ptwise} with absolute and uniform convergence on $D \times D$. Let us point out that, to our surprise, Theorem~\ref{thm:intro3} (Theorem~\ref{sampling_numbers}) holds already true under the finite trace condition \eqref{integrab} if $H(K)$ is separable and $\varrho_D$ is $\sigma$-finite. We do not have to assume continuity of the kernel. Note that the finite trace condition is natural in this context as \cite{HiNoVy08} shows. 
\end{remark}

\section{Least squares regression}
\label{sect_lsqr}
Our algorithm essentially boils down to the solution of an over-determined system
\begin{equation*} %
		\bL\, \mathbf{c} = \mathbf{f}
\end{equation*}
where $\bL \in \C^{n \times m}$ is a matrix with $n>m$. It is well-known that the above system may not have a solution. However, we can ask for the vector $\mathbf{c}$ which minimizes the residual $\|\mathbf{f}-\bL\, \mathbf{c}\|_2$. Multiplying the system with $\bL^\ast$ gives 
$$
	\bL^{\ast}\,\bL\, \mathbf{c} = \bL^{\ast}\, \mathbf{f}
$$
which is called the system of normal equations. If $\bL$ has full rank then the unique solution of the least squares problem is given by
$$
	\mathbf{c} = (\bL^{\ast}\,\bL)^{-1}\,\bL^{\ast}\, \mathbf{f}\,.
$$
From the fact that the singular values of $\bL$ are bounded away from zero we get the following quantitative bound on the spectral norm of the Moore-Penrose inverse $(\bL^{\ast}\,\bL)^{-1}\,\bL^{\ast}$. 

\begin{proposition}\label{prop2} Let $\bL\in \C^{n\times m}$ be a matrix with $m<n$ with full rank and singular values  $\tau_1,...,\tau_m  >0$ arranged in non-increasing order. 
\begin{description}
 \item[(i)] Then also the matrix $(\bL^{\ast}\,\bL)^{-1}\,\bL^{\ast}$ has full rank 
and singular values $\tau_m^{-1},...,\tau_1^{-1}$ (arranged in non-increasing order). 

\item[(ii)] In particular, it holds that 
$$
   (\bL^{\ast}\,\bL)^{-1}\,\bL^{\ast} = \bV^{\ast}\, \tilde{\mathbf{\Sigma}}\bU
$$
whenever $\bL = \bU^{\ast} \mathbf{\Sigma}  \bV$, where $\mathbf{\Sigma} \in \re^{n\times m}$ is a rectangular matrix only with $(\tau_1,...,\tau_m)$ on the ``main diagonal'' and orthogonal matrices $\bU \in \C^{n\times n}$ and $\bV \in \C^{m\times m}$. Here $\tilde{\mathbf{\Sigma}} \in \re^{m\times n}$ denotes the matrix with $(\tau_1^{-1},...,\tau_m^{-1})$ on the ``main diagonal''\,.

\item[(iii)] The operator norm $\|(\bL^{\ast}\,\bL)^{-1}\,\bL^{\ast}\|$ can be controlled as follows 
$$
	\|(\bL^{\ast}\,\bL)^{-1}\,\bL^{\ast}\| \leq \tau_m^{-1}\,.
$$
\end{description}
\end{proposition}

For function recovery we will use the following matrix 
\begin{equation}\label{f0}
	\bL_m := \bL_m(\bX) = \left(\begin{array}{cccc}
			\eta_1(\bx^1) &\eta_2(\bx^1)& \cdots & \eta_{m-1}(\bx^1)\\
			\vdots & \vdots && \vdots\\
			\eta_1(\bx^n) &\eta_2(\bx^n)& \cdots & \eta_{m-1}(\bx^n)
	\end{array}\right)\,,
\end{equation}
for $\bX = \{\bx^1,...,\bx^n\}\subset D$ of distinct sampling nodes and the system $(\eta_k)_{k\in\N} := (\lambda_k^{-1/2}e_k)_{k\in\N}$. Below we will see that this matrix behaves well with high probability
if $n$ is large enough and the nodes in $\bX$ are chosen independently and identically $\varrho_D$-distributed from $D$.

\begin{algorithm}[H]
\caption{Least squares regression.}\label{algo1}
  \begin{tabular}{p{1.2cm}p{5.0cm}p{8.1cm}}
    Input: & $\bX = \{\bx^1,...,\bx^n\}\subset D$ \hfill & set of distinct sampling nodes, \\
      & $\mathbf{f} = (f(\bx^1),...,f(\bx^n))^\top$ \hfill & samples of $f$ evaluated at the nodes from $\bX$, \\
      & $m\in\N$ & $m < n$ such that the matrix $\bL_m := \bL_m(\bX)$ from~\eqref{f0} has full (column) rank.
  \end{tabular}
  \begin{algorithmic}
  \STATE
  Solve the over-determined linear system 
  $$
	\bL_m \, (c_1,...,c_{m-1})^\top = \mathbf{f}\,
  $$
  via least squares (e.g.\ directly or via the LSQR algorithm \cite{PaSa82}), i.e., compute
  $$
	 (c_1,...,c_{m-1})^\top := (\bL_m^{\ast}\,\bL_m)^{-1}\,\bL_m^{\ast}\, \mathbf{f}\,.
  $$
  \end{algorithmic}
   Output:  $\mathbf{c} = (c_1,...,c_{m-1})^\top\in \C^{m-1}$ coefficients of the approximant $S^m_{\bX}f:=\sum_{k = 1}^{m-1} c_k \, \eta_k$.
\end{algorithm}

Using Algorithm~\ref{algo1}, we compute the coefficients $c_k$, $k=1,\ldots,m-1$, of the approximant
\begin{equation}
S^m_{\bX}f:=\sum_{k = 1}^{m-1} c_k\, \eta_k\,.\label{eq:def_SmX}
\end{equation}
Note that the mapping $f \mapsto S^m_{\bX}f$ is linear for a fixed set of sampling nodes $\bX\subset D.$

\section{Concentration inequalities}
\label{sect_prob}
We will consider complex-valued random variables $X$ and random vectors $(X_1,...,X_N)$ on a probability space $(\Omega,\mathcal{A},\Prob)$. As usual we will denote with $\Ept X$ the expectation of $X$. With $\Prob(A|B)$ and $\Ept(X|B)$ we denote the conditional probability 
$$
		\Prob(A|B) := \frac{\Prob(A\cap B)}{\Prob(B)}
$$
and the conditional expectation
\begin{equation}\label{eq:expectation_cond}
	\Ept(X|B) = \frac{\Ept(\chi_B\cdot X)}{\Prob(B)}\,,
\end{equation}
where $\chi_B:\Omega\to\{0,1\}$ is the indicator function on $B$.

Let us start with the classical Markov inequality. If $Z$ is a random variable then 
\begin{equation*} %
	\Prob(|Z|>t) \leq \frac{\Ept|Z|}{t}\,,\quad t>0\,.
\end{equation*}
Of course, there is also a version involving conditional probability and expectation. In fact, 
\begin{equation}
	\Prob(|Z|>t~|~B) \leq \frac{\Ept(|Z|~|~B)}{t}\,,\quad t>0\,.
	\label{eq:markov_cond}
\end{equation}
Let us state concentration inequalities for the norm of sums of complex rank-$1$ matrices. For the first result we refer to Oliveira \cite{Ol10}. We will need the following notational convention: For a complex 
(column) vector $\by\in \C^N$ (or $\ell_2$) we will often use the tensor notation for the matrix
$$
	\by \otimes \by := \by\, \by^\ast = \by \, \overline{\by}^\top\in\C^{N \times N}\;\textnormal{(or $\C^{\N\times\N}$)}\,.
$$

\begin{proposition}\label{Oliv} Let $\by^i, i=1,...,n$, be i.i.d.\ copies of a random vector $\by \in \C^{N}$ such that $\|\by^i\|_2 \leq M$ almost surely. Let further $\Ept(\by^i \otimes \by^i) = \mathbf{\Lambda}\in\C^{N\times N}$ and $0<t <1$. Then it holds
$$
	\Prob\left(\Bigg\|\frac{1}{n}\sum\limits_{i=1}^n\by^i \otimes \by^i-\mathbf{\Lambda}\Bigg\|>t\right) \leq (2\min(n,N))^{\sqrt{2}}\exp\left(-\frac{nt^2}{12M^2}\right).
$$
\end{proposition}

\begin{proof}
In order to show the probability estimate, we refer to the proof of \cite[Lem.\ 1]{Ol10} and observe
$$
\Prob\left(\Bigg\|\frac{1}{n}\sum\limits_{i=1}^n\by^i \otimes \by^i-\mathbf{\Lambda}\Bigg\|>t\right)
\le(2\min(n,N))^{\frac{1}{1-2M^2s/n}}\exp\left(-st+\frac{2M^2s^2}{n-2M^2s}\right)
$$
for $ 0\le s\le n/(2M^2)$. Since we restrict $0<t<1$, the choice $s=(4+2\sqrt{2})^{-1}nt/M^2$
yields
\begin{align*}
(2\min(n,N))^{\frac{1}{1-2M^2s/n}}&\exp\left(-st+\frac{2M^2s^2}{n-2M^2s}\right)
=
(2\min(n,N))^{\sqrt{2}}\exp\left(-\frac{nt^2}{(6+4\sqrt{2})M^2}\right)
\end{align*}
and, finally, the assertion holds.
\end{proof}

\begin{remark} %
A slightly stronger version for the case of real matrices can be found in Cohen, Davenport, Leviatan \cite{CoDaLe13} (see also the correction in \cite{CoDaLe19}).
For $\by^i, i=1,...,n$, i.i.d.\ copies of a random vector $\by \in \R^{N}$ sampled from a bounded orthonormal system,
one obtains the concentration inequality
$$
	\Prob\left(\Bigg\|\frac{1}{n}\sum\limits_{i=1}^n\by^i \otimes \by^i-\mathbf{I}\Bigg\|>t\right) \leq 2N\exp(-c_tn/M^2)\,,
$$
where $c_t = (1+t)(\ln(1+t))-t$. This leads to improved constants for the case of real matrices. 

\end{remark}

The following result goes back to Lust-Piquard, Pisier \cite{LPPi91}, and Rudelson \cite{Ru99}. The complex version with precise constants can be found in Rauhut \cite[Cor.\ 6.20]{Ra10}.

\begin{proposition}\label{RudLem} Let $\by^i \in \C^N$ (or $\ell_2$), $i=1,...,n$, and $\eps_i$ independent Rademacher variables taking values $\pm 1$ with equal probability. Then
\begin{equation}\label{lem:Rud}
	\Ept_{\varepsilon} \Bigg\|\sum\limits_{i=1}^n \eps_i \, \by^i\otimes \by^i\Bigg\| \leq C_{\mathrm{R}}\sqrt{\log(8\min\{n, N\})}\cdot
	\sqrt{\Bigg\|\sum\limits_{i=1}^n \by^i\otimes \by^i\Bigg\|}\cdot \max\limits_{i=1,...,n}\|\by^i\|_2\,,
\end{equation}
with 
\begin{equation}\label{f31}
		C_{\mathrm{R}} = \sqrt{2}+\frac{1}{4\sqrt{2\log(8)}} \in [1.53,1.54]\,.
\end{equation}
\label{prop_f31}
\end{proposition}

\begin{remark}\label{reminfty} The result is proved for complex (finite) matrices. Note that the factor $\sqrt{\log(8\min\{n,N\})}$ is already an upper bound for $\sqrt{\log(8r)}$, where $r$ is the rank of the matrix $\sum_i \by^i\otimes \by^i$. The proof of Proposition \ref{RudLem} with the precise constant is based on \cite[Lem.\ 6.18]{Ra10} which itself is based on a non-commutative Khintchine inequality, see \cite[6.5]{Ra10}. This technique allows for controlling all the involved constants. Let us comment on the situation $N=\infty$, i.e., $\by^j \in \ell_2$, where this inequality keeps valid with the factor $\sqrt{\log(8n)}$. In fact, if the matrices $\mathbf{B}_j$ in \cite[Thm.\ 6.14]{Ra10} are replaced by rank-$1$-operators $\mathbf{B}_j\colon\ell_2\to\ell_2$ of type  $\mathbf{B}_j = \by^j \otimes \by^j$ with $\|\by^j\|_2 < \infty$ then all the arguments keep valid and an $\ell_2$-version of this non-commutative Khintchine inequality is available. This implies an $\ell_2$-version of \cite[Lem.\ 6.18]{Ra10} which reads as follows: Let $\by^j \in \ell_2$, $j=1,...,n$, and $p\geq 2$. Then
$$
	\Bigg(\Ept_{\varepsilon} \Bigg\|\sum\limits_{i=1}^n \eps_i \, \by^i\otimes \by^i\Bigg\|^p\Bigg)^{1/p} \leq 2^{3/(4p)}n^{1/p}\sqrt{p} \, \mathrm{e}^{-1/2}\sqrt{\Bigg\|\sum\limits_{i=1}^n \by^i\otimes \by^i\Bigg\|}\cdot \max\limits_{i=1,...,n}\|\by^i\|_2\,.
$$ 
Since we control the moments of the random variable representing the norm on the left-hand side we are now able to derive a concentration inequality by standard arguments (\cite[Prop.~6.5]{Ra10}). This  concentration inequality then easily implies the $\ell_2$-version of \eqref{lem:Rud}.
\end{remark}

As a consequence of this result we obtain the following deviation inequality in the mean which will be sufficient for our purpose. 

\begin{corollary}\label{cor5} Let $\by^i$, $i=1,...,n$, be i.i.d.\ random vectors from $\C^N$ or $\ell_2$ with $\|\by^i\|_2 \leq M$ almost surely. Let further $\mathbf{\Lambda} = \Ept(\by^i \otimes \by^i)$. Then with $N>n$ we obtain
$$
	\Ept \Bigg\|\frac{1}{n}\sum\limits_{i=1}^n \by^i \otimes \by^i-\mathbf{\Lambda}\Bigg\| \leq 
	4 C_{\mathrm{R}}^2\frac{\log (8n)}{n}M^2+2C_{\mathrm{R}}\sqrt{\frac{\log (8n)}{n}}M\sqrt{\|\mathbf{\Lambda}\|}\,.
$$

\end{corollary}
\begin{proof} By well-known symmetrization technique (see \cite[Lem.\ 8.4]{FoRa13}), Proposition~\ref{prop_f31}, and the Cauchy-Schwarz inequality, we obtain

\begin{equation}\nonumber
  \begin{split}
	F := \Ept \Bigg\|\frac{1}{n}\sum\limits_{i=1}^n \by^i \otimes \by^i-\mathbf{\Lambda}\Bigg\| &\leq 2\Ept_{\by} \Ept_{\eps} \Bigg\|\frac{1}{n}\sum\limits_{i=1}^n \eps_i\by^i \otimes \by^i\Bigg\|\\
	&\leq 2C_{\mathrm{R}}\frac{\sqrt{\log (8n)}}{n}\Big(\Ept \max\limits_{i=1,...,n}\|\by^i\|^2_2\Big)^{1/2}
\Bigg(\Ept\Bigg\|\sum\limits_{i=1}^n \by^i \otimes \by^i\Bigg\|\Bigg)^{1/2}\\
&\leq \underbrace{2C_{\mathrm{R}}\frac{\sqrt{\log (8n)}}{\sqrt{n}}M}_{=:a}\Bigg(\underbrace{\Ept\Bigg\|\frac{1}{n}\sum\limits_{i=1}^n \by^i \otimes \by^i-\mathbf{\Lambda}\Bigg\|}_{=F}
+\underbrace{\|\mathbf{\Lambda}\|}_{=:b}\Bigg)^{1/2}\,.
  \end{split}	
\end{equation}
Hence, we get
$F^2\leq a^2(F+b)$
and we solve this inequality with respect to $F$, which yields
$$0\leq F \le \frac{a^2}{2}+\sqrt{\frac{a^4}{4}+a^2b}\le a^2+a\sqrt{b}$$
and this corresponds to the assertion.
\end{proof}

\section{Worst-case errors for least-squares regression}
\label{worst_case_analysis}

\subsection{Random matrices from sampled orthonormal systems}

Let us start with a concentration inequality for the spectral norm of a matrix of type \eqref{f0}. 
It turns out that the complex matrix $\bL_m:=\bL_m(\bX)\in\C^{n\times(m-1)}$ has full rank with high probability, where the elements of the set $\bX\subset D$ of sampling nodes are drawn i.i.d.\ at random according to $\varrho_D$.
We will find below that the eigenvalues of 
\begin{equation*}
\mathbf{H}_m:=\mathbf{H}_m(\bX) = \frac{1}{n}\bL_m^\ast \, \bL_m \in\C^{(m-1)\times(m-1)}
\end{equation*}
are bounded away from zero with high probability if $m$ is small enough compared to $n$. We speak of an ``oversampling factor'' $n/m$. In case of a bounded orthonormal system with BOS constant $B$, see \eqref{f20}, it will turn out that a logarithmic oversampling is sufficient, see \eqref{f1c} below. Note that the boundedness constant $B$ may also depend on the underlying spatial dimension $d$. However, if for instance the complex Fourier system $\{\exp(2\pi\mathrm{i}\,\bk\cdot \bx) \colon \bk \in \Z^d\}$ is considered, we are in the comfortable situation that $B = 1$.

\begin{proposition}\label{prop1}
 Let $n,m \in \N$, $m\ge2$. Let further $\{\eta_1(\cdot),  \eta_2(\cdot), \eta_3(\cdot),...\}$ be the orthonormal system in $L_2(D,\varrho_D)$ induced by the kernel $K$ and the $n$ sampling nodes in $\bX$ be drawn i.i.d.\ at random according to $\varrho_D$. Then it holds for $0<t <1$ that 
$$
 	\Prob(\|\mathbf{H}_m - \mathbf{I}_m\| > t) \leq (2n)^{\sqrt{2}} \, \exp\Bigg(-\frac{nt^2}{12\cdot N(m)}\Bigg)\,,
$$
where $N(m)$ is defined in \eqref{f1b} and $\mathbf{I}_m=\diag(\boldsymbol{1})\in\{0,1\}^{(m-1)\times(m-1)}$.
\end{proposition}
\begin{proof}
We set $\by^i := \left(\eta_1(\bx^i), \ldots, \eta_{m-1}(\bx^i)\right)^\ast$, $i=1,...,n,$
and observe
\begin{equation*}
\mathbf{H}_m:=\mathbf{H}_m(\bX) = \frac{1}{n}\bL_m^\ast \, \bL_m = \frac{1}{n} \sum\limits_{i = 1}^n \by^i \otimes \by^i.
\end{equation*}
 Moreover, due to the fact that we have an orthonormal system $(\eta_k)_{k\in \N}$, we obtain that $\Ept(\mathbf{H}_m) = \mathbf{I}_m$.
The result follows by noting that $M^2 \leq N(m)$ in Proposition \ref{Oliv}.
\end{proof}

\begin{remark}
\sloppy
From this proposition we immediately obtain that the matrix $\bH_m \in \C^{(m-1) \times (m-1)}$ has only eigenvalues larger than $t:=1/2$ with probability at least $1-\delta$ if 
\begin{equation}\label{f1}
N(m) \leq \frac{n}{48(\sqrt{2}\log(2n)-\log\delta)}\,.
\end{equation}
Hence, in case of a bounded orthonormal system with BOS constant $B>0$, see \eqref{f20}, we may choose
\begin{equation}\label{f1c}
	m \leq \kappa_{\delta,B}\frac{n}{\log(2n)}
\end{equation}
with $\kappa_{\delta,B} := (\log(1/\delta)+\sqrt{2})^{-1}B^{-2}/48$. %
\end{remark}

From Proposition \ref{prop1} we get that all $m-1$ singular values $\tau_1,..., \tau_{m-1}$ of $\bL_m$ from \eqref{f0} are all not smaller than $\sqrt{n/2}$ and not larger than $\sqrt{3n/2}$ with probability at least $1-\delta$ if $m$ is chosen such that~\eqref{f1} holds. In terms of Proposition \ref{prop2} this means that $\tau_1,..., \tau_{m-1} \geq \sqrt{n/2}$. This leads to an upper bound on the norm of the Moore-Penrose inverse that is required for the least squares algorithm. 

\begin{proposition}\label{prop:norm_Lm_inv}
 Let $\{\eta_1(\cdot),  \eta_2(\cdot), \eta_3(\cdot),...\}$ be the orthonormal system in $L_2(D,\varrho_D)$ 
induced by the kernel $K$. Let further $m,n \in \N$, $m\ge 2$, and $0 < \delta<1$ be chosen such that they satisfy \eqref{f1}. Then the random matrix $\bL_m$ from \eqref{f0} satisfies
$$
	\|(\bL_m^{\ast}\,\bL_m)^{-1}\,\bL_m^{\ast}\| \leq \sqrt{\frac{2}{n}}
$$
with probability at least $1-\delta$\,.
\end{proposition}

In addition to the matrix $\bL_m$, we need to consider a second linear operator that is defined using sampling values of the eigenfunctions $e_j$. The importance of this operator has been pointed out in \cite{KrUl19}, where strong results on the concentration of infinite dimensional random matrices have been used. Since we only need the expectation of the norm, we only use Rudelson's lemma, see Proposition \ref{RudLem}, and a symmetrization technique. This allows us to control the constants.  

\begin{proposition}\label{cor8} Let $\bX=\{\bx^1,\ldots,\bx^n\}\subset D$ be a set of $n$ 
sampling nodes drawn uniformly and i.i.d.\ at random according to $\varrho_D$, and consider the $n$ i.i.d.\ random sequences   
$$
	\by^i = (e_m(\bx^i),e_{m+1}(\bx^i),...)^\top\,,\quad i=1,...,n\,,
$$
together with $T(m):= \sup\limits_{\bx \in D}\sum\limits_{k= m}^{\infty} |e_k(\bx)|^2 < \infty$, see \eqref{f1b}. Then the operator 
\begin{equation*}
		 \mathbf{\Phi}_m\colon\ell_2 \to \R^n,\quad
		 \mathbf{z} \mapsto \left(\begin{array}{c}
		 \<\mathbf{z},\by^1\>_{\ell_2}\\
		 \vdots\\
		 \<\mathbf{z},\by^n\>_{\ell_2}
		 \end{array}\right)	    
\end{equation*}
has expected norm
\begin{equation}\label{expLambda}
	\Ept(\|\mathbf{\Phi}_m\|^2)\leq 
	n\left(\sigma_m^2+4\,C_{\mathrm{R}}^2\frac{\log (8n)}{n}T(m)+2\,C_{\mathrm{R}}\,\sigma_m\,\sqrt{\frac{\log (8n)}{n}T(m)}\right)\,.
\end{equation}
\end{proposition}
\begin{proof}
Note that $\mathbf{\Phi}_m^{\ast}\mathbf{\Phi}_m = \sum\limits_{i=1}^n \by^i \otimes \by^i$ and
$$
	\mathbf{\Lambda}_m := \Ept\Big(\frac{1}{n}\mathbf{\Phi}_m^{\ast}\mathbf{\Phi}_m\Big) = \Ept(\by^i \otimes \by^i) = 
	\diag(\sigma^2_m,\sigma^2_{m+1},...)\,.
$$
This gives
$$
	\|\mathbf{\Phi}_m\|^2 = \|\mathbf{\Phi}_m^{\ast}\mathbf{\Phi}_m\| 
	\le \|\mathbf{\Phi}_m^{\ast}\mathbf{\Phi}_m-n\mathbf{\Lambda}_m\|+n\|\mathbf{\Lambda}_m\|
	\,.
$$
Finally, the bound in \eqref{expLambda} follows from Corollary \ref{cor5} (see also Remark \ref{reminfty} for $N = \infty$), the fact that $\|\mathbf{\Lambda}_m\| = \lambda_m = \sigma_m^2$ and $M^2 = T(m)$\,.
\end{proof}

\subsection{Worst-case errors with high probability}

\begin{theorem}\label{thm9} Let $H(K)$ be a separable reproducing kernel Hilbert space on a domain $D \subset \R^d$ with a positive semidefinite kernel $K(\bx,\by)$ such that $\sup_{\bx \in D} K(\bx,\bx) <\infty$. We denote with $(\sigma_j)_{j\in\N}$ the non-increasing sequence of singular numbers of the embedding $\Id:H(K) \to L_2(D,\varrho_D)$ for a probability measure $\varrho_D$\,.
Let further $0<\delta <1$ and $m,n \in \N$,  where $m\ge 2$ is chosen such that \eqref{f1} holds. Drawing the $n$ sampling nodes in $\bX$ i.i.d.\ at random according to $\varrho_D$, we have for the conditional expectation of the worst-case error
\begin{equation}
\begin{split}
&\Ept\Big(\sup\limits_{\|f\|_{H(K)}\leq 1} \|f-S^m_\bX f\|^2_{L_2(D,\varrho_D)}\Big|\|\mathbf{H}_m-\mathbf{I}_m\| \leq 1/2\Big)
\\
&\hspace{2.5cm}\leq \frac{1}{1-\delta}\left(3\sigma_m^2+8\,C_{\mathrm{R}}^2\frac{\log (8n)}{n}T(m)+4\,C_{\mathrm{R}}\,\sigma_m\,\sqrt{\frac{\log (8n)}{n}T(m)}\right)\label{eq:Ept_sum_detail}
\\
&\hspace{2.5cm}\leq \frac{3+8\,C_{\mathrm{R}}^2+4\,C_{\mathrm{R}}}{1-\delta}\,\max\Big\{\sigma^2_m ,\frac{\log (8n)}{n}T(m)\Big\}\,
\end{split}
\end{equation}
with $C_{\mathrm{R}}$ from \eqref{f31}.
\end{theorem}

\begin{proof}\sloppy Let $f \in H(K)$ such that $\|f\|_{H(K)} \leq 1$\,. Let further $\bX=\{\bx^1,\ldots,\bx^n\}$ be such that $\|\mathbf{H}_m-\mathbf{I}_m\| \leq 1/2$. Using orthogonality and the reproducing property $S^m_\bX P_{m-1}f=P_{m-1}f$, we estimate 
\begin{equation}\label{estim0}
  \begin{split}
	\|f-S^m_\bX f\|^2_{L_2(D,\varrho_D)} &= \|f-P_{m-1}f\|^2_{L_2(D,\varrho_D)}+\|P_{m-1}f-S^m_\bX f\|^2_{L_2(D,\varrho_D)}\\
	&\leq \sigma^2_m + \|S^m_\bX (P_{m-1}f-f)\|^2_{L_2(D,\varrho_D)}\\
	&=\sigma^2_m + \left\|(\bL_m^{\ast}\,\bL_m)^{-1}\,\bL_m^{\ast} \left(
	(P_{m-1}f-f)(\bx^k)
		\right)_{k=1}^n\right\|^2_{2}\\
		&\leq \sigma_m^2+\frac{2}{n}\sum\limits_{k=1}^n \Big|\Big(f-P_{m-1}f\Big)(\bx^k)\Big|^2\,,
   \end{split}
\end{equation}	
where $P_{m-1}f$ denotes the projection $\sum_{j=1}^{m-1} \langle f,e_j \rangle e_j$ in $H(K)$ yielding  
$\|f - P_{m-1}f\|_{L_2(D,\varrho_D)} \leq \sigma_m$. Note further, that for any $\bx \in D$
\begin{equation*}
  \begin{split}
	\Big(f-P_{m-1}f\Big)(\bx) &= \langle f, K(\cdot,\bx)-\sum\limits_{j=1}^{\infty}e_j(\cdot)\overline{e_j(\bx)}
	+\sum\limits_{j=1}^{\infty}e_j(\cdot)\overline{e_j(\bx)}-\sum\limits_{j=1}^{m-1}e_j(\cdot)\overline{e_j(\bx)} \rangle_{H(K)}\\
	&= \sum\limits_{j=m}^\infty \langle f,e_j\rangle_{H(K)} e_j(\bx) + \langle f, T(\cdot,\bx)\rangle_{H(K)}\,,
  \end{split}
\end{equation*}
where $T(\cdot,\bx) = K(\cdot,\bx)-\sum_{j=1}^{\infty}e_j(\cdot)\overline{e_j(\bx)}$ denotes an element in $H(K)$. Its norm is given by
\begin{equation}\label{norm_T}
	\|T(\cdot,\bx)\|_{H(K)}^2 := \langle T(\cdot,\bx), T(\cdot,\bx)\rangle_{H(K)} = K(\bx,\bx)-\sum_{j=1}^{\infty}|e_j(\bx)|^2\,.
\end{equation}
This gives 
\begin{equation*}
  \begin{split}
	\Big|\Big(f-P_{m-1}f\Big)(\bx)\Big|^2 \leq& \Bigg|\sum\limits_{j=m}^\infty \langle f,e_j\rangle_{H(K)} e_j(\bx)\Bigg|^2
	+ 2\|f\|_{H(K)}^2\|T(\cdot,\bx)\|_{H(K)}\sqrt{\sum\limits_{j=m}^{\infty}|e_j(\bx)|^2}\\
	&~~~+ \|f\|_{H(K)}^2 \|T(\cdot,\bx)\|_{H(K)}^2\\
	\overset{\eqref{K(x,x)}}{\leq}& \Bigg(\sum\limits_{i=m}^\infty \overline{\langle f,e_i\rangle_{H(K)} e_i(\bx)}\Bigg) \Bigg(\sum\limits_{j=m}^\infty \langle f,e_j\rangle_{H(K)} e_j(\bx)\Bigg)\\
	&~~~+\|f\|_{H(K)}^2\|T(\cdot,\bx)\|_{H(K)}\Big(\|T(\cdot,\bx)\|_{H(K)}+2\|K\|_{\infty}\Big)\\
	\leq& \sum\limits_{i=m}^\infty \sum\limits_{j=m}^\infty \overline{\langle f,e_i\rangle_{H(K)}} \langle f,e_j\rangle_{H(K)} \overline{e_i(\bx)} e_j(\bx)\\
	&~~~+3\|T(\cdot,\bx)\|_{H(K)}\|K\|_{\infty}\|f\|_{H(K)}^2\,.
\end{split}
\end{equation*}
Returning to \eqref{estim0} we estimate 
\begin{equation}\label{estim}
  \begin{split}
\sum\limits_{k=1}^n\Big|\Big(f-P_{m-1}f\Big)(\bx^k)\Big|^2 &\leq \|(\langle f,e_j\rangle_{H(K)})_{j\in \N}\|^2_2 \; \|\mathbf{\Phi}_m\|^2
	+3\|K\|_{\infty}\|f\|_{H(K)}^2\sum\limits_{k=1}^{n}\|T(\cdot,\bx^k)\|\\
	&\leq \|f\|_{H(K)}^2\Bigg(\|\mathbf{\Phi}_m\|^2 +3\|K\|_{\infty}\sum\limits_{k=1}^{n}\|T(\cdot,\bx^k)\|\Bigg)\,,
\end{split}
\end{equation}
where $\mathbf{\Phi}_m$ denotes the infinite matrix from Proposition~\ref{cor8}. Note that we used \eqref{eig_vectors} in the last but one step. 
The relation in \eqref{estim} together with \eqref{estim0} and $\|f\|_{H(K)} \leq 1$ implies
\begin{align*}
	\|f-S^m_\bX f\|^2_{L_2(D,\varrho_D)}
	&\leq \sigma^2_m + \|(\bL_m^{\ast}\,\bL_m)^{-1}\,\bL_m^{\ast}\|^2\cdot \sum\limits_{k=1}^n \Big|\Big(f-P_{m-1}f\Big)(\bx^k)\Big|^2\\	&= \sigma^2_m + \frac{2}{n}\|\mathbf{\Phi}_m\|^2 + \frac{6\|K\|_{\infty}}{n}\sum\limits_{k=1}^n \|T(\cdot,\bx^k)\|\,.
\end{align*}
Integrating on both sides yields 
\begin{equation}\label{eq55}
  \begin{split}
	&\int\limits_{\|\mathbf{H}_m-\mathbf{I}_m\| \leq 1/2}\sup\limits_{\|f\|_{{H(K)}}\leq 1}\|f-S^m_\bX f\|^2_{L_2(D,\varrho_D)} \, \varrho^n_D(\mathrm{d}\bX) \\
	&~~~~\leq \sigma^2_m + \frac{2}{n}\Ept(\|\mathbf{\Phi}_m\|^2)+
\frac{6\|K\|_{\infty}}{n}\sum\limits_{k=1}^n \int_D\|T(\cdot,\bx^k)\| \, \varrho_D(\mathrm{d}\bx)	\\
&~~~~= \sigma^2_m + \frac{2}{n}\Ept(\|\mathbf{\Phi}_m\|^2)\,.
\end{split}
\end{equation}
Note that the integral on the right-hand side of \eqref{eq55} vanishes because of \eqref{norm_T} and the fact that we have an equality sign in \eqref{C_K} due to our assumptions (separability of $H(K)$). This gives
$$
	0 = \int_D K(\bx,\bx) \, \varrho_D(\mathrm{d}\bx)-\int_D\sum\limits_{j=1}^{\infty}\lambda_j \, \varrho_D(\mathrm{d}\bx) = \int_D  \Bigg(K(\bx,\bx)-\sum\limits_{j=1}^{\infty}|e_j(\bx)|^2\Bigg) \, \varrho_D(\mathrm{d}\bx)\,.
$$
Taking Proposition \ref{cor8} and \eqref{eq:expectation_cond} into account and noting that $\Prob(\|\mathbf{H}_m-\mathbf{I}_m\| \leq 1/2)$ is larger than $1-\delta$, we obtain the assertion.
\end{proof}

In addition to that we may easily get a deviation inequality by using Markov's inequality and standard arguments. It reads as follows. 
\begin{corollary}\label{cor10} Under the same assumptions as in Theorem \ref{thm9} it holds for fixed $\delta >0$
\begin{equation}  \Prob\Bigg(\sup\limits_{\|f\|_{{H(K)}}\leq 1} \|f-S^m_\bX f\|^2_{L_2(D,\varrho_D)}
  \leq \frac{C}{\delta}\,\max\Big\{\sigma^2_m,\frac{\log(8n)}{n}T(m)	
 \Big\}\Bigg) \geq 1-3\delta\,,
\label{eq:cor_main_prob}
\end{equation}
where $C:=3+8\,C_{\mathrm{R}}^2+4\,C_{\mathrm{R}}<28.05$ is an absolute constant.
\end{corollary}
\begin{proof}
We define the events
\begin{align*}
A&:=\Bigg\{\bX\;\colon\sup\limits_{\|f\|_{H(K)}\leq 1} \|f-S^m_\bX f\|^2_{L_2(D,\varrho_D)}\leq t\Bigg\}\,,\\
B&:=\{\bX\;\colon \|\mathbf{H}_m-\mathbf{I}_m\| \leq 1/2\}
\end{align*}
and split up
$$
\Prob(A)
=1- \Prob(A^\complement)=1-\Prob(A^\complement \cap B)-\Prob(A^\complement\cap B^\complement)\,.
$$
Treating each summand separately, we have
\begin{align*}
  \Prob(A^\complement\cap B)&=   \Prob(A^\complement| B)\Prob(B)\le \Prob(A^\complement| B)\,,\\
\Prob(A^\complement\cap B^\complement)&\le \Prob(B^\complement)\le \delta\,.
\end{align*}
Next we estimate $\Prob(A^\complement| B)$ using the Markov inequality \eqref{eq:markov_cond}, 
Theorem~\ref{thm9},
and setting 
$$
t:=\frac{3+8\,C_{\mathrm{R}}^2+4\,C_{\mathrm{R}}}{\delta}\,\max\left\{\sigma_m^2,\frac{\log (8n)}{n}T(m)\right\}
$$
which yields
$$
\Prob(A^\complement| B)\le \frac{\Ept(A^\complement| B)}{t}\le \frac{\delta}{1-\delta}\overset{\delta\leq 1/2}{\le} 2\delta
$$
and the assertion follows.
\end{proof}

\begin{example}\label{ex:legendre}
Theorem~\ref{thm9} as well as Corollary~\ref{cor10} can also be applied to non-bounded orthonormal systems which may lead to non-optimal error bounds.
For instance, let $D=[-1,1]$ and $\varrho_D$ the normalized Lebesgue measure on $D = [-1,1]$. Then
the second order operator $\operatorname{A}$ defined by
$$\operatorname{A}\!f(x)=-((1-x^2) v')'$$
characterizes for $s>1$ weighted Sobolev spaces
$$H(K_s):=\{f\in L_2(D)\colon \operatorname{A}^{s/2}\!f\in L_2(D)\}$$
which are in fact reproducing kernel Hilbert spaces with reproducing kernel
$$
K_s(x,y)=\sum_{k\in\N}(1+(k(k+1))^s)^{-1}\cP_k(x)\cP_k(y),
$$
where $\cP_k\colon D\to\R$, $k\in\N$, are $L_2(D)$-normalized Legendre polynomials.
Clearly, $(\cP_k)_{k\in\N}$ provides an ONB in $L_2(D)$ and plays the role of $(\eta_k)_{k\in\N}$ in our setting. Moreover, we have $e_k=\lambda_k\eta_k=(1+(k(k+1))^s)^{-1/2}\cP_k$, and, accordingly, $\sigma_k=(1+(k(k+1))^s)^{-1/2}$. The two quantities $N(m)$ and $T(m)$ are given by
\begin{align*}
N(m)&=\sup_{x\in D}\sum_{k=1}^{m-1} |\cP_k(x)|^2=\sum_{k=0}^{m-2}\frac{2k+1}{2}=\frac{(m-1)^2}{2},\\
T(m)&=\sup_{x\in D}\sum_{k=m}^\infty\frac{2k-1}{2(1+(k(k+1))^s}\asymp \sum_{k=m}^\infty k^{-2s+1} \asymp m^{-2s+2}\,.
\end{align*}
Applying Theorem~\ref{thm9} or Corollary~\ref{cor10} and choosing $m$ as large as possible, leads to the relation $m^2\sim n/\log{n}$, which is far from optimal with respect to the number $n$ of used samples, i.e., we observe worst case error estimates of the form
$$
\sup\limits_{\|f\|_{{H(K_s)}}\leq 1} \|f-S^m_\bX f\|_{L_2(D)}
  \lesssim \sigma_m \sim m^{-s}\lesssim \left(\frac{\log(n)}{n}\right)^{s/2}
$$
in expectation and with high probability, respectively.

On the other hand, the result by Bernardi and Maday \cite[Thm. 6.2]{BeMa92} using a polynomial interpolation operator $j_{n-1}$ at Gauss points guarantees 
$$
\sup\limits_{\|f\|_{{H(K_s)}}\leq 1} \|f-j_{n-1} f\|_{L_2(D)}
  \lesssim \frac{\log(n)}{n^s}\,,
$$
which is optimal in its main rate $s$. However, as it turns out below (see Example \ref{leg_2}), it is not optimal with respect to the power of the logarithm. 
\qed
\end{example}
Example \ref{ex:legendre} illustrates that the suggested approach will lead to worst case error estimates with high probability. However, the achieved upper bounds are not optimal in specific situations. We will overcome this limitation in the next section by drawing the sampling nodes with respect to a weighted, tailored distribution and apply a weighted least squares algorithm.

\subsection{Improvements due to importance sampling}
\label{Sect5_opt}
We are interested in the question of optimal sampling recovery of functions from reproducing kernel Hilbert spaces in $L_2(D,\varrho_D)$. 
The goal is to get reasonable bounds in $n$, preferably in terms of the singular numbers of the embedding. Theorem \ref{thm9} already gives a satisfying answer in case of bounded kernels and $N(m) \in \mathcal{O}(m)$. In order to drop both conditions, we will use a weighted (deterministic) least squares algorithm (see Algorithm 2) to recover functions $f\in H(K)$ from samples at random nodes (``random information'' in the sense of \cite{HiKrNoPrUl19}). The approach is a slight modification of the one proposed earlier in \cite{KrUl19}. A technique will be used, which is known as ``(optimal) importance sampling'', where one defines a density function depending on the spectral properties of the embedding operator. The sampling nodes are then drawn according to this density. In the Monte-Carlo setting (or ``randomized setting'') this has been successfully applied, e.g., by Cohen and Migliorati in \cite{CoMi17}, see Remark \ref{imp_samp} below. Also in connection with compressed sensing it led to substantial improvements when recovering multivariate functions, see \cite{RaWa16, RaWa12}. Authors originally applied this technique e.g.\ for the approximation of integrals, see~\cite{Hi10}. However, the setting in which we are interested in requires additional work since the sampling nodes are supposed to be drawn in advance for the whole class of functions. 

\begin{algorithm}[tb]
\caption{Weighted least squares regression.}\label{algo1:reweighted}
  \begin{tabular}{p{1.2cm}p{4.5cm}p{8.9cm}}
    Input: & $\bX = \{\bx^1,...,\bx^n\}\subset D$ \hfill & set of distinct sampling nodes, \\
      & $\mathbf{f} = (f(\bx^1),...,f(\bx^n))^\top$ \hfill & samples of $f$ evaluated at the nodes from $\bX$, \\
      & $m\in\N$ & $m < n$ such that the matrix $\tilde{\bL}_m$ in \eqref{eq:tilde_L} has full (column) rank.
  \end{tabular}
  \begin{algorithmic}
	\STATE
      Compute reweighted samples $\boldsymbol{g}:=(g_k)_{k=1}^n$ with $g_k:=\begin{cases}0,  & \varrho_m(\bx^k)=0,\\
      f(\bx^k)/\sqrt{\varrho_m(\bx^k)}, & \varrho_m(\bx^k)\neq 0.
       \end{cases}$
  
  	\STATE
  	Solve the over-determined linear system 
  	\begin{equation}
  	\widetilde{\bL}_m \, (\tilde{c}_1,...,\tilde{c}_{m-1})^\top = \mathbf{g}\,, \; \widetilde{\bL}_m:=\left(l_{k,j}\right)_{k=1,j=1}^{n,m-1},\; l_{k,j}:=\begin{cases}0,  & \varrho_m(\bx^k)=0,\\
  	      \eta_j(\bx^k)/\sqrt{\varrho_m(\bx^k)}, & \varrho_m(\bx^k)\neq 0,
  	       \end{cases}
  	       \label{eq:tilde_L}
  	\end{equation}
  	via least squares (e.g.\ directly or via the LSQR algorithm \cite{PaSa82}), i.e., compute
  	$$
  	(\tilde{c}_1,...,\tilde{c}_{m-1})^\top := (\widetilde{\bL}_m^{\ast}\,\widetilde{\bL}_m)^{-1} \,\widetilde{\bL}_m^{\ast}\, \mathbf{g}.
  	$$
  \end{algorithmic}
   Output:  $\mathbf{\tilde{c}} = (\tilde{c}_1,...,\tilde{c}_{m-1})^\top\in \C^{m-1}$ coefficients of the approximant $\widetilde{S}_{\bX}^m f:=\sum_{j = 1}^{m-1} \tilde{c}_j \, \eta_j$.
\end{algorithm}

As already mentioned, we construct a more suitable distribution which is used to draw the sampling nodes at random.
In particular, we tailor a probability density function $\varrho_m\colon D\to\C$ such that $\mu_m$, which is given by
\begin{equation}
\mu_m(A):=\int_A\varrho_m(\bx)\varrho_D(\mathrm{d}\bx)\quad,\quad A\subset D\mbox{ measurable }\,.\label{eq:def_mu_m}
\end{equation}
Then we may draw the sampling nodes in $\bX\subset D$ i.i.d.\ at random according to $\mu_m$.
For the chosen set $\bX$, we define the approximation operator $\widetilde{S}^m_\bX$ as indicated in Algorithm~\ref{algo1:reweighted}.

Choosing the specific density function
\begin{equation}\label{density1}
	\varrho_m(\bx) := \frac{1}{2}\Bigg(\frac{1}{m-1}\sum\limits_{j=1}^{m-1} |\eta_j(\bx)|^2 + 
	\Bigg(\sum\limits_{j=m}^{\infty} \lambda_j\Bigg)^{-1}\Bigg(K(\bx,\bx) - \sum\limits_{j=1}^{m-1}|e_j(\bx)|^2\Bigg)\Bigg)
\end{equation}
guarantees worst case error estimates which are optimal up to logarithmic factors and up to a specific failure probability.

\begin{theorem}\label{sampling_numbers} Let $H(K)$ be a separable reproducing kernel Hilbert space of complex-valued functions defined on $D$ such that 
$$
\int_{D} K(\bx,\bx) \, \varrho_D(\mathrm{d}\bx) < \infty$$
for some non-trivial $\sigma$-finite measure $\varrho_D$ on $D$, where $(\sigma_j)_{j\in\N}$ denotes the non-increasing sequence of singular numbers of the embedding $\Id:H(K) \to L_2(D,\varrho_D)$.
Let further $\delta\in(0,1/3)$ and $n\in \N$ large enough, such that
\begin{equation} \label{choice_m2}
    m := \left\lfloor\frac{n}{96(\sqrt{2}\log(2n)-\log\delta)}\right\rfloor\ge 2
\end{equation}
holds.
Moreover, we assume $\varrho_m\colon D\to\C$ and $\mu_m$ as stated in~\eqref{density1} and~\eqref{eq:def_mu_m}.
We draw each node in $\bX:=\{\bx^1,\ldots,\bx^n\}\subset D$ i.i.d.\ at random according to $\mu_m$, which yields
\begin{equation}  \Prob\Bigg(\sup\limits_{\|f\|_{{H(K)}}\leq 1} \|f-\widetilde{S}^m_\bX f\|^2_{L_2(D,\varrho_D)}
  \leq \frac{C}{\delta}\,\max\Big\{\sigma^2_m,\frac{\log(8n)}{n}\sum_{j=m}^\infty\sigma_j^2	
 \Big\}\Bigg) \geq 1-3\delta\,,
\label{eq:thm_main_prob_general}
\end{equation}
where $C:=3+16\,C_{\mathrm{R}}^2+4\sqrt{2}\,C_{\mathrm{R}}<49.5$ is an absolute constant.
\end{theorem}

\begin{proof} We emphasize that the 
second argument of the $\max$ term in \eqref{eq:thm_main_prob_general}
makes sense since we know from~\eqref{C_K} that the sequence of singular numbers is square-summable.  As a technical modification of the density function, which has been presented in \cite{KrUl19}, we use the density $\varrho_m \colon D \to \R$ as stated in \eqref{density1}.
As above, the family $(e_j(\cdot))_{j\in \N}$ represents the eigenvectors of the non-vanishing eigenvalues of the compact self-adjoint operator $W_{\varrho_D}:=\Id^*\circ\Id : H(K) \to H(K)$, the sequence $(\lambda_j)_{j\in \N}$ represents the ordered eigenvalues, and finally $\eta_j:=\lambda_j^{-1/2} e_j$. Since we assume the separability of $H(K)$ and the $\sigma$-finiteness of $\varrho_D$ we observe equality in \eqref{C_K}, cf.\ \cite[Thm.\ 4.27]{StChr08}, and thus
we easily see that $\int_D \varrho_m(\bx) \, \varrho_D(\mathrm{d}\bx) = 1$.
Let us define a family of kernels $\widetilde{K}_m(\bx,\by)$, indexed by $m\in \N$, via
\begin{equation}\label{def_Km}
	\widetilde{K}_m(\bx,\by) := \frac{K(\bx,\by)}{\sqrt{\varrho_m(\bx)\varrho_m(\by)}} \,,
\end{equation}	
and a new measure $\mu_m$ as stated in \eqref{eq:def_mu_m}
with the corresponding weighted space $L_2(D,\mu_m)$\,. Clearly, $\widetilde{K}_m(\cdot,\cdot)$ is a positive type function. As a consequence of 
$$
	|K(\bx,\by)| \leq \sqrt{K(\bx,\bx)}\cdot \sqrt{K(\by,\by)}\,,\quad \bx,\by \in D\,,
$$
we obtain by an elementary calculation that with a constant $c_m>0$ it holds
\begin{equation}\label{bounded}
	|K(\bx,\by)| \leq c_m \sqrt{\varrho_m(\bx)}\cdot\sqrt{\varrho_m(\by)}\,.
\end{equation}
Indeed, 
\begin{equation*}
  \begin{split}
	K(\bx,\bx) &= \sum\limits_{j=1}^{m-1} |e_j(\bx)|^2 + \Bigg(K(\bx,\bx) - \sum\limits_{j=1}^{m-1} |e_j(\bx)|^2\Bigg)\\
	&= \sum\limits_{j=1}^{m-1} \lambda_j|\eta_j(\bx)|^2 + \Bigg(\sum\limits_{j=m}^{\infty}\lambda_j\Bigg)\cdot\Bigg(\sum\limits_{j=m}^{\infty}\lambda_j\Bigg)^{-1}\Bigg(K(\bx,\bx) - \sum\limits_{j=1}^{m-1} |e_j(\bx)|^2\Bigg)\\
	&\leq c_m \varrho_m(\bx)
\end{split}
\end{equation*}
with 
$$
	c_m := 2\max\Bigg\{\lambda_1(m-1),\sum\limits_{j=m}^\infty\lambda_j\Bigg\}\,.
$$
Hence, if $\varrho_m(\bx)$ or $\varrho_m(\by)$ happens to be zero then we may put $\widetilde{K}_m(\bx,\by) := 0$ in \eqref{def_Km}. 
In any case, due to \eqref{bounded}, the kernel $\widetilde{K}_m(\bx,\by)$ fits the requirements in Theorem \ref{thm9}. In fact, in Theorem \ref{thm9} it is necessary that $\widetilde{N}(m)$ and $\widetilde{T}(m)$ are well-defined and that we have access to function values in order to create the matrices $\widetilde{\bL}_m$ and take the function values $f(\bx^k)$, $k=1,...,n$. Let us discuss the quantities $\widetilde{N}(m)$ and $\widetilde{T}(m)$ appearing in Theorem \ref{thm9} for this new kernel $\widetilde{K}_m(\bx,\by)$ first. It is clear, that the embedding $\Id:H(\widetilde{K}_m) \to L_2(D,\mu_m)$ shares the same spectral properties as the original embedding. Note that a function $g$ belongs to $H(\widetilde{K}_m)$ if and only if $g(\cdot) = f(\cdot)/\sqrt{\varrho_m(\cdot)}$, $f\in H(K)$, where we always put $0/0:=0$. Clearly, as a consequence of \eqref{bounded} and \eqref{repr} below (together with a density argument), we have that $\varrho(\bx) = 0$ implies $f(\bx) = 0$ for all $f\in H(K)$. 
Moreover, whenever $\|f\|_{H(K)} \leq 1$, the function $g:=f/\sqrt{\varrho_m}$ satisfies 
$\|g\|_{H(\widetilde{K}_m)}\leq 1$. Indeed, let 
\begin{equation}\label{repr}
	f(\cdot) = \sum_{i=1}^N \alpha_i K(\cdot,\bx^i)\,.
\end{equation}
 Then
$\langle f,f\rangle_{H(K)} = \sum_{j=1}^N\sum_{i=1}^N 
\alpha_i \overline{\alpha_j}K(\bx^j,\bx^i)$. We have 
\begin{equation*}
  \begin{split}
	g &= f(\cdot)/\sqrt{\varrho_m(\cdot)} = \sum_{i=1}^N \alpha_i K(\cdot,\bx^i)/\sqrt{\varrho_m(\cdot)}\\
	& = \sum_{i=1}^N \alpha_i \sqrt{\varrho_m(\bx^i)}
	\frac{K(\cdot,\bx^i)}{\sqrt{\varrho_m(\cdot)}\sqrt{\varrho_m(\bx^i)}}\,.
\end{split}
\end{equation*}
This implies 
\begin{equation*}
  \begin{split}
	\langle g,g\rangle_{H(\widetilde{K}_m)} &= 
	\sum_{j=1}^N\sum_{i=1}^N \alpha_i \overline{\alpha_j}\sqrt{\varrho_m(\bx^i)}\sqrt{\varrho_m(\bx^j)}
	\frac{K(\bx^j,\bx^i)}{\sqrt{\varrho_m(\bx^j)}\sqrt{\varrho_m(\bx^i)}}\\
	&= \langle f,f\rangle_{H(K)}\,.
 \end{split}
\end{equation*}
What remains is a standard density argument. 
The singular numbers of the new embedding remain the same. The singular vectors $\tilde{e}_k(\cdot)$ and $\tilde{\eta}_k(\cdot)$ are slightly different. They are the original ones divided by $\sqrt{\varrho_m(\cdot)}$. In fact, 
$$
	\widetilde{N}(m) := \sup\limits_{\bx \in D}\sum\limits_{k=1}^{m-1} |\eta_k(\bx)|^2/\varrho_m(\bx) \leq 
	2(m-1)\sup\limits_{\bx \in D}\sum\limits_{k=1}^{m-1} |\eta_k(\bx)|^2/\sum\limits_{j=1}^{m-1} |\eta_j(\bx)|^2 = 2(m-1)\,.
$$
Furthermore, taking \eqref{K(x,x)} into account, we find
\begin{equation}\label{tildeT(m)}
  \begin{split}
	\widetilde{T}(m) &:= \sup\limits_{\bx \in D}\sum\limits_{k=m}^{\infty} |e_k(\bx)|^2/\varrho_m(\bx)\\ 
	&\leq 
	2\Bigg(\sum\limits_{j=m}^{\infty} \lambda_j\Bigg)\sup\limits_{\bx \in D}\sum\limits_{k=m}^\infty |e_k(\bx)|^2/\Bigg(K(\bx,\bx)-\sum\limits_{j=1}^{m-1}|e_j(\bx)|^2\Bigg)\\
	&\leq 2\Bigg(\sum\limits_{j=m}^{\infty} \lambda_j\Bigg)\sup\limits_{\bx \in D}\sum\limits_{k=m}^\infty |e_k(\bx)|^2/\sum\limits_{j=m}^{\infty}|e_j(\bx)|^2\\
	&\leq 2\sum\limits_{j=m}^{\infty} \lambda_j = 2\sum\limits_{j=m}^{\infty} \sigma_j^2\,.
 \end{split}
\end{equation}

In order to define the new reconstruction operator $\widetilde{S}_{\bX}^m$ we need to create the matrices $\widetilde{\bL}_m$ using the new function system $\tilde{\eta}_k$ and take function evaluations $f(\bx^1),...,f(\bx^n)$.
In more detail, we solve the least squares problem
\begin{equation*}%
\widetilde{\bL}_m{\mathbf{\tilde{c}}}=\boldsymbol{g}, \;
\text{where }
\widetilde{\bL}_m:=\left(\tilde\eta_j(\bx^k)\right)_{k=1,j=1}^{n,m-1}, \;
\boldsymbol{g}:=\left(\frac{f(\bx^1)}{\sqrt{\varrho_m(\bx^1)}},\ldots,\frac{f(\bx^n)}{\sqrt{\varrho_m(\bx^n)}}\right)^\top,
\end{equation*}
and the vector~$\mathbf{\tilde{c}}$ contains the coefficients of the least squares approximation
$S_{\bX}^m(g)=\sum_{j=1}^{m-1}  \tilde{c}_j\tilde{\eta}_j$ of $g:=f/\sqrt{\varrho_m}$. 
This leads to Algorithm~\ref{algo1:reweighted}.
Consequently, Theorem~\ref{thm9} allows to estimate the error
\begin{equation*}%
\|g-S_{\bX}^m(g)\|_{L_2(D,\mu_m)}=\|f-\sqrt{\varrho_m}\,S_{\bX}^m(g)\|_{L_2(D,\varrho_D)}=
\|f-\widetilde{S}^m_\bX f\|_{L_2(D,\varrho_D)}\,,
\end{equation*}
where $\widetilde{S}^m_\bX f:=\sqrt{\varrho_m}\,S_{\bX}^m(g)=\sum_{j=1}^{m-1} \tilde{c}_j\eta_j(\bx)$.

We stress that $\widetilde{S}^{m}_\bX f$ and the direct computation of $S^{m}_\bX f$ using $\bL_m\, \mathbf{c} = \mathbf{f}$ may not coincide since both are based on different least squares problems in general.

It remains to note that for fixed $n$ and $m$ as in \eqref{choice_m2} we have for $\bX=(\bx^1,...,\bx^n)$ the relation
\begin{equation*}
  \begin{split}
	\sup\limits_{\|f\|_{ H(K)} \leq 1}\|f-\widetilde{S}^m_{\bX}(f)\|^2_{L_2(D,\varrho_D)}
	&= \sup\limits_{\|f\|_{ H(K)} \leq 1}\|f/\sqrt{\varrho_m} - S^m_{\bX}(f/\sqrt{\varrho_m})\|^2_{L_2(D,\mu_m)}\\
	& \leq
	\sup\limits_{\|g\|_{ H(\widetilde{K}_m)} \leq 1}\|g - S^m_{\bX}(g)\|^2_{L_2(D,\mu_m)}\,.
  \end{split}
\end{equation*}
Applying a slight modification of Corollary~\ref{cor10}, i.e., setting 
\begin{align*}t:=\frac{1}{\delta}\left(3\sigma_m^2+8\,C_{\mathrm{R}}^2\frac{\log (8n)}{n}\widetilde{T}(m)+4\,C_{\mathrm{R}}\,\sigma_m\,\sqrt{\frac{\log (8n)}{n}\widetilde{T}(m)}\right)
\end{align*}
in the proof, yields
\begin{equation*}
\begin{split}
\Prob\Bigg(\sup\limits_{\|f\|_{{H(K)}}\leq 1} \|f-\widetilde{S}^m_\bX f\|^2_{L_2(D,\varrho_D)}
  \leq \frac{C}{\delta}\max\left(\sigma_m^2,\frac{\log(8n)}{n}\sum_{j=m}^\infty\sigma_j^2\right)
\Bigg) \geq 1-3\delta\,,
\end{split}
\end{equation*}
where we took \eqref{tildeT(m)} into account and $C=3+16C_R^2+4\sqrt{2}C_R$ is as stated above.
\end{proof}

\begin{remark}\label{DavMar} {\em (i)}  In order to prove the theorem in full generality for separable RKHS, we use a technical modification of the density function presented in \cite{KrUl19}. Clearly, as a consequence of \eqref{K(x,x)} the function $\varrho_m$ is positive and defined pointwise for any $\bx \in D$. Moreover, it can be computed precisely from the knowledge of $K(\bx,\bx)$ and the first $m-1$ eigenvalues and corresponding eigenvectors. The pointwise defined density function will be an essential ingredient for drawing the nodes in $\bX$ on the one hand and performing a reweighted least squares fit on the other hand. Note that also point evaluations of the density function are used in the algorithm. To circumvent the lacking injectivity we introduce a new reproducing kernel Hilbert space $H(\widetilde{K}_m)$ built upon the modified density function. To this situation Corollary~\ref{cor10} is applied, and we obtain Theorem~\ref{sampling_numbers}, which also improves the results in \cite{KrUl19} by determining explicit constants. In turn, we show, that the original algorithm proposed by \cite{KrUl19} also works in this more general situation since both densities are equal almost everywhere. 

{\em (ii)} The situation of a non-separable RKHS $H(K)$ is not considered in Theorem \ref{sampling_numbers}. It has been considered in the subsequent paper \cite{MoUl20} by Moeller and T. Ullrich. Here the sampling density has to be modified even further to get a bound 
$$
	\sup\limits_{\|f\|_{H(K)} \leq 1} \|f-\widetilde{S}^m_{\bX}f\|^2_{L_2(D,\varrho_D)} \leq c_1 \max\Big\{\frac{1}{n},\frac{1}{m}\sum\limits_{k\geq m/2} \sigma_k^2\Big\}
$$
with probability larger than $1-c_2n^{1-r}$ and $m:=\lfloor n/(c_3r\log(n)) \rfloor$. With this method we can not get beyond the rate $n^{-1/2}$ in case of non-separable Hilbert spaces $H(K)$. 

{\em (iii)} While this paper was under review, the statement ``recovery with high probability'' has been refined by M. Ullrich \cite{M_Ul20} and, independently, by Moeller, T. Ullrich \cite{MoUl20}. In fact, the failure probability of the above recovery can be controlled by $n^{-r}$ (and is therefore decaying polynomially in the number of samples) by only paying a multiplicative constant $r$ in the bounds, see also (ii). 

{\em (iv)} As a further follow up, Nagel, Sch\"afer and T. Ullrich \cite{NSU20} developed a subsampling technique to select a subset $\mathbf{J} \subset \bX$ of $\mathcal{O}(m)$ nodes out of the randomly chosen node set $\bX$ such that the corresponding least squares recovery operator $\widetilde{S}^m_{\mathbf{J}}$ gives
$$
	\sup\limits_{\|f\|_{H(K)} \leq 1} \|f-\widetilde{S}^m_{\mathbf{J}}f\|^2_{L_2(D,\varrho_D)} \leq \frac{C\log m}{m}\sum\limits_{k=cm}^{\infty} \sigma_k^2
$$
with precisely determined universal constants $C,c>0$.
\end{remark}

\begin{example}\label{leg_2} {\em (i)} If we choose $m \asymp n/\log(n)$ according to \eqref{choice_m2} and assume a polynomial decay of the singular values, i.e., $\sigma_m\lesssim m^{-s}\log^\alpha m$ with $s>1/2$ (which is for instance the case for Sobolev embeddings into $L_2$), we find
$$
\sup\limits_{\|f\|_{{H(K_s)}}\leq 1} \|f-\widetilde{S}_{\mathbf{X}}^mf\|_{L_2(D)} \lesssim m^{-s}\log^{\alpha} m
\asymp n^{-s}\log^{\alpha+s} n\,.
$$
Accordingly, we observe the best possible main rate $-s$ with respect to the used number of samples $n$, i.e., we achieve optimality up to logarithmic factors. For a more specific example, see (ii) below and Section~\ref{mixed1}.

{\em (ii)} Let us reconsider the setting from Example~\ref{ex:legendre}.
Theorem~\ref{sampling_numbers} provides a highly improved sampling strategy compared to the one discussed in Example~\ref{ex:legendre}. Certainly, we change the underlying distribution for the random selection of the sampling nodes and we incorporate weights in the least squares algorithm, cf.\ Algorithm~\ref{algo1:reweighted}. However, this allows for a crucial improvement of the
relation of the maximal polynomial degree $m$ and the number $n$ of sampling nodes to $m\sim n/\log(n)$, which leads to estimates
$$
\sup\limits_{\|f\|_{{H(K_s)}}\leq 1} \|f-\widetilde{S}^m_\bX f\|_{L_2(D)}
  \lesssim \max\left\{\sigma_m,\sqrt{m^{-1}\sum_{k=m}^\infty\sigma_k^2}\right\} \asymp m^{-s}\lesssim \left(\frac{\log(n)}{n}\right)^{s}\,
$$
that are optimal in~$m$ and -- up to logarithmic factors -- even optimal in~$n$. Here $s>1/2$ is admitted. Compared to the results in Example~\ref{ex:legendre}, we obtain the same rates of convergence with respect to the polynomial degree~$m$ and significantly improved rates of convergence with respect to the number~$n$ of used sampling values.
Note, that instead of the density given in \eqref{density1} we may use the Chebychev measure given by the density $\varrho(\bx) = (\pi\sqrt{1-x^2})^{-1}$ on $D = [-1,1]$ since the $L_2$-normalized Legendre polynomials $\mathcal{P}_k(x)$ are dominated by $C(1-x^2)^{-1/4}$ for all $k \in \N_0$, see \cite{RaWa12}. Hence, in contrast to \eqref{density1}, this sampling measure is universal for all $m$. 

In comparison to the results in Example~\ref{ex:legendre}, we obtain a significantly better rate of convergence with respect to the number~$n$ of used sampling values. Applying the recent result mentioned in Remark \ref{DavMar},(iv) we obtain for the subsampled weighted least squares operator the worst-case error bound
$$
\sup\limits_{\|f\|_{{H(K_s)}}\leq 1} \|f-\widetilde{S}_{\mathbf{J}}^mf\|_{L_2(D)}
  \lesssim \frac{\sqrt{\log(n)}}{n^s}\,.
$$
Note, that $\widetilde{S}_{\mathbf{J}}^mf$ uses $n \in \mathcal{O}(m)$ many samples of $f$. 
\qed
\end{example}

\subsection{The power of standard information and tractability}
\label{tract}
The approximation framework studied in this paper has been first considered by Wasilkowski, Wo{\'z}niakowski in 2001, see \cite{WaWo01}. After that many authors dealt with the problem on how well can we approximate functions from RKHS by only using function values. For further historical remarks see \cite{NoWoIII} and \cite[Rem.\ 1]{NSU20}. 

In this context, sampling values are called \emph{standard information} abbreviated by $\Lambda^{\rm std}$.
In contrast to this, one may also allow linear information (abbreviated by $\Lambda^{\rm all}$), which means that the approximation is computed from a number of information about the function coming from arbitrary linear functionals.
Clearly, $\Lambda^{\rm std}$ is a subset of $\Lambda^{\rm all}$ due to \eqref{eq:eval_inner_product}.
It is well known that the best possible worst case error with respect to 
$m$ information coming from $\Lambda^{\rm all}$ can be achieved by approximating functions by means of their corresponding exact Fourier partial sum within $\operatorname{span}\{\eta_j\colon j=1,\ldots,m\}$. Then the corresponding worst case error is determined by $\sigma_{m+1}$, see \cite[Cor.\ 4.12]{NoWoI}, which establishes a natural lower bound on the worst case error with respect to $\Lambda^{\rm std}$.
One crucial question in IBC is to determine whether or not standard information 
$\Lambda^{\rm std}$ is as powerful as linear information $\Lambda^{\rm all}$. In this context, ``power'' is usually identified with the order of convergence of the considered error with respect to the number of used information.

From Theorem \ref{sampling_numbers}, see also Example \ref{leg_2},(i), we obtain that the polynomial order of convergence for standard information is the same as for linear information if the kernel has a finite trace (i.e., $\Id:H(K)\to L_2(D,\varrho_D)$ is a Hilbert-Schmidt operator).
This problem has been addressed in \cite{KWW09}, \cite[Open Problem 1]{NoWo11} and \cite[Open Problem 126]{NoWoIII}. In \cite{KrUl19} this observation has been already made for the situation that the embedding operator is injective (such that the eigenvectors of the strictly positive eigenvalues form an orthonormal basis in $H(K)$, see Remark \ref{injective} above). The contribution of Theorem \ref{sampling_numbers} is to get explicitly determined constants on the one hand. On the other hand, it shows that separability of the RKHS and a finite trace condition is essentially enough for this purpose. Note that the finite trace condition can not be avoided,  see~\cite{HiNoVy08}. 

We will further discuss some consequences for the tractability of this problem. For the necessary notions and definitions from the field of Information Based Complexity, see \cite{NoWoI,NoWoII,NoWoIII}. We comment on polynomial tractability with respect to linear information $\Lambda^{\rm all}$ and standard information $\Lambda^{\rm std}$. Let us consider the family of approximation problems 
$$
	{\operatorname{APP}}_d:H(K_d) \to L_2(D_d,\varrho_{D_d})\,,\quad d\in \N\,,
$$
where $K_d:D_d\times D_d \to \C$ is a family of reproducing kernels. In \cite[Thm.\ 5.1]{NoWoI} {\em strong polynomial tractability} of the family $\{\operatorname{APP}_d\}$ with respect to $\Lambda^{\text{all}}$ is characterized as follows: There is a $\tau>0$ such that 
\begin{equation}\label{pt}
	C:=\sup_d \Bigg(\sum\limits_{j=1}^{\infty} \sigma_{j,d}^{\tau}\Bigg)^{1/\tau} < \infty\,,
\end{equation}
where $\sigma_{j,d}$, $j=1,\ldots$, are the singular values belonging to  ${\operatorname{APP}}_d$ for fixed $d$.
From our analysis in Theorem \ref{sampling_numbers} above we directly obtain a sufficient condition for polynomial tractability with respect to $\Lambda^{\text{std}}$. 
Without going into detail, we denote by
$n^{\rm wor}(\varepsilon,d;\Lambda)$, $\Lambda\in\{\Lambda^{\rm  std},\Lambda^{\rm  all}\}$,
the minimal number $n$ of information out of the specified class $\Lambda$ any algorithm requires in order to achieve a worst case error $\varepsilon$ for the $d$-dimensional approximation problem.
Certainly, $n^{\rm wor}(\varepsilon,d;\Lambda)$ usually depends on $\varepsilon$, the dimension~$d$ and $\Lambda$. At this point, we would like to mention that polynomial tractability means that $n^{\rm wor}(\varepsilon,d;\Lambda)\le C\,\varepsilon^{-p}\,d^{q}$, $C,p,q\ge 0$, i.e.,  $n^{\rm wor}(\varepsilon,d;\Lambda)$ can be bounded by terms that are both polynomial in $d$ and polynomial in $\varepsilon$. Furthermore, the problem $\operatorname{APP}_d$ is called
strongly polynomial tractable in the case that the estimate on $n^{\rm wor}$
holds with $q=0$.

\begin{theorem}\label{ptstd} The family $\{\operatorname{APP}_d\}$ is strongly polynomially tractable with respect to $\Lambda^{\rm std}$ 
\begin{itemize}
\item[(i)] if there exists $\tau \leq 2$ such that \eqref{pt} holds true or
\item[(ii)] if $\{\operatorname{APP}_d\}$ is strongly polynomially tractable with respect to $\Lambda^{\rm all}$ with exponent 
$0<p_{\rm all}<2$, i.e., $n^{\rm wor}(\varepsilon,d;\Lambda^{\rm  all}) \leq C_{\rm all}\,\varepsilon^{-p_{\rm all}}$.
\end{itemize} 
More precisely, there are constants $C_{\rm std}, \delta>0$ only depending on $(C,\tau)$ or $(C_{\rm all}, p_{\rm all})$ such that $$n^{\text{wor}}(\varepsilon,d;\Lambda^{\rm std}) \leq C_{\rm std}\,\varepsilon^{-p_{\rm std}}\log(1/\varepsilon)^{\delta}$$ with 
$p_{\rm std} = \tau$ in case (i) and $p_{\rm std} = p_{\rm all}$ in case (ii).
\end{theorem}
\begin{proof} Note that (ii) implies (i) with $\tau = 2$. Furthermore, Theorem \ref{sampling_numbers} implies strong polynomial tractability if (i) is assumed. In case (i) we may use Stechkin's lemma \cite[Lem.~7.8]{DuTeUl19} which gives that $\sum_{j=m}^{\infty} \sigma_{j,d}^2 \leq C^2m^{-2/\tau +1}$ for all $d$. This gives the exponent $p_{\textnormal{std}} = \tau$ and an additional $\log$ due to \eqref{choice_m2}. If (ii) is assumed then $\sum_{j=m}^{\infty} \sigma_{j,d}^2 \leq C'm^{-2/p_{\rm all} +1}$ for all~$d$.
\end{proof}

Theorem \ref{ptstd} is stronger than \cite[Thm.~26.20]{NoWoIII} in two aspects. As pointed out in the proof, assumption (i) is weaker than (ii), which is essentially the one in \cite[Thm.~26.20]{NoWoIII}. Furthermore, our statement is stronger since $p_{\rm std}$ equals $p_{\rm all}$.
The authors in \cite[26.6.1]{NoWoIII} showed that $p_{\rm std} = p_{\rm all}+\frac{1}{2} [p_{\rm all}]^2$ and proposed that ``the lack of the exact exponent represents another major challenge'' and formulated Open Problem 127.
Our considerations prove that the dependence on $\varepsilon$ of the tractability estimates on $n^{\rm wor}(\varepsilon,d;\Lambda^{\rm all})$ and $n^{\rm wor}(\varepsilon,d;\Lambda^{\rm std})$ coincide up to logarithmic factors. Similar assertions hold true when {\em strong polynomial tractability} is replaced by {\em polynomial tractability}. The modifications are straight forward.  

\begin{example} Let us consider an example from \cite{PaWo10} and \cite{KSU3}, namely, if $s_1 \leq s_2 \leq ... \leq s_d$ then
$$
  H(K_d):=H_{\text{mix}}^{\vec{s}}(\tor^d)= H^{{s_1}}(\tor) \otimes \cdots \otimes H^{s_d}(\tor)
$$
and $\operatorname{APP}_d:H_{\text{mix}}^{\vec{s}}(\tor^d)\to L_2(\tor^d)$. Here $\tor = [0,1]$, where opposite points are identified, and $H^s(\tor)$ is normed by 
$$
	\|f\|^2_{H^s} := \sum\limits_{k \in \Z} |\hat{f}_k|^2 \, w_{s}(k)^2 	
$$
with $w_s(k) = \max\{1,(2\pi)^s|k|^s\}$, see also Section \ref{mixed1}.
The smoothness vector $\vec{s}$ is supposed to grow in $j$, e.g., for $\beta>0$ we have  
\begin{equation}\label{PaWo}
	s_j \geq \beta \log_2(j+1)\,,\quad j\in \N\,.
\end{equation}
It has been shown in \cite{PaWo10}, see also \cite{KSU3}, that the growth condition \eqref{PaWo} is necessary and sufficient for polynomial tractability with respect to $\Lambda^{\textnormal{all}}$. Taking Theorem \ref{ptstd} into account we easily check that \eqref{pt} is satisfied with $\tau \leq 2$ whenever $\beta\cdot \tau > 1$, which means that $\beta >  1/2$ is sufficient for strong polynomial tractability. In this case we obtain for any $1/\beta < \tau \leq 2$ that 
$$n^{\text{wor}}(\varepsilon,d;\Lambda^{{\textnormal{std}}}) \leq C_{\tau}\,\varepsilon^{-\tau}\,.$$

\end{example}

\section{Recovery of individual functions}
\label{sect_individual}
In this section we are interested in the reconstruction of an individual function $f$ (taken from the unit ball of~$H(K)$) from samples at random nodes. In the IBC community, see \cite{NoWoI, NoWoII, NoWoIII}, such a scenario is called ``randomized setting'', which refers to the occurrence of a random element in the algorithm (Monte Carlo). 
The following investigations improve on the results in \cite{WaWo07} from different points of view. On the one hand, it is not necessary to choose sampling nodes according to a whole bunch of probability density functions according to Remark~\ref{imp_samp}. On the other hand, assuming a bounded kernel turns out to be sufficient to obtain the rate of convergence matching that of $(\sigma_k)_{k\in\N}$ when randomly choosing sampling nodes according to the given probability measure $\varrho_D$ due to Corollary~\ref{cor:individual}.

However, the subsequent analysis is related to the one in \cite{CoDaLe13, CoDaLe19}.
With similar techniques as above we will get an estimate for the conditional mean of the individual error $\|f-S^m_\bX f\|_{L_2(D,\varrho_D)}$.

\begin{theorem}\label{individual} Let $\varrho_D$ be a probability measure on $D$ and $H(K)$ denote a reproducing kernel Hilbert which is compactly embedded into the space $L_2(D,\varrho_D)$. Let $m,n\in \N$, $m\ge 2$, be chosen such that \eqref{f1} holds for some $0<\delta <1$. Let further $f$ be a fixed function such that $\|f\|_{ H(K)} \leq 1$. Drawing each sampling node in $\bX=\{\bx^1,\ldots,\bx^n\} \subset D$  i.i.d.\ at random according to $\varrho_D$, we have for the conditional expectation of the individual error
$$
  \Ept\Big(\|f-S^m_\bX f\|^2_{L_2(D,\varrho_D)}\Big|\|\mathbf{H}_m-\mathbf{I}_m\| \leq 1/2\Big) \leq \frac{1}{1-\delta}\Bigg(\sigma_m^2 + \frac{C_{\delta}}{\log n}\sigma_m^2\Bigg)
  \le \frac{1.1}{1-\delta}\,\sigma_m^2\,,
$$
where $0<C_{\delta}\le 0.06$ depends on $\delta$.
\end{theorem}

\begin{proof} This time we follow the proof of \cite[Thm.\ 2]{CoDaLe13} and  obtain
\begin{equation*}%
  \begin{split}
	\|f-S^m_\bX f\|^2_{L_2(D,\varrho_D)} &= \|f-P_{m-1}f\|^2_{L_2(D,\varrho_D)}+\|P_{m-1}f-S^m_\bX f\|^2_{L_2(D,\varrho_D)}\\
	&\leq \sigma^2_m + \|S^m_\bX (P_{m-1}f-f)\|^2_{L_2(D,\varrho_D)}\\
	&\leq \sigma^2_m + \|(\bL_m^{\ast}\,\bL_m)^{-1}\|^2 \, \|\bL_m^{\ast}\, (g(\bx^1),...,g(\bx^n))^\top\|^2_2\\
	&\leq \sigma^2_m + \frac{4}{n^2}\sum\limits_{k=1}^{m-1} \Bigg|\sum\limits_{j=1}^n \overline{{\eta_k}(\bx^j)}g(\bx^j)\Bigg|^2\\
	&= \sigma^2_m + \frac{4}{n^2}\sum\limits_{k=1}^{m-1}\sum\limits_{j=1}^n\sum\limits_{i=1}^n g(\bx^i)\overline{g(\bx^j)}\overline{\eta_k(\bx^i)}\eta_k(\bx^j)\,,
	  \end{split}
\end{equation*}
where $g = f - P_{m-1}f$\,. Averaging on both sides yields 
\begin{equation*}
  \begin{split}
	&\int\limits_{\|\mathbf{H}_m-\mathbf{I}_m\| \leq 1/2}\|f-S^m_\bX f\|^2_{L_2(D,\varrho_D)} \, \varrho^n_D(\mathrm{d}\bx)\\ 
	&\hspace{1cm}\leq\sigma^2_m + \frac{4}{n^2}\sum\limits_{k=1}^{m-1}\sum\limits_{j=1}^n\sum\limits_{i=1}^n \Ept\big(g(\bx^i)\overline{g(\bx^j)}\overline{\eta_k(\bx^i)}\eta_k(\bx^j)\big)	  \\
	&\hspace{1cm}= \sigma_m^2+\frac{4}{n}
	\sum\limits_{k=1}^{m-1}\int_D |g(\bx)|^2|\eta_k(\bx)|^2\,\varrho_D(\mathrm{d}\bx)
	+ \frac{4n(n-1)}{n^2}\sum\limits_{k=1}^{m-1} \Bigg|\int_D g(\bx)\overline{\eta_k(\bx)} \, 
	\varrho_D(\mathrm{d}\bx)\Bigg|^2\,.
	\end{split}
\end{equation*}
Note that the last summand on the right-hand side vanishes since $g$ is orthogonal to $\eta_k$, $k = 1,...,m-1$. 
This gives 
\begin{equation*}
\begin{split}
&\int\limits_{\|\mathbf{H}_m-\mathbf{I}_m\| \leq 1/2}\|f-S^m_\bX f\|^2_{L_2(D,\varrho_D)} \, \varrho^n_D(\mathrm{d}\bx)  \\ 
&\hspace{1cm}\leq \sigma_m^2 + \frac{4 N(m)}{n}\|f-P_{m-1}f\|^2_{L_2(D,\varrho_D)}\\
&\hspace{1cm}\leq \sigma_m^2 + \frac{C_{\delta}}{\log(2n)}\sigma_m^2
\end{split}
\end{equation*}
by taking $4 N(m)/n \leq C_{\delta}/\log(2n)$ into account, see \eqref{f1}, and there exists a $C_\delta\le\frac{4}{48\sqrt{2}}\le 0.06$\,.
\end{proof}

Analogously to the proof of Corollary~\ref{cor10}, one shows the following result.

\begin{corollary}\label{cor:individual}
Under the same assumptions as in Theorem \ref{individual} it holds for fixed $\delta >0$
\begin{equation*}
  \Prob\left(
  \|f-S^m_\bX f\|^2_{L_2(D,\varrho_D)} \leq \frac{1}{\delta}\, \sigma_m^2\,\left(1+\frac{0.06}{\log n}\right)\,\right) \geq 1-3\delta\,.
\end{equation*}
\end{corollary}

\begin{remark}\label{imp_samp} 
{\em(i)} Following \cite{CoMi17} we may relax the condition \eqref{f1} to 
$$
    m := \left\lfloor\frac{n}{48(\sqrt{2}\log(2n)-\log{\delta})}\right\rfloor+1\,,
$$
when sampling with respect to the new measure (importance sampling)
$$
	\mu_m(A) := \int_A \varrho_m(\bx) \, \varrho_D(\mathrm{d}\bx)\,,
$$
where  $\varrho_m$ is given by 
$$
	\varrho_m(\bx) := \frac{1}{m-1}\sum\limits_{k=1}^{m-1}|\eta_k(\bx)|^2\,.
$$
Then, the corresponding approximation of the function $f$ is based on the solution of a weighted least squares problem similar to the one discussed in Section~\ref{Sect5_opt}.
Note that we do not have to assume the square summability of the singular numbers of the embedding $\Id:H(K)\to L_2(K,\varrho_D)$ here. 

{\em(ii)} The result of Theorem~\ref{individual} advanced with (i) directly leads to estimates on the power of standard information, see Section~\ref{tract}, in the so-called randomized setting, cf.~\cite{NoWoIII}.
Roughly speaking, $n$ sampling values (standard information) in the randomized setting are at least as powerful as $ c n/\log{n}$ linear information, i.e., the
supremum over all $f\in H(K)$ of the expected approximation error that is caused by the weighted least squares regression is almost as good as the best possible worst case approximation error caused by the projection.
The recent work \cite{LuWa21a} treats that topic in full detail.
Moreover, in \cite{LuWa21b} the authors present improvements on the power of standard information in the so-called average case setting.
Both paper investigate the weighted least squares regression given in Algorithm~\ref{algo1:reweighted} using the techniques that yielded the results of this section.

{\em (iii)} Together with the Weaver subsampling technique in \cite{NSU20} we obtain 
a further significant improvement of the results in \cite{LuWa21a} in the situation above if the sequence $(\sigma_k)_{k\in\N}$ decays ``faster'' than $k^{-1/2}$. 
Just for the sake of completeness, we state the corresponding general relation  
\begin{equation*}
	e^{\rm ran}(n,d;\Lambda^{\rm std}) \leq C\min\{\sqrt{\log n}\cdot e^{\rm det}(c_1n,d;\Lambda^{\rm all}), 
	e^{\rm det}(c_2n/\log n,d;\Lambda^{\rm all})\}\,, 
\end{equation*}	
with three universal constants $C,c_1,c_2>0$, where we use the notation from \cite{NoWoI,NoWoII,NoWoIII}.
\end{remark}

\section{Optimal weights for numerical integration}
\label{Numint}
We consider the problem of approximating the integral with respect to $\mu_D$ of a function $f$ by the cubature rule $\operatorname{Q}_\bX^m$
\begin{align*}
\IntEx{\mu_D} f:=\int_D f \, \mathrm{d}\mu_D
\approx
\operatorname{Q}_\bX^m f:=\sum_{j=1}^n q_j f(\bx^j)= \boldsymbol{q}^\top \boldsymbol{f}\,,
\end{align*}
where the weights $q_j$ are determined by $\bX = \{\bx^1,...,\bx^n\}$. Indeed, assuming $\bL_m$ of full column rank and following Oettershagen \cite{Oe17}, we set
\begin{align}
\boldsymbol{q}&:= \overline{\bL_m}\, \overline{(\bL_m^{\ast}\,\bL_m)^{-1}}\, \boldsymbol{b}\;\in\C^n\,,\nonumber
\\
\text{where } \; \boldsymbol{b}&:=\big(b_k\big)_{k=1}^{m-1} \in\C^{m-1} \text{ with } \; b_k:=\int_D \eta_k\,\mathrm{d}\mu_D\,.
\label{eq:compute_int_weights:b}
\end{align}
In fact, the cubature rule $\operatorname{Q}_\bX^m$ is the implicit integration of the least squares solution
$S^m_\bX f:=\sum_{k=1}^{m-1} c_k \eta_k$, cf.~\eqref{eq:def_SmX}, $\boldsymbol{c}:= (\bL_m^{\ast}\,\bL_m)^{-1}\,\bL_m^\ast \, \boldsymbol{f}$, since
\begin{align*}
\operatorname{Q}_\bX^m f&=\sum_{j=1}^n q_j f(\bx^j)
= \boldsymbol{b}^\top (\bL_m^{\ast}\,\bL_m)^{-1} \, \bL_m^\ast \, \boldsymbol{f}
= \sum_{k=1}^{m-1} c_k\int_D \eta_k\,\mathrm{d}\mu_D=\int_D S^m_\bX f\mathrm{d}\mu_D\,.
\end{align*}
Using this, we give upper bounds on the integration error caused by $\operatorname{Q}_\bX^m$.

\begin{theorem}\label{thm:int_error}
Let $\mu_D$ be a measure on $D$ such that $L_1(D,\varrho_D) \hookrightarrow L_1(D,\mu_D)$. Denote with 
$$
	C_{\varrho,\mu} := \|\operatorname{Id}:L_1(D,\varrho_D) \to L_1(D,\mu_D)\| < \infty
$$
the norm of the embedding. Under the same assumptions as in Theorem \ref{thm9} 
it holds for fixed $\delta >0$

\begin{equation*}
  \Prob\left(\sup\limits_{\|f\|_{{H(K)}}\leq 1} \Bigg|\int_D f \, \mathrm{d}\mu_D-\operatorname{Q}_\bX^m f\Bigg|^2
  \leq \frac{29\, C_{\varrho,\mu}^2}{\delta}\,\max\Bigg\{\sigma^2_m,\frac{\log(8n)}{n}T(m)\Bigg\}\right) \geq 1-3\delta\,.
\end{equation*}

\end{theorem}

\begin{proof}
Using the embedding relation of the different $L_1$-spaces we have 
\begin{align}
\Bigg|\int_D f \, \mathrm{d}\mu_D-\operatorname{Q}_\bX^m f\Bigg|=\Bigg|\int_D f-S^m_\bX f \, \mathrm{d}\mu_D\Bigg|\nonumber
&\le
\int_D \Big|f-S^m_\bX f \Big| \, \mathrm{d}\mu_D\label{f100}\\
&\le C_{\varrho,\mu}\int_D \Big|f-S^m_\bX f \Big| \, \mathrm{d}\varrho_D\vspace{0.7cm}\\
&\le C_{\varrho,\mu}\,\|f-S^mf\|_{L_2(D,\varrho_D)}.\nonumber
\end{align}
We conclude the proof using Corollary \ref{cor10}.
\end{proof}

In Theorem \ref{thm:int_error} we assume that the kernel is bounded, i.e., $\sup_{\bx \in D} K(\bx,\bx)<\infty$. However, as we will see below it is enough to assume $\int_{D}K(\bx,\bx) \, \varrho_D(\mathrm{d}\bx)<\infty$ and that $\varrho_D$ is a finite measure. We define the following modified (reweighted) cubature formula
\begin{equation*}
  \begin{split}
	\widetilde{Q}^m_{\bX} f&:= \int_D \widetilde{S}^m_{\bX}f(\bx)\varrho_D(\bx) = \widetilde{\boldsymbol{q}}^\top \boldsymbol{f}\\&=(\mathbf{D}_{\varrho_m}\overline{\widetilde{\bL}_m}\,\overline{(\widetilde{\bL}_m^{\ast}\widetilde{\bL}_m)^{-1}}\, \boldsymbol{b})^\top \boldsymbol{f}
= \boldsymbol{b}^\top (\widetilde{\bL}_m^{\ast}\widetilde{\bL}_m)^{-1} \widetilde{\bL}_m^\ast \mathbf{D}_{\varrho_m}\boldsymbol{f}\,,
\end{split}
\end{equation*}
where $\boldsymbol{f}$ is the vector of function samples $\big(f(\bx^1),...,f(\bx^n)\big)$, the vector $\boldsymbol{b}$ is given as in~\eqref{eq:compute_int_weights:b}, $\widetilde{\bL}_m$ as in Algorithm \ref{algo1:reweighted}, and 
$$\mathbf{D}_{\varrho_m} = \diag\bigg(1/\sqrt{\varrho_m(\bx^1)},...,1/\sqrt{\varrho_m(\bx^n)}\bigg)\,.$$ 
Here $\varrho_m(\cdot)$ is given by \eqref{density1} above and depends on the spectral properties of the kernel. Note that we simply put the respective entry to $0$ in $\mathbf{D}_{\varrho_m}$ if $\varrho_m(\bx^n)$ happens to be zero. 
\begin{theorem}\label{thm:int_error2} If $\varrho_D$ denotes a finite measure and $\int_D K(\bx,\bx) \, \varrho_D(\mathrm{d}\bx)<\infty$, then
\begin{equation}\label{f101}
	\sup\limits_{\|f\|_{H(K)}\leq 1}\Bigg|\int_D f(\bx) \, \varrho_D(\mathrm{d}\bx) - \widetilde{Q}^m_{\bX}(f)\Bigg|^2 \leq 
	\frac{50}{\delta}\varrho_D(D)\max\left\{\sigma^2_m ,\frac{\log (8n)}{n}\sum\limits_{j=m}^{\infty}\sigma_j^2\right\}
\end{equation}
holds with probability larger than $1-3\delta$ if the $n$ nodes in $\bX$ are drawn i.i.d.\ according to \eqref{eq:def_mu_m} above.
\end{theorem}
\begin{proof}
Using the bound on the $L_2(D,\varrho_D)$ error from Theorem \ref{sampling_numbers} together with \eqref{f100} above, we obtain \eqref{f101} with high probability if the nodes $\bX$ are sampled according to the measure \eqref{eq:def_mu_m}. \end{proof}

\begin{remark} The previous result essentially improves on the bound in \cite[Thm.\ 1]{BeBaCh19} if the singular values decay polynomially. Note that we assume that $\varrho_D$ is a finite measure and $\int_D K(\bx,\bx) \, \varrho_D(\mathrm{d}\bx)<\infty$. In \cite[Thm.\ 1]{BeBaCh19} a logarithmic oversampling is not required, however the error bounds are worse by a factor $n$, which is substantial, e.g., in the case of the Sobolev kernel, see Section~\ref{Num_int_per}.
\end{remark}

\section{Hyperbolic cross Fourier regression}
\label{mixed1}
In the sequel we are interested in the recovery of functions from periodic Sobolev spaces. That is, we consider functions on the torus $\tor^d \simeq [0,1)^d$ where opposite points are identified. Note that the unit cube $[0,1]^d$ is preferred here since it has Lebesgue measure $1$ and is therefore a probability space. We could have also worked with $[0,2\pi]^d$ and the Lebesgue measure (which can be made a probability measure by a $d$-dependent rescaling). The general error bounds for the recovery error given below (in terms of $(\sigma_j)_{j\in\N}$ like in Theorem \ref{thm:asymp_hsmix}) are not affected by this rescaling since the sequence $(\sigma_j)_{j\in \N}$ then also changes. However, some of the preasymptotic estimates for the $(\sigma_j)_{j\in \N}$ are sensitive with respect to a different domain as the results in Krieg \cite{Kr18} point out.

For $\alpha\in \N$ we define the space $H^{\alpha}_{\textnormal{mix}}(\tor^d)$ as the Hilbert space with the inner product 
\begin{equation}\label{equ:inner_product_star}
	\langle f , g\rangle_{H^{\alpha,\ast}_{\textnormal{mix}}} := 	\sum\limits_{\bj \in \{0,\alpha\}^d}\langle D^{(\bj)}f,D^{(\bj)}g \rangle_{L_2(\tor^d)}\,.
\end{equation}
Defining the weight 
\begin{equation}\label{norm2}
	w_{\alpha,\ast}(k) = (1+(2\pi |k|)^{2\alpha})^{1/2}\,,\quad k\in \Z\,,
\end{equation}
and the univariate kernel function 
\begin{equation}\nonumber
	K^1_{\alpha,\ast}(x,y):=\sum\limits_{k\in \Z} \frac{\exp(2\pi\mathrm{i}k(y-x) )}{w_{\alpha,\ast}(k)^2}\,,\quad x,y\in \tor\,,
\end{equation}
directly leads to
\begin{equation}\label{kerneld}
		K^d_{\alpha,\ast}(\bx,\by):=	K^1_{\alpha,\ast}(x_1,y_1)\otimes \cdots \otimes 	K^1_{\alpha,\ast}(x_d,y_d)\,,\quad \bx,\by \in \tor^d\,,
\end{equation}
which is a reproducing kernel for $H^\alpha_{\textnormal{mix}}(\tor^d)$.  The Fourier series representation of
$K^d_{\alpha,\ast}(\bx,\by)$ is specified by
$$
	K^d_{\alpha,\ast}(\bx,\by):=\sum\limits_{\bk\in \Z^d} \frac{\exp(2\pi\mathrm{i}\,\bk\cdot(\by-\bx) )}{w_{\alpha,\ast}(k_1)^2\cdots w_{\alpha,\ast}(k_d)^2}
	= \sum\limits_{\bk\in \Z^d} \frac{\exp(2\pi\mathrm{i}\,\bk\cdot(\by-\bx))}{\prod_{j=1}^d (1+(2\pi |k_j|)^{2\alpha})} \,,\quad \bx,\by\in \tor^d\,.
$$
In particular, for any $f\in H^\alpha_{\textnormal{mix}}(\tor^d)$ we have 
$$
	f(\bx) = \langle f, K^d_{\alpha,\ast}(\bx,\cdot) \rangle_{H^{\alpha,\ast}_{\textnormal{mix}}}\,.
$$
The kernel defined in \eqref{kerneld} associated to the inner product \eqref{equ:inner_product_star} can be extended to the case of fractional smoothness $s>0$ replacing $\alpha$ by $s$ in \eqref{norm2}--\eqref{kerneld} which in turn leads to the inner product
\begin{equation*}%
 \langle f , g\rangle_{H^{s,\ast}_{\textnormal{mix}}} := \sum\limits_{\bk\in \Z^d} \hat{f}_{\bk} \, \overline{\hat{g}_{\bk}} \, \prod_{j=1}^d w_{s,\ast}(k_j)^2\,
\end{equation*}
and the corresponding norm
\begin{equation*}%
 \|f\|_{\ast}:=\|f\|_{H^{s,\ast}_{\textnormal{mix}}} := \left(\sum\limits_{\bk\in \Z^d} |\hat{f}_{\bk}|^2 \, \prod_{j=1}^d w_{s,\ast}(k_j)^2\right)^{1/2}\,.
\end{equation*}

The (ordered) sequence $(\lambda_j^\ast)_{j=1}^{\infty}$ of eigenvalues of the corresponding mapping $W = \Id^{\ast}\circ \Id$, where $\Id\colon H\big(K^d_{s,\ast}\big) \to L_2(\tor^d)$ is the non-increasing rearrangement of the numbers 
\begin{equation*}
	\Bigg\{\lambda_{\bk}^\ast:=\prod\limits_{j=1}^d w_{s,\ast}(k_j)^2=\prod\limits_{j=1}^d (1+(2\pi |k_j|)^{2s})^{-1} \colon \bk \in \Z^d\Bigg\}\,.
\end{equation*}
The corresponding orthonormal system $(e_j^\ast)_{j=1}^{\infty}$ in $H\big(K^d_{s,\ast}\big)$ is given by 
$$
	\Bigg\{e_{\bk}(\bx):=\exp(2\pi\mathrm{i}\,\bk\cdot \bx)\prod\limits_{j=1}^d (1+(2\pi |k_j|)^{2s})^{-1/2} \colon \bk \in \Z^d\Bigg\}\,.
$$
Consequently, the orthonormal system $(\eta_j^\ast)_{j=1}^{\infty}$ in $L_2(\tor^d)$ is the properly ordered classical Fourier system $\eta_{\bk}(\bx) = \exp(2\pi\mathrm{i}\,\bk\cdot \bx)$. 
This directly implies the following behavior of the quantity $N(m)$ defined in \eqref{f1b}. It holds 
\begin{equation*}
		N(m) = m-1\quad \text{and}\quad T_\ast(m) = \sum\limits_{j = m}^\infty \lambda_j^\ast = \sum\limits_{j = m} ^\infty (\sigma^\ast_j)^2\,.
\end{equation*}

\begin{remark}
It is possible to define a smoothness vector $\mathbf{s} = (s_1,...,s_d)$ to emphasize different smoothness in different coordinate directions. Such kernels will be denoted with $K_{\mathbf{s}}(\bx,\by)$. In \cite{KSU3} the authors establish preasymptotic error bounds which can be used for the least squares analysis as we will see below.  
\end{remark}

Recent estimates in \cite{Ku19} allow for determining uniform recovery guarantees with preasymptotic error bounds.
For this study, we need to change the kernel weight to a less natural, but for preasymptotic considerations more convenient
structure of the weight
\begin{equation}\label{equ:weight_pound}
 w_{s,\#}(k) = (1+2\pi |k|)^{s} \,,\quad k\in \Z\,,
\end{equation}
$s>0$.
As a consequence, the univariate kernel
\begin{equation*}
	K^1_{s,\#}(x,y):=\sum\limits_{k\in \Z} \frac{\exp(2\pi\mathrm{i}k(y-x) )}{w_{s,\#}(k)^2}\,,\quad x,y\in \tor\,,
\end{equation*}
as well as the tensor product kernel
\begin{equation*}
 K^d_{s,\#}(\bx,\by) := \sum\limits_{\bk\in \Z^d} \frac{\exp(2\pi\mathrm{i}\,\bk\cdot(\by-\bx))}{\prod_{j=1}^d (1+2\pi |k_j|)^{s}}\,,\quad \bx,\by\in \tor^d\,,
\end{equation*}
changes and has modified Fourier series expansions.
Of course, the weight $w_{s,\#}$ yields an equivalent norm
$$
 \|f\|_{\#}:=\|f\|_{H^{s,\#}_{\textnormal{mix}}(\tor^d)} = \left(\sum\limits_{\bk\in \Z^d} |\hat{f}_{\bk}|^2 \prod_{j=1}^d (1+2\pi |k_j|)^{2s} \right)^{1/2}
$$
in the space $H^s_{\textnormal{mix}}(\tor^d)$.
However, from a complexity theoretic point of view, it is worth noting the difference of both approaches. The respective unit balls belonging to both norms differ significantly since the equivalence constants for both norms may depend badly on $d$.
Moreover, we stress on the fact that the non-increasing rearrangements of the eigenvalues
$(\lambda_j^\#)_{j=1}^{\infty}$ of the mapping $W = \Id^{\ast}\circ \Id$, where $\Id\colon H\big(K^d_{s,\#}\big) \to L_2(\tor^d)$, differs from $(\lambda_j^\ast)_{j=1}^{\infty}$ since
the associated mappings $j\to\bk$ do not coincide.
Accordingly, the sampling operators~$S_\bX^{m,\ast}$ and~$S_\bX^{m,\#}$ are defined with respect to the different non-increasing rearrangements of the basis functions $\eta_{\bk}$, and thus, may also differ.

Despite the fact that there are many more different equivalent norms for $H^s_{\textnormal{mix}}(\tor^d)$, we will use only the mentioned ones in order to apply advantageous
known upper bounds on the eigenvalues $\lambda_j^{\ast}$ and $\lambda_j^\#$. In the regime of interest here, the recovery of periodic functions with mixed smoothness, we even have preasymptotic bounds for these eigenvalues/singular values available (see \cite{KuSiUl15, Ku19, Kr18}).
Note that theoretically everything is known about the singular values $\sigma_m^\square$, $\square\in\{\ast,\#\}$, since the behavior of this sequence is determined by the non-increasing rearrangement of the reciprocals of the tensor product weights $\prod_{j=1}^d w_{s,\square}(k_j)$, cf.~\eqref{norm2} and~\eqref{equ:weight_pound}.
The analysis of the rearrangement of multi-indexed sequences has been revealed
that the upper bound
\begin{equation}
\sigma_m^\#\le\min\left(1,\frac{16}{3m}\right)^{\frac{s}{1+\log_2d}}\le \left(\frac{16}{3m}\right)^{\frac{s}{1+\log_2d}}
\label{eq:sigma_upper_bound_kuehn}
\end{equation}
holds for all $m\in\N$, cf.~\cite[Thm.\ 4.1]{Ku19}.

One may argue that the kernel $K^d_{s,\#}$ is less ``natural'' than the kernel~$K^d_{s,*}$. For this purpose  we use the observation                                                                   by
Krieg~\cite{Kr18},
which yields the preasymptotic bound
$$
	\sigma_{m}^\ast \leq \Bigg(\frac{1.26}{m}\Bigg)^{\frac{1.83\cdot s}{4+\log_2(d)}}
$$
in the range $2\leq m\leq 3^d$, 
where the exponent scales also linearly in $s$.

\subsection{Uniform recovery of functions from $H^s_{\textnormal{mix}}(\tor^d)$}
\label{sect:uniform}
First, we consider the asymptotic behavior of the sampling error caused by the presented least squares approach. Since the asymptotic bounds on $\lambda_j^\ast$ and $\lambda_j^\#$ differ only by constants, that we do not specify explicitly, we study both cases collectively. For $\square\in\{\ast,\#\}$, the
asymptotic behavior of the sequence $(\sigma_j^\square)_{j=1}^\infty$ for the embedding $\Id:H(K_{s,\square}^d) \to L_2(\tor^d)$ has been known for a long time, cf. \cite[Chapt.\ 4]{DuTeUl19} and the references therein. There is a constant $\tilde{C}_d^\square$ which depends exponentially on $d$ such that 
\begin{equation}\label{f27}
		\sigma_m^\square \leq \tilde{C}_d^\square\, m^{-s}(\log m)^{s(d-1)}\,, \quad m\in \N\,.
\end{equation}

As a direct consequence of Corollary~\ref{cor10} we determine 
explicit error bounds as well as the asymptotic error behavior~\eqref{equ:fourier:wc_sampl_error}, where the latter holds for all equivalent norms in $H^s_{\textnormal{mix}}(\T^d)$. Please note 
that the constants $C_d^\square$ depend heavily on the specific norm. 
Moreover, a statement similar to~\eqref{equ:fourier:wc_sampl_error} can be also obtained with the technique in \cite{KrUl19}.

\begin{theorem}\label{thm:asymp_hsmix} Let $d\in \N, s>1/2$, $\delta>0$ and $n\in \N$ be given. Choose $m\in\N$, $m\ge 2$, 
$$m \leq \left\lfloor\frac{n}{48(\sqrt{2}\log(2n)-\log\delta)}\right\rfloor+1\,
$$
and draw the $n$ sampling nodes in $\bX$ uniformly i.i.d.\ at random in $\T^d$.
Then, we achieve
\begin{align*}
 \Prob\Bigg(\sup\limits_{\|f\|_{\square}\leq 1} \|f-S_\bX^{m,\square} f\|^2_{L_2(\T^d)}
  \leq \frac{29}{\delta}\,\max\Bigg\{(\sigma_m^\square)^2,\frac{\log(8n)}{n}\sum_{k=m}^\infty(\sigma_k^\square)^2	
 \Bigg\}\Bigg) \geq 1-3\delta\,.
\end{align*}
In particular, there is a constant $C_d^\square>0$ depending on $d$ such that for $0<\delta<1/3$ it holds
\begin{equation}\label{equ:fourier:wc_sampl_error}
  \Prob\Bigg(\sup\limits_{\|f\|_{\square}\leq 1} \|f-S_\bX^{m,\square} f\|_{L_2(\tor^d)}\\
  \leq \frac{C_d^\square}{\sqrt{\delta}}n^{-s}(\log n)^{sd}\Bigg) \geq 1-3\delta\,.
\end{equation}
\end{theorem}

\begin{proof} The result follows from Corollary \ref{cor10} and our specific situation where $B=1$, $N(m)=m-1$ and $T_\square (m)=\sum_{k=m}^\infty(\sigma_k^\square)^2$.
We estimate $T_\square(m)$ using \cite{CoKuSi16}
$$
	T_\square(m) := \sup\limits_{\bx \in D}\sum\limits_{k=m}^{\infty}|e_k(\bx)|^2=
	\sum\limits_{k = m}^{\infty} (\sigma_k^\square)^2
	\asymp m^{-2(s-1/2)}(\log m)^{2(d-1)s} \asymp m(\sigma_m^\square)^2\,.
$$
Hence, the right-hand side in \eqref{eq:cor_main_prob} can be bounded from above by a constant times
$\sigma_m^\square$, which behaves as $n^{-s}(\log n)^{sd}$\,.
\end{proof}

In addition, we investigate the preasymptotic error behavior using the aforementioned
estimates \eqref{eq:sigma_upper_bound_kuehn} on the singular values $\sigma_m^\#$ that
belongs to $\Id\colon H\big(K^d_{s,\#}\big) \to L_2(\tor^d)$. Since the upper bounds
have been proven only for this specific type of mappings, the following results, in particular the explicitly determined constants, may only hold for RKHS with weight functions
$w_s(\bk)\ge\prod_{j=1}^d(1+|k_j|)^s$, which is fulfilled for $w_{s,\#}$.

\begin{theorem}\label{thm:preasymp_hsmix}
Let $d\in \N$, $s>(1+\log_2d)/2$, $0<\delta<1$, and $n\in \N$ such that   
\begin{equation}\label{choos_m}
m:= \left\lfloor\frac{n}{48(\sqrt{2}\log(2n)-\log{\delta})}\right\rfloor+1
\end{equation}
is at least $2$. Drawing the $n$ sampling nodes in $\bX$ uniformly i.i.d.\ at random in $\T^d$ yields
\begin{align*}
&\Ept\Bigg(\sup\limits_{\|f\|_{\#}\leq 1} \|f-S^{m,\#}_\bX f\|^2_{L_2(\tor^d)}\Big|\|\mathbf{H}_{m,\#}-\mathbf{I}_m\| \leq 1/2\Bigg) \!\le
\frac{29}{6\,(1-\delta)}\frac{2s}{2s-1-\log_2d}\left(\frac{16}{3m}\right)^{\frac{2s}{1+\log_2d}}.
\end{align*}
\end{theorem}
\begin{proof}
We apply Theorem~\ref{thm9} and take into account that we choose 
$$
m := \left\lfloor\frac{n}{48(\sqrt{2}\log(2n)-\log{\delta})}\right\rfloor+1\le\frac{n}{48\sqrt{2}\log(2n)}
+1$$
for large enough $n$. Furthermore,
we set $c:=16/3$, $\beta:=2s/(1+\log_2d)$ and estimate $T_\#(m)$, see \eqref{eq:sigma_upper_bound_kuehn},
\begin{align*}
	T_\#(m) &:= \sup\limits_{\bx \in D}\sum\limits_{k=m}^{\infty}|e_k(\bx)|^2=
	\sum\limits_{k = m}^{\infty} (\sigma_k^\#)^2
	\le c^{\beta}\sum_{k=m}^\infty k^{-\beta}\\
	&\le c^{\beta}\left(
	m^{-\beta}+\frac{1}{\beta-1}m^{-\beta+1}\right)
	\le c^{\beta}\frac{\beta}{\beta-1}m^{-\beta+1}\,.
\end{align*}
Taking \eqref{eq:Ept_sum_detail} into account, we bound
\begin{align*}
&\Ept\Bigg(\sup\limits_{\|f\|_{\#}\leq 1} \|f-S^{m,\#}_\bX f\|^2_{L_2(\tor^d)}\Big|\|\mathbf{H}_{m,\#}-\mathbf{I}_m\| \leq 1/2\Bigg)\\
&\quad\le
\frac{c^\beta m^{-\beta}}{1-\delta}\left(3+8\,C_{\mathrm{R}}^2\frac{m\log (8n)}{n}\frac{\beta}{\beta-1}+4\,C_{\mathrm{R}}\,\sqrt{\frac{m\log (8n)}{n}\frac{\beta}{\beta-1}}\right)\\
&\quad\le 
\frac{\beta\,c^\beta m^{-\beta}}{(1-\delta)(\beta-1)}\left(3+8\,C_{\mathrm{R}}^2
b+4\,C_{\mathrm{R}}\,\sqrt{b}\right)\,,
\end{align*}
where $b:=\frac{\log (8n)}{48\sqrt{2}\log(2n)}+\frac{\log (8n)}{n}$.
The term in the brackets is monotonically decreasing in~$n$. We stress that the last estimates are reasonable for $m\ge 2$ and thus, we need at least $n\ge 464$. This choice of $n$ leads to an upper bound of the term within the brackets which is $29/6$. Thus, for $m\ge2$, the estimate
\begin{align*}
&\Ept\Big(\sup\limits_{\|f\|_{\#}\leq 1} \|f-S^{m,\#}_\bX f\|^2_{L_2(\tor^d)}\Big|\|\mathbf{H}_{m,\#}-\mathbf{I}_m\| \leq 1/2\Big)\le
\frac{29\,\beta\,c^\beta m^{-\beta}}{6\,(1-\delta)(\beta-1)}
\end{align*}
holds and the assertion follows.
\end{proof}
Similar to Corollary~\ref{cor10}, we apply Markov's inequality to get a lower bound on the success probability of the randomly chosen sampling set.
\begin{corollary}\label{cor:hsmix_preasymp}
Under the same assumptions as in Theorem~\ref{thm:preasymp_hsmix} it holds
\begin{equation*}
  \begin{split}
  \Prob\Bigg(\sup\limits_{\|f\|_{\#}\leq 1} &\|f-S^{m,\#}_\bX f\|^2_{L_2(\T^d)}\\
  &\leq \frac{29}{6\,\delta}\frac{2s}{2s-1-\log_2d}\left(\frac{256(\sqrt{2}\log(2n)-\log\delta)}{n}\right)^{\frac{2s}{1+\log_2d}}
  \Bigg)\\
  \ge\Prob\Bigg(&\sup\limits_{\|f\|_{\#}\leq 1}\|f-S^{m,\#}_\bX f\|^2_{L_2(\T^d)}
  \leq \frac{29}{6\,\delta}\frac{2s}{2s-1-\log_2d}\left(\frac{16}{3m}\right)^{\frac{2s}{1+\log_2d}}
  \Bigg)  
   \geq 1-3\delta\,.
 \end{split}
\end{equation*}
\end{corollary}
\begin{proof}
We follow the argumentation in the proof of Corollary~\ref{cor10} using
the inequality
\begin{align*}
m^{-1}\le\frac{48(\sqrt{2}\log(2n)-\log\delta)}{n}\,.
\end{align*}
\end{proof}

\begin{example}
For $s=5$ and $d=16$, we would like to fulfill
\begin{equation*}
  \Prob\Bigg(\sup\limits_{\|f\|_{\#}\leq 1} \|f-S^{m,\#}_\bX f\|_{L_2(\T^d)}\leq 0.1
  \Bigg) \geq 0.99\,,
\end{equation*}
i.e., with $\delta=1/300$ we choose $m:=2\,873$ smallest possible such that
$$
2900 \left(\frac{16}{3m}\right)^{2}\le 0.01
$$
holds. Clearly, for $n$ such that $m-1=2\,872\le\frac{n}{48(\sqrt{2}\log(2n)-\log{\delta})}$ holds,
we observe the desired estimate. We choose the smallest possible $n:=3\,879\,166$.
\end{example}

\subsection{Numerical integration of periodic functions}\label{Num_int_per}
In \cite[Sec.~4.2]{Oe17}, the author discussed the construction of stable cubature weights for the approximation of the integral
$$
	\operatorname{I}(f):=\int_{\tor^d} f(\bx)\,\mathrm{d}\bx \approx \operatorname{Q}_\bX^mf := \sum\limits_{j=1}^n q_jf(\bx^j)
$$
for functions from periodic Sobolev spaces with dominating mixed smoothness. The integration nodes $\bX$ are drawn uniformly and independently at random from $\T^d$. Below, in Corollary~\ref{cor:fourier:int_error_wc}, the cubature rule $\operatorname{Q}_\bX^m$ is fixed for the whole class.
The result in \cite[Thm.~4.5]{Oe17} bounds the worst-case integration error from above by
$$
\lesssim_{d,s}n^{-s+1/2}(\log n)^{sd-1/2}.
$$
Corresponding numerical tests promise better behavior of the integration error, cf.\ \cite[Rem.~4.6]{Oe17}.
Our theoretical results of this section confirm that the optimal main rate of the presented approach in~\cite{Oe17} is $n^{-s}$, in particular
we obtain the upper bound $\lesssim_{d,s}n^{-s}(\log n)^{sd} $ for this specific setting. We achieve the following statement on the worst-case integration error.

\begin{corollary}\label{cor:fourier:int_error_wc}
Let $d\in \N, s>1/2$ and $0<\delta<1/3$. We choose $n\in \N$ such that $m$ as stated in \eqref{choos_m} is at least $2$.
Drawing the $n$ sampling nodes in $\bX$ uniformly i.i.d.\ at random from $\T^d$, we put the cubature weight vector $\boldsymbol{q}$ to be the first column of $\overline{\bL_m}\, \overline{(\bL_m^{\ast}\,\bL_m)^{-1}}.$
Then, with probability at least $1-3\delta$, we obtain
\begin{align*}
  \sup\limits_{\|f\|_{\square}\leq 1} |I(f)-\operatorname{Q}_\bX^mf|
  \leq \sqrt{\frac{29}{\delta}}\,\max\left\{\sigma_m^\square,\sqrt{\frac{\log(8n)}{n}\sum_{k=m}^\infty(\sigma_k^\square)^2}	
 \right\}\,.
\end{align*}
In particular, there is a constant $C_d^\square>0$ depending on $d$ such that for $0<\delta<1/3$ it holds with probability $1-3\delta$
\begin{equation}\label{equ:fourier:wc_int_error}
  \sup\limits_{\|f\|_{\square}\leq 1} |I(f)-\operatorname{Q}_\bX^mf|\\
  \leq \frac{C_d^\square}{\sqrt{\delta}}n^{-s}(\log n)^{sd}.
\end{equation}
\end{corollary}

\begin{proof}
We apply Theorem~\ref{thm:int_error} with $\mu_{\tor^d} = \varrho_{\tor^d} \equiv 1$
followed by \eqref{f27}.
\end{proof}

\begin{remark} By the same reasoning, the result in Theorem \ref{thm:preasymp_hsmix} transfers almost literally to the integration problem. In fact, having $s>(1+\log_2 d)/2$ we see a non-trivial preasymptotic behavior. The above bounds show that this method based on random points competes with most of the quasi-Monte-Carlo methods studied in the literature, see \cite[pp.\ 195, 247]{DiKuSl13}. 
\end{remark}

\section{Hyperbolic wavelet regression}
\label{HWR}
The following scenario of replacing the Fourier system by dyadic wavelets has already been investigated by Bohn \cite[Sec.~5.5.2]{Bo17}, \cite{Bo18_1} using piecewise linear prewavelets. Here we use orthogonal wavelets and will improve the result in \cite{Bo17} in two directions. First, we remove a $d$-dependent $\log$-factor and second, our result holds for the whole class and not just one individual function, i.e., we control the worst-case error. It is worth mentioning, that we only loose a $\log$-factor which is independent of $d$ compared to the benchmark result in \cite{DeVKoTe98}.

\begin{sloppypar}
Let us start with the necessary definitions since we are now in a non-periodic setting. For $s>0$ let us define the space $H^s_{\textnormal{mix}}(\R^d)$ as the collection of all functions from $L_2(\R^d)$ such that 
$$
	\|f\|_{H^s_{\textnormal{mix}}(\R^d)} = \Bigg\|\Bigg(\prod\limits_{i=1}^d (1+|y_i|)^s\Bigg) \mathcal{F}f(\by)\Bigg\|_{L_2(\R^d)} < \infty\,. 
$$
Here $\mathcal{F}$ denotes the Fourier transform on $\R^d$ given by 
$$
  \mathcal{F}f(\bx) = \frac{1}{\sqrt{2\pi}^d}\int_{\R^d} f(\by)\exp(-\rm{i} \, \by \cdot\bx)\,\mathrm{d}\by\,,\quad \bx \in \R^d\,.
$$
It is well-known, that $H^s_{\textnormal{mix}}(\R^d)$ can be characterized 
using hyperbolic wavelets. Let $(\psi_{j,k})_{j\in \N_0, k\in \Z}$ be a univariate orthonormal wavelet system (if $j=0$ then $\psi_{0,k}$ denotes the orthogonal scaling function). Then we denote with
$$
	\psi_{\bj,\bk}(\bx) := \psi_{i_1,k_1}(x_1)\cdots\psi_{j_d,k_d}(x_d)\,,\quad \bx \in \R^d, \; \bj\in \N_0^d, \; \bk \in \Z^d\,,
$$
the corresponding hyperbolic wavelet basis in $L_2(\R^d)$\,. For our analysis we need that the univariate wavelet is a compactly supported wavelet, which means that $\psi_{j,k}$ is supported ``near'' the interval $[k2^{-j}, (k+1)2^{-j}]$. If the wavelet basis has sufficient smoothness and vanishing moments, then $f \in H^s_{\textnormal{mix}}(\R^d)$ holds if and only if 
$$
	\Bigg(\sum\limits_{\bj \in \N_0^d}\sum\limits_{\bk \in \Z^d}2^{2\|\bj\|_1s}|\langle f,\psi_{\bj,\bk}\rangle|^2\Bigg)^{1/2} <\infty\,. 
$$
This leads to the norm equivalence
\begin{equation}\label{norm_equ}
	\|f\|_{H^s_{\textnormal{mix}}(\R^d)} \asymp  	\left(\sum\limits_{\bj \in \N_0^d}\sum\limits_{\bk \in \Z^d}2^{2\|\bj\|_1s}|\langle f,\psi_{\bj,\bk}\rangle|^2\right)^{1/2}\,.
\end{equation}
Clearly, if $\|f\|_{\Hsm(\R^d)} \leq 1$, then the sequence $(2^{\|\bj\|_1s}\langle f,\psi_{\bj,\bk}\rangle)_{\bj,\bk}$ has an $\ell_2$-norm bounded by a constant, which will be important for our later analysis. 
\end{sloppypar}

Let us consider the unit cube $[0,1]^d$. Let further $D_{\bj}$ be the set of all $\bk \in \Z^d$ such that the wavelet $\supp \psi_{\bj,\bk}$ has a non-empty intersection with $[0,1]^d$. This directly leads to the extended domain $\Omega$ given by 
$$
	\Omega:=\bigcup\limits_{\bj \in \N_0^d} \bigcup\limits_{\bk \in D_{\bj}} \supp \psi_{\bj,\bk}\,.
$$
It holds $[0,1]^d \subset \Omega$, and the system $(\psi_{\bj,\bk})_{\bj \in \N_0^d,\bk \in D_{\bj}}$ is an orthonormal system in $L_2(\Omega)$, however not a basis. Note that $\Omega$ is still a bounded tensor domain with a measure proportional to~1 depending on the support length of the wavelet basis. It is also clear that this orthonormal system is not uniformly bounded in $L_\infty$.

In the sequel we want to recover functions $f\in H^s_{\textnormal{mix}}(\R^d)$ on the domain $[0,1]^d$ from samples on the slightly larger extended domain $\Omega$ in a uniform way. In other words, the discrete locations of the sampling nodes
$\bX = \{\bx^1,...,\bx^n\}$ are chosen in advance for the whole class of functions.
Let us consider the operator 
$$
	\tilde{P}_\ell f := \sum\limits_{\|\bj\|_1 \leq \ell}\sum\limits_{\bk \in D_{\bj}}\langle f,\psi_{\bj,\bk}\rangle\psi_{\bj,\bk}\,,\quad \ell\in \N\,,
$$
which is known from {\it hyperbolic wavelet approximation}, see \cite{DeVKoTe98,SU09}. The following worst-case error bound is well-known and follows directly from \eqref{norm_equ}:
\begin{equation*}
	\sup\limits_{\|f\|_{\Hsm(\R^d)} \leq 1} \|f-\tilde{P}_\ell f\|_{L_2([0,1]^d)} \asymp 2^{-s\ell}\,.
\end{equation*}
We now consider a special case of the matrix $\bL_m$ from \eqref{f0}, namely 
\begin{equation}\label{f0_wav}
	\btL_\ell := \left(\begin{array}{c}
			(\widetilde{\psi}_{\bj,\bk}(\bx^1))^\top_{\|\bj\|_1\leq \ell,\bk \in D_{\bj}}\\
			\vdots \\
			(\widetilde{\psi}_{\bj,\bk}(\bx^n))^\top_{\|\bj\|_1\leq \ell,\bk \in D_{\bj}}
	\end{array}\right)\,.
\end{equation}
Here, $m = m(\ell) \asymp 2^\ell \ell^{d-1}$ and the functions $\widetilde{\psi}_{\bj,\bk} = \sqrt{|\Omega|}\psi_{\bj,\bk}$ enumerate the properly re-normalized wavelets $\psi_{\bj,\bk}, \|\bj\|_1 \leq \ell, \bk \in D_{\bj}$, which is now an orthonormal system in the space $L_2(\Omega,\varrho_\Omega)$ with the probability measure $\varrho_\Omega = \frac{\mathrm{d}\bx}{|\Omega|}$. The $n$ sampling nodes in $\bX$ i.i.d.\ at random according to $\varrho_\Omega$. Note that, due to the construction, we have that $|\Omega|$ is bounded by a constant which depends on the chosen wavelet system. This, on the other hand,   depends on the assumed mixed regularity properties of the function $f$, i.e., the mixed smoothness $s>0$. The larger $s$ is chosen, the larger the support of a properly chosen orthonormal wavelet system has to be. 
We propose Algorithm~\ref{algo2} for computing the wavelet coefficients of an approximation $S_\bX^\ell f$ to $f$. Note, that there is a little abuse of notation since we use the wavelet level $\ell$ as upper index (in contrast to $m$, the dimension of the subspace used earlier). 

\begin{algorithm}[tb]
\caption{Hyperbolic wavelet regression.}\label{algo2}
  \begin{tabular}{p{1.2cm}p{5.0cm}p{8.1cm}}
    Input: &$\ell \in \N$,\\ 
     &$n \colon \tilde{N}(\ell) \asymp 2^\ell \ell^{d-1} \lesssim \frac{n}{\log(n)}$,\\
	 & $\bX = (\bx^1,...,\bx^n)\in D^n$ \hfill & set of distinct sampling nodes, \\
      & $\mathbf{f} = (f(\bx^1),...,f(\bx^n))^\top$ \hfill & samples of $f$ evaluated at the nodes from $\bX$, \\
      & & such that the matrix $\btL_\ell := \btL_\ell(\bX)$ from~\eqref{f0_wav} has full (column) rank.
  \end{tabular}
  \begin{algorithmic}
  \STATE
  Solve the over-determined linear system 
  $$
	\btL_\ell \, (c_{\bj,\bk})_{\bj,\bk} = \mathbf{f}\,
  $$
  via least squares (e.g.\ directly or via the LSQR algorithm \cite{PaSa82}), i.e., compute
  $$
	 (c_{\bj,\bk})_{\bj,\bk} := (\btL_\ell^{\ast}\,\btL_\ell)^{-1}\,\btL_\ell^{\ast}\, \mathbf{f}.
  $$
  \end{algorithmic}
   Output:  $\mathbf{c} = (c_{\bj,\bk})_{\bj,\bk}\in \C^{m(\ell)}$ coefficients of the approximant $S^\ell_{\bX}f:=\sum\limits_{\|\bj\|_1 \leq \ell}\sum\limits_{\bk \in D_{\bj}} c_{\bj,\bk} \, \widetilde{\psi}_{\bj,\bk}.$
\end{algorithm}

\begin{theorem}\label{HWRT} Let $0<\delta<1$. Let further $s>1/2$ and $(\psi_{\bj,\bk})_{\bj,\bk}$ be a hyperbolic and compactly supported orthonormal wavelet system such that \eqref{norm_equ} holds true. Then the algorithm $S^\ell_{\bX}$ indicated in Algorithm~\ref{algo2} recovers any 
$f \in \Hsm(\R^d)$ on $L_2([0,1]^d)$ with high probability from $n = n(\ell)$ random samples which are drawn in advance for the whole class. Precisely, 
\begin{equation}\label{f47}
	\sup\limits_{\|f\|_{\Hsm(\R^d)}\leq 1} \|f-S^{\ell}_{\bX} f\|_{L_2([0,1]^d)} \lesssim C_{\delta,d}2^{-\ell s}
\end{equation}
with probability larger than $1-\delta$. The operator $S^{\ell}_{\bX}$ uses $n(\ell) \asymp 2^{\ell}\ell^d$ many samples such that the bound in \eqref{f47} reads as $\widetilde{C}_{\delta,d}n^{-s}\log(n)^{ds}$ in terms of the number of samples. 
\end{theorem}

\begin{remark}\label{sparseW} {\em (i)} Note that the optimal operator $\tilde{P}_\ell$ uses $n(\ell) \asymp 2^{\ell}\ell^{d-1}$ wavelet coefficients. The gap between sampling recovery ($\Lambda^{\text{std}}$) and general linear approximation ($\Lambda^{\text{all}}$), see e.g.~\cite{DuTeUl19}, \cite{NoWoI, NoWoII, NoWoIII}, is reduced to a $\log$-factor, which is independent of $d$. \\
{\em (ii)} The matrix defined in \eqref{f0_wav} is rather sparse. It has $n \asymp 2^{\ell}\ell^{d}$ rows and $m \asymp 2^{\ell}\ell^ {d-1}$ columns. In every row we have only $\asymp \#\{\|\bj\|_1\leq \ell\} \asymp \ell^d$ many non-zero entries. This gives an additional acceleration for the least squares algorithm since matrix vector multiplications are cheap in this situation. 
\end{remark}

\begin{proof}[Proof of Theorem \ref{HWRT}] We follow the proof of Theorem \ref{thm9}. Let $\bX$ be a set of $n$ randomly drawn nodes from $\Omega$ according to $\varrho_\Omega$. If the number $n$ of samples satisfies \eqref{f1}, that is 
\begin{equation}\label{f12b}
	\tilde{N}(\ell):= \sup_{\bx \in \Omega} \sum\limits_{\|\bj\|_1 \leq \ell}\sum\limits_{\bk \in D_{\bj}} \Big|\widetilde{\psi}_{\bj,\bk}(\bx)\Big|^2 \lesssim \frac{n}{\log n -\log\delta} \,,
\end{equation}
then 
$$
	\Big\|(\btL_\ell^{\ast}\,\btL_\ell)^{-1}\,\btL_\ell^{\ast}\Big\| \leq \sqrt{\frac{2}{n}}
$$
is satisfied with probability larger than $1-\delta$\,. Let $\bX$ be such that this is the case. Then we estimate
\begin{equation*}
  \begin{split}
		&\|f-S_{\bX}^{\ell}f\|_{L_2([0,1]^d)} \leq \|f-\tilde{P}_{\ell}f\|_{L_2([0,1]^d)} + \|\tilde{P}_{\ell}f - S_{\bX}^{\ell}f\|_{L_2([0,1]^d)}\\
		&\quad\lesssim \Bigg\|f-\sum\limits_{\|\bj\|_1\leq \ell}\sum\limits_{\bk \in \Z^d} \langle f,\psi_{\bj,\bk}\rangle \psi_{\bj,\bk}\Bigg\|_{L_2(\R^d)}+\|S_{\bX}^{\ell}(\tilde{P}_{\ell}f - f)\|_{L_2(\Omega, \varrho_{\Omega})}\\
	&\quad\lesssim  2^{-\ell s} + \sqrt{\frac{2}{n}}\cdot \left[\sum\limits_{u=1}^n
	\Bigg(\sum\limits_{\|\bj\|_1> \ell}\sum\limits_{\bk \in \Z^d}\overline{\langle f,\psi_{\bj,\bk}\rangle \psi_{\bj,\bk}(\bx^u)}\Bigg) \Bigg(
	\sum\limits_{\|\bj\|_1> \ell}\sum\limits_{\bk \in \Z^d}\langle f,\psi_{\bj,\bk}\rangle \psi_{\bj,\bk}(\bx^u)\Bigg)\right]^{1/2} \\
	&\quad\lesssim  2^{-\ell s} + \sqrt{\frac{2}{n}}\cdot 
	\left[\sum\limits_{\|\bj'\|_1> \ell}\sum\limits_{\bk' \in \Z^d}\sum\limits_{\|\bj\|_1> \ell}\sum\limits_{\bk \in \Z^d}\right.\\&\hspace{3cm}\left.2^{\|\bj'\|_1s}2^{\|\bj\|_1s}\overline{\langle f,\psi_{\bj',\bk'}\rangle}\langle f,\psi_{\bj,\bk}\rangle
	\sum\limits_{u=1}^n \frac{\overline{\psi_{\bj',\bk'}(\bx^u)}}{2^{\|\bj'\|_1s}}\frac{\psi_{\bj,\bk}(\bx^u)}{2^{\|\bj\|_1s}}\right]^{1/2} \\
	&\quad\lesssim 2^{-\ell s} + \sqrt{\frac{2}{n}} \big\|(2^{\|\bj\|_1s}\langle f,\psi_{\bj,\bk}\rangle)_{\bj,\bk}\big\|_2 \; \|\mathbf{\tilde{\Phi}}_\ell\|\\
	&\quad\lesssim 2^{-\ell s} + \sqrt{\frac{2}{n}}\|\mathbf{\tilde{\Phi}}_\ell\|\,,
   \end{split}
\end{equation*}
where $\mathbf{\tilde{\Phi}}_\ell$ is defined similar as in Proposition \ref{cor8}. This time we put 
$$
	\mathbf{\tilde{\Phi}}_\ell := \left(\begin{array}{c}
			(2^{-\|\bj\|_1s}\psi_{\bj,\bk}(\bx^1))^\top_{\|\bj\|_1> \ell,\bk \in D_{\bj}}\\
			\vdots \\
			(2^{-\|\bj\|_1s}\psi_{\bj,\bk}(\bx^n))^\top_{\|\bj\|_1> \ell,\bk \in D_{\bj}}
	\end{array}\right)\,.
$$
Let us define the quantity $$\tilde{T}(\ell) := \sup\limits_{\bx \in \Omega} \sum\limits_{\|\bj\|_1 >\ell}\sum\limits_{\bk \in \Z^d} 2^{-2\|\bj\|_1s}|\psi_{\bj,\bk}(\bx)|^2\,,$$ which goes along the lines of Proposition \ref{cor8}. Then we get with literally the same arguments
\begin{equation}\label{f48}
  \Ept \|\mathbf{\tilde{\Phi}}_{\ell}\|^2 \lesssim n\left(2^{-2\ell s} + \frac{\log n}{n}\tilde{T}(\ell)+
  2^{-\ell s}\sqrt{\frac{\log n}{n}\tilde{T}(\ell)}\right)\,.
\end{equation}
Let us compute $\tilde{T}(\ell)$. Due to the compact support of the wavelet system, there are for fixed~$\bj$ only $O(1)$ many wavelets $\psi_{\bj,\bk}$ such that $\psi_{\bj,\bk}(\bx)$ is non-zero. For those $O(1)$ wavelets, we have $$
	|\psi_{\bj,\bk}(\bx)|^2 \lesssim 2^{\|\bj\|_1}\,.
$$
Hence, we get
\begin{equation}\label{f49}
	\tilde{T}(\ell) \lesssim \sum\limits_{\|\bj\|_1>\ell} 2^{\|\bj\|_1(1-2s)}\lesssim 2^{\ell(1-2s)}\ell^{d-1}\,.
\end{equation}
By the same reasoning we may estimate $\tilde{N}(\ell)$ in \eqref{f12b}. Clearly $n$ may be chosen such that 
\begin{equation}\label{f50}
	\tilde{N}(\ell) \asymp 2^{\ell}\ell^{d-1} \lesssim \frac{n}{\log(n)}\,.
\end{equation}
Plugging \eqref{f49} and \eqref{f50} into \eqref{f48} we obtain 
$$
	\Ept \|\mathbf{\tilde{\Phi}_\ell}\|^2 \lesssim n2^{-2\ell s}\,.
$$
The same standard arguments as used in Theorem \ref{thm9} and Corollary \ref{cor10} lead to the bound in \eqref{f47}. It remains to estimate the number of samples $n$ depending on $\ell$, see \eqref{f50}. This clearly gives 
$\log(n) \gtrsim \ell$ and hence $n\gtrsim 2^\ell \ell^d$ which concludes the proof. 
\end{proof}

\section{Numerical experiments}\label{sec:num}

\subsection{Recovery of functions from spaces with mixed smoothness}
\label{sec:num:fourier}

In this section, we perform numerical tests for the hyperbolic cross Fourier regression based on random sampling nodes from Section~\ref{mixed1}, i.e., we apply Algorithm~\ref{algo1} to periodic test functions~$f$ from the spaces $H^s_{\textnormal{mix}}(\tor^d)$.
In Figure~\ref{fig:rand_nodes_fourier}, we visualize realizations for such random nodes in the two- and three-dimensional case.

Besides random point sets, different types of deterministic lattices have also been used for numerical integration and function recovery, see for instance \cite{KaPoVo13,KaPoVo14} and \cite{ByKaUlVo16}. This motivates us to consider Frolov lattices~\cite{KaOeUlUl18} and Fibonacci lattices (cf.\ e.g.\ \cite[Sec.~IV.2]{Tem93}) in the context of this paper, see Figure~\ref{fig:fib_frolov_nodes_fourier} for examples of such lattices.

In the following, we use the weight function $$w(\bk):=\prod_{i=1}^d (1+|k_i|^2)^{1/2}.$$ Note that for computational reasons we avoid the $2\pi$ in this weight~$w$. By the reasoning after~\eqref{eq:sigma_upper_bound_kuehn}, the weights without $2\pi$ lead to a slightly slower decay of the respective singular numbers.

\begin{figure}[htb]
\centering
	\subfloat[$d=2$, $n=100$]{
\begin{tikzpicture}
\begin{axis}[
plot box ratio = 1 1,
xmin=0,xmax=1,ymin=0,ymax=1,
font=\footnotesize,
height=0.38\textwidth,
width=0.38\textwidth,
]
\addplot[only marks,black,mark=*,mark size=1pt,mark options={solid}] coordinates {
 (0.815,0.162) (0.906,0.794) (0.127,0.311) (0.913,0.529) (0.632,0.166) (0.098,0.602) (0.278,0.263) (0.547,0.654) (0.958,0.689) (0.965,0.748) (0.158,0.451) (0.971,0.084) (0.957,0.229) (0.485,0.913) (0.800,0.152) (0.142,0.826) (0.422,0.538) (0.916,0.996) (0.792,0.078) (0.959,0.443) (0.656,0.107) (0.036,0.962) (0.849,0.005) (0.934,0.775) (0.679,0.817) (0.758,0.869) (0.743,0.084) (0.392,0.400) (0.655,0.260) (0.171,0.800) (0.706,0.431) (0.032,0.911) (0.277,0.182) (0.046,0.264) (0.097,0.146) (0.823,0.136) (0.695,0.869) (0.317,0.580) (0.950,0.550) (0.034,0.145) (0.439,0.853) (0.382,0.622) (0.766,0.351) (0.795,0.513) (0.187,0.402) (0.490,0.076) (0.446,0.240) (0.646,0.123) (0.709,0.184) (0.755,0.240) (0.276,0.417) (0.680,0.050) (0.655,0.903) (0.163,0.945) (0.119,0.491) (0.498,0.489) (0.960,0.338) (0.340,0.900) (0.585,0.369) (0.224,0.111) (0.751,0.780) (0.255,0.390) (0.506,0.242) (0.699,0.404) (0.891,0.096) (0.959,0.132) (0.547,0.942) (0.139,0.956) (0.149,0.575) (0.258,0.060) (0.841,0.235) (0.254,0.353) (0.814,0.821) (0.244,0.015) (0.929,0.043) (0.350,0.169) (0.197,0.649) (0.251,0.732) (0.616,0.648) (0.473,0.451) (0.352,0.547) (0.831,0.296) (0.585,0.745) (0.550,0.189) (0.917,0.687) (0.286,0.184) (0.757,0.368) (0.754,0.626) (0.380,0.780) (0.568,0.081) (0.076,0.929) (0.054,0.776) (0.531,0.487) (0.779,0.436) (0.934,0.447) (0.130,0.306) (0.569,0.509) (0.469,0.511) (0.012,0.818) (0.337,0.795)
};
\end{axis}
\end{tikzpicture}
}
\hspace{5em}
\subfloat[$d=3$, $n=316$]{
\begin{tikzpicture}
\begin{axis}[axis background/.style={fill=white},
every axis/.append style={font=\footnotesize},
width=0.40\textwidth,
height=0.40\textwidth,
enlargelimits=false,
axis lines = none,
enlargelimits=false,
clip=false,
view={42}{25},
grid=major,
plot box ratio = 1 1 1,
clip mode=individual,
tickwidth=0pt,
z buffer=sort,
xmin=0,xmax=1,
ymin=0,ymax=1,
zmin=0, zmax=1,
]
\addplot3[black!30,very thick,-] coordinates {(0,1,1) (0,1,0)};
\addplot3[black!30,very thick,-] coordinates {(0,1,0) (0,0,0)};
\addplot3[black!30,very thick,-] coordinates {(0,1,0) (1,1,0)};
\addplot3[black,ultra thick,-] coordinates {(0,0,1) (0,0,0)};
\addplot3[black,ultra thick,-] coordinates {(0,0,0) (1,0,0)};
\addplot3+[only marks, mark size=1pt, mark=ball, solid, ball color=black!75, mark options={black!75, draw=black}] coordinates{
 (0.249,0.227,0.729) (0.386,0.049,0.938) (0.431,0.169,0.517) (0.831,0.258,0.903) (0.825,0.198,0.218) (0.453,0.606,0.873) (0.381,0.824,0.083) (0.926,0.811,0.465) (0.741,0.802,0.022) (0.738,0.708,0.808) (0.947,0.859,0.179) (0.510,0.781,0.165) (0.792,0.204,0.182) (0.452,0.993,0.691) (0.849,0.094,0.214) (0.390,0.651,0.298) (0.738,0.215,0.768) (0.976,0.244,0.501) (0.523,0.340,0.909) (0.430,0.198,0.058) (0.207,0.507,0.437) (0.323,0.951,0.572) (0.111,0.395,0.565) (0.375,0.584,0.824) (0.330,0.607,0.126) (0.342,0.715,0.300) (0.817,0.402,0.002) (0.532,0.859,0.951) (0.521,0.920,0.766) (0.774,0.751,0.751) (0.120,0.286,0.139) (0.625,0.797,0.349) (0.347,0.143,0.151) (0.335,0.505,0.497) (0.575,0.611,0.809) (0.864,0.704,0.633) (0.199,0.383,0.688) (0.672,0.729,0.640) (0.902,0.887,0.729) (0.199,0.056,0.860) (0.298,0.138,0.627) (0.497,0.863,0.181) (0.890,0.422,0.573) (0.501,0.411,0.164) (0.277,0.959,0.906) (0.534,0.750,0.077) (0.574,0.981,0.339) (0.413,0.234,0.581) (0.015,0.096,0.475) (0.703,0.385,0.805) (0.507,0.500,0.531) (0.381,0.570,0.227) (0.065,0.977,0.709) (0.359,0.493,0.149) (0.234,0.401,0.658) (0.204,0.995,0.634) (0.814,0.261,0.229) (0.393,0.665,0.182) (0.054,0.964,0.166) (0.375,0.671,0.150) (0.775,0.299,0.203) (0.165,0.531,0.955) (0.912,0.001,0.016) (0.319,0.884,0.958) (0.330,0.404,0.026) (0.204,0.301,0.971) (0.767,0.951,0.298) (0.070,0.461,0.525) (0.950,0.288,0.862) (0.158,0.085,0.896) (0.286,0.582,0.189) (0.687,0.153,0.661) (0.141,0.073,0.941) (0.512,0.581,0.976) (0.721,0.287,0.108) (0.929,0.362,0.179) (0.732,0.725,0.747) (0.750,0.858,0.049) (0.407,0.348,0.071) (0.239,0.962,0.489) (0.521,0.954,0.850) (0.219,0.206,0.997) (0.842,0.768,0.004) (0.663,0.616,0.543) (0.816,0.919,0.861) (0.794,0.603,0.909) (0.469,0.702,0.845) (0.310,0.744,0.879) (0.688,0.385,0.746) (0.987,0.252,0.117) (0.770,0.037,0.509) (0.830,0.472,0.169) (0.706,0.645,0.831) (0.595,0.279,0.928) (0.753,0.518,0.169) (0.497,0.246,0.884) (0.865,0.298,0.388) (0.068,0.650,0.383) (0.969,0.891,0.271) (0.099,0.861,0.868) (0.547,0.210,0.742) (0.403,0.399,0.448) (0.107,0.888,0.710) (0.724,0.257,0.944) (0.614,0.967,0.174) (0.783,0.619,0.245) (0.567,0.165,0.641) (0.811,0.826,0.809) (0.577,0.656,0.853) (0.944,0.546,0.398) (0.871,0.251,0.115) (0.508,0.040,0.080) (0.789,0.233,0.360) (0.473,0.361,0.829) (0.829,0.633,0.215) (0.322,0.986,0.791) (0.976,0.207,0.655) (0.278,0.757,0.026) (0.073,0.886,0.786) (0.751,0.472,0.923) (0.831,0.159,0.492) (0.922,0.811,0.834) (0.327,0.477,0.131) (0.804,0.116,0.760) (0.538,0.876,0.926) (0.463,0.635,0.833) (0.821,0.097,0.259) (0.952,0.908,0.213) (0.076,0.035,0.522) (0.709,0.040,0.397) (0.235,0.989,0.479) (0.399,0.686,0.994) (0.268,0.377,0.604) (0.833,0.504,0.945) (0.995,0.763,0.490) (0.650,0.049,0.438) (0.704,0.726,0.773) (0.932,0.701,0.744) (0.688,0.459,0.443) (0.568,0.582,0.053) (0.381,0.339,0.088) (0.635,0.171,0.798) (0.363,0.399,0.656) (0.408,0.920,0.032) (0.369,0.226,0.557) (0.468,0.361,0.720) (0.503,0.325,0.110) (0.911,0.084,0.217) (0.206,0.513,0.811) (0.339,0.833,0.139) (0.574,0.905,0.882) (0.487,0.724,0.924) (0.262,0.383,0.013) (0.580,0.298,0.377) (0.878,0.692,0.168) (0.061,0.880,0.540) (0.441,0.925,0.102) (0.084,0.081,0.039) (0.563,0.483,0.933) (0.539,0.128,0.972) (0.768,0.253,0.361) (0.233,0.884,0.644) (0.587,0.196,0.068) (0.459,0.121,0.208) (0.861,0.544,0.040) (0.661,0.315,0.469) (0.354,0.382,0.150) (0.347,0.792,0.991) (0.254,0.839,0.427) (0.953,0.680,0.955) (0.298,0.417,0.724) (0.158,0.643,0.581) (0.361,0.214,0.540) (0.742,0.617,0.705) (0.706,0.675,0.005) (0.701,0.601,0.783) (0.006,0.346,0.927) (0.374,0.364,0.008) (0.901,0.171,0.825) (0.318,0.795,0.767) (0.597,0.493,0.997) (0.298,0.355,0.228) (0.125,0.775,0.920) (0.388,0.237,0.642) (0.818,0.845,0.105) (0.981,0.817,0.268) (0.862,0.846,0.764) (0.084,0.370,0.806) (0.338,0.383,0.104) (0.236,0.861,0.470) (0.318,0.464,0.219) (0.984,0.571,0.923) (0.548,0.695,0.320) (0.749,0.961,0.858) (0.842,0.546,0.260) (0.167,0.637,0.878) (0.903,0.571,0.188) (0.105,0.927,0.759) (0.745,0.864,0.032) (0.729,0.170,0.642) (0.717,0.179,0.567) (0.133,0.244,0.376) (0.446,0.752,0.213) (0.509,0.199,0.792) (0.530,0.983,0.145) (0.860,0.710,0.489) (0.678,0.175,0.013) (0.806,0.858,0.187) (0.531,0.909,0.485) (0.956,0.962,0.838) (0.067,0.571,0.141) (0.542,0.563,0.732) (0.282,0.177,0.691) (0.481,0.514,0.034) (0.685,0.548,0.489) (0.208,0.165,0.971) (0.608,0.494,0.112) (0.326,0.535,0.743) (0.881,0.199,0.639) (0.133,0.623,0.594) (0.102,0.026,0.499) (0.959,0.319,0.568) (0.153,0.533,0.427) (0.153,0.327,0.076) (0.156,0.602,0.291) (0.090,0.362,0.561) (0.454,0.135,0.633) (0.669,0.914,0.931) (0.831,0.641,0.978) (0.790,0.659,0.094) (0.713,0.675,0.662) (0.473,0.745,0.603) (0.709,0.842,0.474) (0.958,0.517,0.356) (0.506,0.152,0.476) (0.305,0.381,0.671) (0.790,0.821,0.960) (0.236,0.171,0.089) (0.234,0.330,0.798) (0.465,0.966,0.591) (0.619,0.806,0.912) (0.615,0.222,0.101) (0.123,1.000,0.293) (0.124,0.064,0.052) (0.284,0.425,0.504) (0.736,0.404,0.768) (0.411,0.400,0.283) (0.829,0.112,0.225) (0.935,0.424,0.331) (0.399,0.614,0.453) (0.052,0.988,0.737) (0.571,0.220,0.510) (0.748,0.354,0.383) (0.320,0.266,0.905) (0.493,0.291,0.965) (0.222,0.188,0.628) (0.939,0.023,0.132) (0.482,0.449,0.618) (0.540,0.244,0.383) (0.221,0.869,0.991) (0.096,0.529,0.287) (0.060,0.914,0.706) (0.820,0.974,0.535) (0.771,0.585,0.193) (0.196,0.119,0.689) (0.895,0.927,0.050) (0.684,0.594,0.184) (0.657,0.884,0.046) (0.990,0.424,0.885) (0.034,0.607,0.840) (0.424,0.071,0.118) (0.490,0.925,0.410) (0.584,0.642,0.120) (0.083,0.104,0.572) (0.660,0.700,0.949) (0.052,0.396,0.256) (0.557,0.085,0.990) (0.712,0.214,0.350) (0.488,0.249,0.209) (0.618,0.227,0.666) (0.214,0.703,0.973) (0.646,0.754,0.623) (0.381,0.547,0.064) (0.104,0.553,0.374) (0.378,0.631,0.166) (0.263,0.985,0.231) (0.241,0.634,0.052) (0.623,0.600,0.902) (0.523,0.909,0.793) (0.413,0.571,0.373) (0.218,0.335,0.832) (0.859,0.957,0.754) (0.861,0.440,0.622) (0.284,0.602,0.394) (0.615,0.720,0.359) (0.779,0.679,0.089) (0.955,0.213,0.342) (0.920,0.082,0.549) (0.385,0.274,0.461) (0.163,0.868,0.645) (0.797,0.559,0.514) (0.114,0.465,0.814) (0.159,0.430,0.097) (0.356,0.774,0.464) (0.848,0.654,0.590) (0.583,0.658,0.187) (0.586,0.161,0.611) (0.926,0.432,0.052) (0.575,0.505,0.576) (0.010,0.375,0.842) (0.809,0.480,0.500) (0.609,0.342,0.439) (0.480,0.777,0.149) (0.268,0.384,0.028) (0.258,0.712,0.757) (0.481,0.481,0.796)
};
\addplot3[black,ultra thick,-] coordinates {(0,0,1) (0,1,1)};
\addplot3[black,ultra thick,-] coordinates {(0,1,1) (1,1,1)};
\addplot3[black,ultra thick,-] coordinates {(1,1,1) (1,0,1)};
\addplot3[black,ultra thick,-] coordinates {(1,0,1) (0,0,1)};
\addplot3[black,ultra thick,-] coordinates {(1,0,1) (1,0,0)};
\addplot3[black,ultra thick,-] coordinates {(1,0,0) (1,1,0)};
\addplot3[black,ultra thick,-] coordinates {(1,1,0) (1,1,1)};
\end{axis}
\end{tikzpicture}
}
\caption{Realizations of random nodes for hyperbolic Fourier regression.}\label{fig:rand_nodes_fourier}
\end{figure}
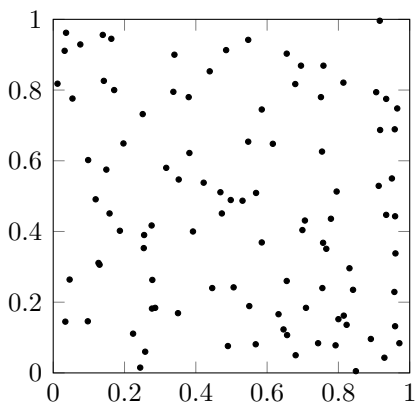
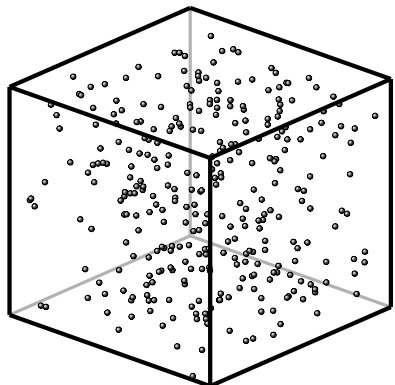

\begin{figure}[htb]
\centering
	\subfloat[Fibonacci lattice $n=144$]{
		\begin{tikzpicture}
		\begin{axis}[
		plot box ratio = 1 1,
		xmin=0,xmax=1,ymin=0,ymax=1,
		font=\footnotesize,
		height=0.38\textwidth,
		width=0.38\textwidth,
		]
		\addplot[only marks,black,mark=*,mark size=1pt,mark options={solid}] coordinates {
 (0.000,0.000) (0.007,0.618) (0.014,0.236) (0.021,0.854) (0.028,0.472) (0.035,0.090) (0.042,0.708) (0.049,0.326) (0.056,0.944) (0.062,0.562) (0.069,0.181) (0.076,0.799) (0.083,0.417) (0.090,0.035) (0.097,0.653) (0.104,0.271) (0.111,0.889) (0.118,0.507) (0.125,0.125) (0.132,0.743) (0.139,0.361) (0.146,0.979) (0.153,0.597) (0.160,0.215) (0.167,0.833) (0.174,0.451) (0.181,0.069) (0.188,0.688) (0.194,0.306) (0.201,0.924) (0.208,0.542) (0.215,0.160) (0.222,0.778) (0.229,0.396) (0.236,0.014) (0.243,0.632) (0.250,0.250) (0.257,0.868) (0.264,0.486) (0.271,0.104) (0.278,0.722) (0.285,0.340) (0.292,0.958) (0.299,0.576) (0.306,0.194) (0.312,0.812) (0.319,0.431) (0.326,0.049) (0.333,0.667) (0.340,0.285) (0.347,0.903) (0.354,0.521) (0.361,0.139) (0.368,0.757) (0.375,0.375) (0.382,0.993) (0.389,0.611) (0.396,0.229) (0.403,0.847) (0.410,0.465) (0.417,0.083) (0.424,0.701) (0.431,0.319) (0.438,0.938) (0.444,0.556) (0.451,0.174) (0.458,0.792) (0.465,0.410) (0.472,0.028) (0.479,0.646) (0.486,0.264) (0.493,0.882) (0.500,0.500) (0.507,0.118) (0.514,0.736) (0.521,0.354) (0.528,0.972) (0.535,0.590) (0.542,0.208) (0.549,0.826) (0.556,0.444) (0.562,0.062) (0.569,0.681) (0.576,0.299) (0.583,0.917) (0.590,0.535) (0.597,0.153) (0.604,0.771) (0.611,0.389) (0.618,0.007) (0.625,0.625) (0.632,0.243) (0.639,0.861) (0.646,0.479) (0.653,0.097) (0.660,0.715) (0.667,0.333) (0.674,0.951) (0.681,0.569) (0.688,0.188) (0.694,0.806) (0.701,0.424) (0.708,0.042) (0.715,0.660) (0.722,0.278) (0.729,0.896) (0.736,0.514) (0.743,0.132) (0.750,0.750) (0.757,0.368) (0.764,0.986) (0.771,0.604) (0.778,0.222) (0.785,0.840) (0.792,0.458) (0.799,0.076) (0.806,0.694) (0.812,0.312) (0.819,0.931) (0.826,0.549) (0.833,0.167) (0.840,0.785) (0.847,0.403) (0.854,0.021) (0.861,0.639) (0.868,0.257) (0.875,0.875) (0.882,0.493) (0.889,0.111) (0.896,0.729) (0.903,0.347) (0.910,0.965) (0.917,0.583) (0.924,0.201) (0.931,0.819) (0.938,0.438) (0.944,0.056) (0.951,0.674) (0.958,0.292) (0.965,0.910) (0.972,0.528) (0.979,0.146) (0.986,0.764) (0.993,0.382)
		};
		\end{axis}
		\end{tikzpicture}
	}
\hspace{5em}
	\subfloat[Frolov lattice $n=127$]{
	\begin{tikzpicture}
	\begin{axis}[
	plot box ratio = 1 1,
	xmin=0,xmax=1,ymin=0,ymax=1,
	font=\footnotesize,
	height=0.38\textwidth,
	width=0.38\textwidth,
	]
	\addplot[only marks,black,mark=*,mark size=1pt,mark options={solid}] coordinates {
 (0.500,0.500) (0.559,0.559) (0.618,0.618) (0.677,0.677) (0.736,0.736) (0.796,0.796) (0.855,0.855) (0.914,0.914) (0.973,0.973) (0.441,0.441) (0.382,0.382) (0.323,0.323) (0.264,0.264) (0.204,0.204) (0.145,0.145) (0.086,0.086) (0.027,0.027) (0.537,0.404) (0.596,0.463) (0.655,0.523) (0.714,0.582) (0.773,0.641) (0.832,0.700) (0.891,0.759) (0.950,0.818) (0.477,0.345) (0.418,0.286) (0.359,0.227) (0.300,0.168) (0.241,0.109) (0.182,0.050) (0.632,0.368) (0.691,0.427) (0.750,0.486) (0.809,0.545) (0.869,0.604) (0.928,0.663) (0.987,0.722) (0.573,0.309) (0.514,0.250) (0.455,0.191) (0.396,0.131) (0.337,0.072) (0.278,0.013) (0.669,0.272) (0.728,0.331) (0.787,0.390) (0.846,0.450) (0.905,0.509) (0.964,0.568) (0.610,0.213) (0.550,0.154) (0.491,0.095) (0.432,0.036) (0.764,0.236) (0.823,0.295) (0.883,0.354) (0.942,0.413) (0.705,0.177) (0.646,0.117) (0.587,0.058) (0.801,0.140) (0.860,0.199) (0.919,0.258) (0.978,0.317) (0.742,0.081) (0.683,0.022) (0.897,0.103) (0.956,0.163) (0.837,0.044) (0.933,0.008) (0.992,0.067) (0.463,0.596) (0.523,0.655) (0.582,0.714) (0.641,0.773) (0.700,0.832) (0.759,0.891) (0.818,0.950) (0.404,0.537) (0.345,0.477) (0.286,0.418) (0.227,0.359) (0.168,0.300) (0.109,0.241) (0.050,0.182) (0.368,0.632) (0.427,0.691) (0.486,0.750) (0.545,0.809) (0.604,0.869) (0.663,0.928) (0.722,0.987) (0.309,0.573) (0.250,0.514) (0.191,0.455) (0.131,0.396) (0.072,0.337) (0.013,0.278) (0.331,0.728) (0.390,0.787) (0.450,0.846) (0.509,0.905) (0.568,0.964) (0.272,0.669) (0.213,0.610) (0.154,0.550) (0.095,0.491) (0.036,0.432) (0.236,0.764) (0.295,0.823) (0.354,0.883) (0.413,0.942) (0.177,0.705) (0.117,0.646) (0.058,0.587) (0.199,0.860) (0.258,0.919) (0.317,0.978) (0.140,0.801) (0.081,0.742) (0.022,0.683) (0.103,0.897) (0.163,0.956) (0.044,0.837) (0.067,0.992) (0.008,0.933)
	};
	\end{axis}
	\end{tikzpicture}
}
\caption{Examples of Fibonacci and Frolov lattices in $d=2$ spatial dimensions.}\label{fig:fib_frolov_nodes_fourier}
\end{figure}
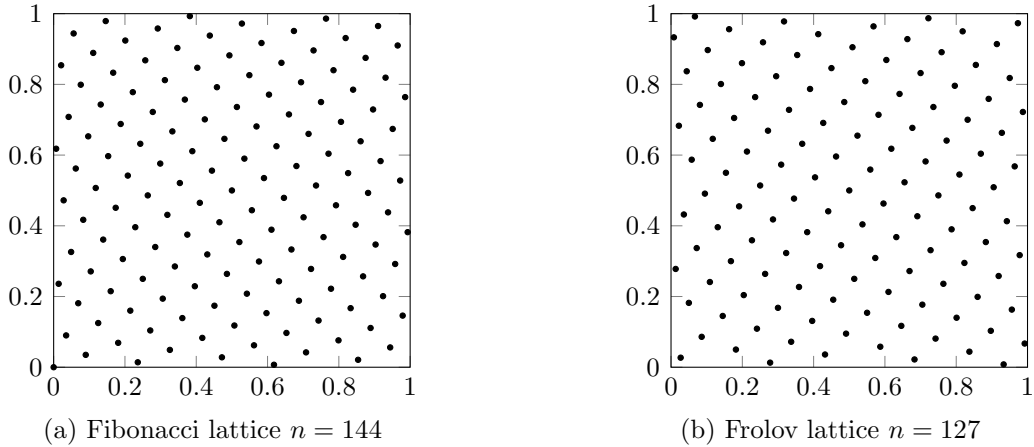

For a given number $n$ of samples, we use the $m=\lfloor n/(4\log n)\rfloor$ frequencies $\bk\in\Z^d$ where $w(\bk)$ is smallest, i.e., we define $I_m:=\{\bk_1,\ldots,\bk_m\}\subset\Z^d$, $|I_m|=m$, assuming an arrangement of $\bk_j$ fulfilling $w(\bk_1)\le w(\bk_2)\le\ldots$.  Here ties are broken in numerical order starting with the first component $k_1$ of $\bk$ until the last one $k_d$. Corresponding to our theoretical results, the goal is to compute a least squares approximation $S_\bX^mf$ of the projection $P_{m-1}f$ of the function $f$ to the $\operatorname{span}\{\exp(2\pi\mathrm{i}\,\bk\cdot \bx) \colon \bk \in I_m\}$ using the $n$ sampling values at the nodes in $\bX$.

\paragraph{Comments on the arithmetic cost of Algorithm \ref{algo1}.}
Building the index set $I_{m-1}$, i.e., enumerating the basis functions $\eta_1,\ldots,\eta_{m-1}$, requires
$$\leq 4 \, C_1 \, d \, m^2 \log m \leq C_1 \, d \, n \, m$$
arithmetic operations, and setting up the matrix $\bL_m$ requires
$$\leq C_2 \, d \, n \, m$$ arithmetic operations,
where $C_1,C_2>0$ are absolute constants.
Afterwards, running Algorithm~\ref{algo1} requires
$$\leq C_3\, R \, n \, m \leq C_3\, R \, \frac{n^2}{4\log n}$$
arithmetic operations, where $C_3>0$ is an absolute constant and $R\in\N$ is the number of LSQR iterations. If one chooses $m$ as in Theorem~\ref{thm9}, cf.~\eqref{f1} and~\eqref{f1c}, the condition number of the matrix $\bL_m$ is $\leq \sqrt{3}$ with high probability and one obtains $R\leq 17$ for a LSQR accuracy of $\approx 10^{-8}$. Please note that we choose $m=\lfloor n/(4\log n)\rfloor$ in our experiments, which is slightly larger than the theory \eqref{f1c} requires.

\begin{remark}
Let us compare the hyperbolic cross Fourier regression from Section~\ref{mixed1}, which uses random samples, with the single rank-1 lattice sampling approach from~\cite{KaPoVo13, ByKaUlVo16}, which uses highly structured deterministic sampling nodes.
Up to logarithmic factors, both approaches have comparable error estimates w.r.t.\ the number $m$ of basis functions and comparable arithmetic complexities.
Single rank-1 lattice sampling has slightly worse recovery error estimates w.r.t.\ $m$ than Algorithm~\ref{algo1}, cf.\ Theorem~\ref{thm:asymp_hsmix}. On the other hand, the arithmetic complexity for single rank-1 lattice sampling is slightly better. Moreover, the error estimates when using rank-1 lattices are guaranteed upper bounds, whereas the worst case upper bounds in Section~\ref{mixed1} hold with high probability. However, for fixed $m$, the used number of samples for single rank-1 lattice sampling is distinctly higher, i.e., almost quadratic compared to the approach from this paper.
This results in error estimates w.r.t.\ the number~$n$ of used sampling values that are distinctly worse for single rank-1 lattices.
\end{remark}
Subsequently, we consider three different test functions $f\colon \tor^d \rightarrow \R$,
where the Fourier coefficients $\hat{f}_{\bk}:=\int_{\tor^d}f(\bx) \,\exp(-2\pi\mathrm{i}\,\bk\cdot\bx) \,\mathrm{d}\bx$, $\bk\in\Z^d$, of $f$, decay like $|\hat{f}_{\bk}|\sim \prod_{i=1}^d \big(1+|k_i|^2\big)^{-\alpha/2}$ for $\alpha\in\{5/4,\, 2,\, 6\}$ and, consequently, $f\in H^s_{\textnormal{mix}}(\tor^d)$ with $s=\alpha-1/2-\varepsilon$ for $\varepsilon>0$.

\subsubsection*{Test function $f$ from $H^{3/4-\varepsilon}_{\textnormal{mix}}(\tor^d)$}
We start with the test function
$$
 f\colon \tor^d \rightarrow \R, \quad
 f(\bx) := \left(\frac{3}{2}\right)^{d/2} \prod_{i=1}^d \left(1 - \big|2 (x_i \bmod 1) - 1\big|\right)^{1/4},
$$
where we have for the Fourier coefficients $|\hat{f}_{\bk}|\sim \big(\prod_{i=1}^d (1+|k_i|^2)^{1/2}\big)^{-5/4}$ and, consequently, $f\in H^{3/4-\varepsilon}_{\textnormal{mix}}(\tor^d)$, $\varepsilon>0$.

\begin{figure}[htb]
\subfloat[approximation error]{\label{fig:num:fourier:H0.75:trunc}
		\begin{tikzpicture}[baseline]
		\begin{axis}[
			font=\footnotesize,
			enlarge x limits=true,
			enlarge y limits=true,
			height=0.4\textwidth,
			grid=major,
			width=0.48\textwidth,
            xmode=log,
            ymode=log,
			xlabel={$m=|I|$},
			ylabel={$L_2$ approximation error},
			legend style={legend cell align=left, at={(1.0,1.35)}},
			legend columns = 2,
		]
		\addplot[dashed,red,mark=triangle,mark size=3pt,mark options={solid}] coordinates {
 (100,3.118e-02) (316,1.646e-02) (1000,8.278e-03) (3162,4.064e-03) (10000,1.952e-03) (31623,9.230e-04) (100000,4.311e-04) (316228,1.993e-04) (1000000,9.138e-05) (3162278,4.161e-05)
};
\addlegendentry{$\tilde{a}_m$, $d=2$}
\addplot [forget plot,black,domain=1e3:6e6, samples=100, dotted,ultra thick]{0.6*x^(-0.75)*(ln(x))^(1*0.75)};
		\addplot[dashed,blue,mark=square,mark size=2.2pt,mark options={solid}] coordinates {
 (100,6.842e-02) (316,4.241e-02) (1000,2.422e-02) (3162,1.389e-02) (10000,7.665e-03) (31623,4.063e-03) (100000,2.115e-03) (316228,1.079e-03) (1000000,5.413e-04) (3162278,2.677e-04) (10000000,1.309e-04)
};
\addlegendentry{$\tilde{a}_m$, $d=3$}
\addplot [forget plot,black,domain=1e5:2e7, samples=100, dotted,ultra thick]{0.55*x^(-0.75)*(ln(x))^(2*0.75)};
		\addplot[dashed,magenta,mark=o,mark size=2.5pt,mark options={solid}] coordinates {
 (100,1.139e-01) (316,7.896e-02) (1000,4.602e-02) (3162,3.061e-02) (10000,1.893e-02) (31623,1.097e-02) (100000,6.260e-03) (316228,3.502e-03) (1000000,1.923e-03) (3162278,1.030e-03) (10000000,5.438e-04) (31622777,2.827e-04)
};
\addlegendentry{$\tilde{a}_m$, $d=4$}
\addplot [forget plot,black,domain=1e6:8e7, samples=100, dotted,ultra thick]{0.3*x^(-0.75)*(ln(x))^(3*0.75)};
		\addplot[dashed,cyan,mark=triangle,mark size=3pt,mark options={solid,rotate=180}] coordinates {
 (100,1.419e-01) (316,1.008e-01) (1000,6.991e-02) (3162,5.163e-02) (10000,3.335e-02) (31623,2.170e-02) (100000,1.365e-02) (316228,8.258e-03) (1000000,4.872e-03) (3162278,2.822e-03) (10000000,1.597e-03) (31622777,8.901e-04)
};
\addlegendentry{$\tilde{a}_m$, $d=5$}
\addplot [black,domain=2e6:8e7, samples=100, dotted,ultra thick]{0.12*x^(-0.75)*(ln(x))^(4*0.75)};
\addlegendentry{$m^{-0.75}(\log m)^{(d-1)\cdot 0.75}$}
		\end{axis}
		\end{tikzpicture}
}
\hfill
\subfloat[sampling error]{\label{fig:num:fourier:H0.75:sampl}
		\begin{tikzpicture}[baseline]
		\begin{axis}[
			font=\footnotesize,
			enlarge x limits=true,
			enlarge y limits=true,
			height=0.4\textwidth,
			grid=major,
			width=0.48\textwidth,
            xmode=log,
            ymode=log,
			xlabel={$n$},
			ylabel={$L_2$ error},
			legend style={legend cell align=left, at={(1.0,1.5)}},
			legend columns = 2,
		]
		\addplot[red,mark=triangle,mark size=3pt,mark options={solid}] coordinates {
 (100,1.167e-01) (316,8.917e-02) (1000,5.576e-02) (3162,3.168e-02) (10000,1.819e-02) (31623,9.831e-03) (100000,5.209e-03) (316228,2.668e-03) (1000000,1.340e-03) (3162278,6.636e-04)
};
\addlegendentry{$\tilde{g}_n, d=2$}
		\addplot[dashed,red,no marks,ultra thick] coordinates {
 (100,1.159e-01) (316,8.685e-02) (1000,5.505e-02) (3162,3.119e-02) (10000,1.794e-02) (31623,9.726e-03) (100000,5.126e-03) (316228,2.642e-03) (1000000,1.329e-03) (3162278,6.584e-04)
};
\addlegendentry{$\tilde{a}_{\lfloor n/(4\log n) \rfloor}, d=2$}
\addplot [forget plot,black,domain=1e5:6e6, samples=100, dotted,ultra thick]{1.28*x^(-0.75)*(ln(x))^(2*0.75)};
		\addplot[blue,mark=square,mark size=2.2pt,mark options={solid}] coordinates {
 (100,2.346e-01) (316,1.440e-01) (1000,1.057e-01) (3162,6.969e-02) (10000,4.443e-02) (31623,2.866e-02) (100000,1.704e-02) (316228,1.000e-02) (1000000,5.620e-03) (3162278,3.072e-03)
};
\addlegendentry{$\tilde{g}_n, d=3$}
		\addplot[dashed,blue,no marks,ultra thick] coordinates {
 (100,2.300e-01) (316,1.417e-01) (1000,1.035e-01) (3162,6.843e-02) (10000,4.386e-02) (31623,2.831e-02) (100000,1.686e-02) (316228,9.905e-03) (1000000,5.566e-03) (3162278,3.046e-03)
};
\addlegendentry{$\tilde{a}_{\lfloor n/(4\log n) \rfloor}, d=3$}
\addplot [forget plot,black,domain=4e5:6e6, samples=100, dotted,ultra thick]{0.72*x^(-0.75)*(ln(x))^(3*0.75)};
		\addplot[magenta,mark=o,mark size=2.5pt,mark options={solid}] coordinates {
 (100,3.022e-01) (316,1.744e-01) (1000,1.477e-01) (3162,1.157e-01) (10000,8.021e-02) (31623,5.103e-02) (100000,3.433e-02) (316228,2.309e-02) (1000000,1.436e-02) (3162278,8.661e-03)
};
\addlegendentry{$\tilde{g}_n, d=4$}
		\addplot[dashed,magenta,no marks,ultra thick] coordinates {
 (100,3.013e-01) (316,1.728e-01) (1000,1.459e-01) (3162,1.140e-01) (10000,7.919e-02) (31623,5.044e-02) (100000,3.396e-02) (316228,2.284e-02) (1000000,1.423e-02) (3162278,8.588e-03)
};
\addlegendentry{$\tilde{a}_{\lfloor n/(4\log n) \rfloor}, d=4$}
\addplot [forget plot,black,domain=4e5:6e6, samples=100, dotted,ultra thick]{0.21*x^(-0.7)*(ln(x))^(4*0.7)};
		\addplot[cyan,mark=triangle,mark size=3pt,mark options={solid,rotate=180}] coordinates {
 (100,3.624e-01) (316,2.046e-01) (1000,1.917e-01) (3162,1.441e-01) (10000,1.078e-01) (31623,7.777e-02) (100000,5.699e-02) (316228,3.922e-02) (1000000,2.787e-02) (3162278,1.790e-02)
};
\addlegendentry{$\tilde{g}_n, d=5$}
		\addplot[dashed,cyan,no marks,ultra thick] coordinates {
 (100,3.566e-01) (316,2.011e-01) (1000,1.860e-01) (3162,1.419e-01) (10000,1.061e-01) (31623,7.687e-02) (100000,5.636e-02) (316228,3.882e-02) (1000000,2.762e-02) (3162278,1.775e-02)
};
\addlegendentry{$\tilde{a}_{\lfloor n/(4\log n) \rfloor}, d=5$}
\addlegendimage{empty legend}
\addlegendentry{}
\addplot [black,domain=4e5:6e6, samples=100, dotted,ultra thick]{0.08*x^(-0.75)*(ln(x))^(5*0.75)};
\addlegendentry{$n^{-0.75}(\log n)^{d\cdot 0.75}$}
		\end{axis}
		\end{tikzpicture}
}
\caption{Approximation errors and least squares sampling errors for test function $f\in H^{3/4-\varepsilon}_{\textnormal{mix}}(\tor^d)$.}
\end{figure}

\begin{sloppypar}
In Figure~\ref{fig:num:fourier:H0.75:trunc}, we visualize the relative approximation errors $\tilde{a}_m:=\|f-P_{m-1}f\|_{L_2(\tor^d)}$ for spatial dimensions $d=2,3,4,5$. Due to~\eqref{f27}, these errors should decay like $m^{-0.75+\varepsilon}(\log m)^{(d-1)\cdot (0.75-\varepsilon)}$ for sufficiently large $m$. Correspondingly, we plot $m^{-0.75}(\log m)^{(d-1)\cdot 0.75}$ as black dotted graphs. We observe that the obtained approximation errors nearly decay as the theory suggests.
\end{sloppypar}

Next, we apply Algorithm~\ref{algo1} on the test function $f$ using $n$ randomly selected sampling nodes as sampling scheme. We do not compute the least squares solution directly but use the iterative method LSQR \cite{PaSa82} on the matrix $\bL_m$, $m=\lfloor n/(4\log n) \rfloor$. The obtained sampling errors $\tilde{g}_n:=\|f-S^m_{\bX}f \|_{L_2(\tor^d)}$ are visualized in Figure~\ref{fig:num:fourier:H0.75:sampl} as well as the graphs $\sim n^{-0.75}(\log n)^{d\cdot 0.75}$ as dotted lines which correspond to the theoretical upper bounds $n^{-0.75+\varepsilon}(\log n)^{d\cdot (0.75-\varepsilon)}$, cf.~\eqref{equ:fourier:wc_sampl_error}. We set the tolerance parameter of LSQR to $5\cdot 10^{-8}$ and the maximum number of iterations to 100. For $d=2$ and $d=3$, the errors nearly decay like these bounds. For $d=4$ and $d=5$, the errors seem to decay slightly slower than the bounds. In order to investigate this further, we also plot the corresponding approximation errors $\tilde{a}_m$ with $m=\lfloor n/(4\log n) \rfloor$ as thick dashed lines. We observe that these approximation errors~$\tilde{a}_{m}$, which are the best possible errors that can be achieved in this setting, almost coincide with the sampling errors. This means that we might still observe preasymptotic behavior.

\begin{figure}[htb]
\subfloat[$d=2$]{\label{fig:num:fourier:H0.75:alias_d2}
		\begin{tikzpicture}[baseline]
		\begin{axis}[
			font=\footnotesize,
			enlarge x limits=true,
			enlarge y limits=true,
			height=0.4\textwidth,
			grid=major,
			width=0.47\textwidth,
            xmode=log,
            ymode=log,
			xlabel={$n$},
			ylabel={$L_2$ error},
			legend style={legend cell align=left, at={(1.0,1.4)}},
			legend columns = 1,
		]
		\addplot[red,mark=triangle,mark size=3pt,mark options={solid}] coordinates {
 (100,1.352e-02) (316,2.020e-02) (1000,8.891e-03) (3162,5.531e-03) (10000,2.978e-03) (31623,1.433e-03) (100000,9.252e-04) (316228,3.717e-04) (1000000,1.734e-04) (3162278,8.301e-05)
};
\addlegendentry{aliasing error, random nodes}
		\addplot[red,mark=o,mark size=2.5pt,mark options={solid}] coordinates {
 (127,1.111e-02) (257,5.454e-03) (511,4.893e-03) (1023,2.387e-03) (2049,2.607e-03) (4093,1.731e-03) (8195,1.095e-03) (16387,8.192e-04) (32771,4.486e-04) (65533,4.068e-04) (131075,2.741e-04) (262147,2.141e-04) (524291,1.127e-04) (1048575,9.136e-05) (2097153,6.678e-05)
};
\addlegendentry{aliasing error, Frolov lattice}
		\addplot[red,mark=square,mark size=2.2pt,mark options={solid}] coordinates {
 (144,1.447e-03) (233,1.019e-03) (377,7.798e-04) (610,5.517e-04) (987,5.604e-04) (1597,4.445e-04) (2584,4.455e-04) (4181,3.375e-04) (6765,3.098e-04) (10946,3.190e-04) (17711,2.333e-04) (28657,2.011e-04) (46368,1.973e-04) (75025,1.585e-04) (121393,1.292e-04) (196418,1.195e-04) (317811,1.140e-04) (514229,1.133e-04) (832040,9.042e-05) (1346269,6.744e-05) (2178309,5.671e-05)
};
\addlegendentry{aliasing error, Fibonacci lattice}
		\addplot[dashed,red,no marks,ultra thick] coordinates {
 (100,1.159e-01) (316,8.685e-02) (1000,5.505e-02) (3162,3.119e-02) (10000,1.794e-02) (31623,9.726e-03) (100000,5.126e-03) (316228,2.642e-03) (1000000,1.329e-03) (3162278,6.584e-04)
};
\addlegendentry{$\tilde{a}_{\lfloor n/(4\log n) \rfloor}, d=2$}
		\end{axis}
		\end{tikzpicture}
}
\hfill
\subfloat[$d=3$]{\label{fig:num:fourier:H0.75:alias_d3}
		\begin{tikzpicture}[baseline]
		\begin{axis}[
			font=\footnotesize,
			enlarge x limits=true,
			enlarge y limits=true,
			height=0.4\textwidth,
			grid=major,
			width=0.47\textwidth,
            xmode=log,
            ymode=log,
			xlabel={$n$},
			ylabel={$L_2$ error},
			legend style={legend cell align=left, at={(1.0,1.31)}},
			legend columns = 1,
		]
		\addplot[blue,mark=triangle,mark size=3pt,mark options={solid}] coordinates {
 (100,4.646e-02) (316,2.566e-02) (1000,2.116e-02) (3162,1.321e-02) (10000,7.105e-03) (31623,4.471e-03) (100000,2.479e-03) (316228,1.411e-03) (1000000,7.726e-04) (3162278,4.044e-04)
};
\addlegendentry{aliasing error, random nodes}
		\addplot[blue,mark=o,mark size=2.5pt,mark options={solid}] coordinates {
 (1021,6.501e-03) (2041,4.719e-03) (4093,4.052e-03) (8193,2.120e-03) (16387,1.878e-03) (32773,1.192e-03) (65537,1.053e-03) (131071,6.320e-04) (262149,5.639e-04) (524291,3.907e-04) (1048581,2.979e-04)
};
\addlegendentry{aliasing error, Frolov lattice}
		\addplot[dashed,blue,no marks,ultra thick] coordinates {
 (100,2.300e-01) (316,1.417e-01) (1000,1.035e-01) (3162,6.843e-02) (10000,4.386e-02) (31623,2.831e-02) (100000,1.686e-02) (316228,9.905e-03) (1000000,5.566e-03) (3162278,3.046e-03)
};
\addlegendentry{$\tilde{a}_{\lfloor n/(4\log n) \rfloor}, d=3$}
		\end{axis}
		\end{tikzpicture}
}
\caption{Least squares aliasing errors and approximation errors for test function $f\in H^{3/4-\varepsilon}_{\textnormal{mix}}(\tor^d)$.}
\end{figure}

For $d=2$ spatial dimensions, we have a closer look at the sampling errors. In Figure~\ref{fig:num:fourier:H0.75:alias_d2}, we again plot the approximation errors $\tilde{a}_{m}$, $m=\lfloor n/(4\log n) \rfloor$. In addition, the aliasing errors $\|P_{m-1}f-S^m_{\bX}f \|_{L_2(\tor^d)}$, $m=\lfloor n/(4\log n) \rfloor$, which are the errors caused by Algorithm~\ref{algo1} since $$\|f-S^m_{\bX}f \|_{L_2(\tor^d)}^2 = \|f-P_{m-1}f \|_{L_2(\tor^d)}^2 + \|P_{m-1}f-S^m_{\bX}f \|_{L_2(\tor^d)}^2,$$ are shown as triangles. We observe that the aliasing errors nearly decay like the approximation errors, and that they are by one order of magnitude smaller. This corresponds to the behavior we observed in Figure~\ref{fig:num:fourier:H0.75:sampl}.

Moreover, we compare least squares using random point sets with least squares using quasi-random point sets. In particular, so-called Frolov lattices \cite{KaOeUlUl18}\footnotemark\footnotetext{In our numerical tests, we used the Frolov lattices constructed by the methods presented in \cite{KaOeUlUl18} which has been published at \url{https://ins.uni-bonn.de/content/software-frolov}.} are considered as sampling sets in $d=2$ spatial dimensions and used in Algorithm~\ref{algo1}. The resulting sampling errors almost coincide with the approximation errors. The aliasing errors are visualized in Figure~\ref{fig:num:fourier:H0.75:alias_d2} as circles. It is remarkable that they decay similarly and are even lower than the aliasing errors for random nodes in most cases. A similar behavior can be observed for dimension $d=3$, cf.\ Figure~\ref{fig:num:fourier:H0.75:alias_d3}.

In addition, we consider Fibonacci lattices in dimension $d=2$, cf.\ Figure~\ref{fig:num:fourier:H0.75:alias_d2}. For $n\geq 832040$, the matrices $\bL_m$, $m=\lfloor n/(4\log n) \rfloor$, contain at least two identical columns and correspondingly, the smallest eigenvalue of $\bL_m^{\ast}\,\bL_m$ is zero. Therefore, obtaining $S^m_{\bX}f$ via Algorithm~\ref{algo1} is not possible if the least squares solution is computed directly. An iterative method like LSQR may still work but the number of iterations may have to be restricted. In Figure~\ref{fig:num:fourier:H0.75:alias_d2}, the obtained aliasing errors via LSQR are shown as squares, and they are smaller than in the other cases but they seem to decay slower. However, when we decreased the tolerance parameter of the LSQR algorithm, we observed for $n\geq 832040$ aliasing errors and sampling errors larger than~1.

\subsubsection*{Kink test function $f$ from $H^{3/2-\varepsilon}_{\textnormal{mix}}(\tor^d)$}

Next, we consider the kink test function
$f\colon \tor^d\rightarrow\R,$
$$
f(\bx) = \prod_{i=1}^d \left(\frac{15}{4\sqrt{3}} \cdot 5^{3/4} \cdot \max\left(\frac{1}{5} - \left((x_i\bmod 1)-\frac{1}{2}\right)^2, 0\right)\right)
\in
H^{3/2-\varepsilon}_{\textnormal{mix}}(\tor^d)
$$
with Fourier coefficients
$$
 \hat{f}_{\bk} = \prod_{i=1}^d \begin{cases}
  \frac{5^{5/4}\sqrt{3}}{8} \, (-1)^{k_i} \, \frac{\sqrt{5} \, \sin(2k_i\pi/\sqrt{5}) - 2 k_i \pi \cos(2k_i\pi/\sqrt{5}) }{\pi^3 k_i^3} & \text{for } k_i\neq 0, \\
  \frac{5^{1/4}}{\sqrt{3}}& \text{for } k_i=0.
 \end{cases}
$$
Besides the different test function $f$, we use the same setting as before.

\begin{figure}[htb]
\subfloat[approximation error]{\label{fig:num:fourier:H1.5:trunc}
		\begin{tikzpicture}[baseline]
		\begin{axis}[
			font=\footnotesize,
			enlarge x limits=true,
			enlarge y limits=true,
			height=0.4\textwidth,
			grid=major,
			width=0.47\textwidth,
            xmode=log,
            ymode=log,
			xlabel={$m=|I|$},
			ylabel={$L_2$ approximation error},
			legend style={legend cell align=left, at={(1.0,1.35)}},
			legend columns = 2,
		]
		\addplot[dashed,red,mark=triangle,mark size=3pt,mark options={solid}] coordinates {
 (100,1.255e-02) (316,3.222e-03) (1000,8.070e-04) (3162,1.982e-04) (10000,4.541e-05) (31623,1.024e-05) (100000,2.250e-06) (316228,4.831e-07) (1000000,1.022e-07) %
};
\addlegendentry{$\tilde{a}_{m},d=2$}
\addplot [forget plot,black,domain=1e4:2e6, samples=100, dotted,ultra thick]{5*x^(-1.5)*(ln(x))^(1*1.5)};
		\addplot[dashed,blue,mark=square,mark size=2.2pt,mark options={solid}] coordinates {
 (100,4.979e-02) (316,1.964e-02) (1000,6.992e-03) (3162,2.385e-03) (10000,6.856e-04) (31623,1.959e-04) (100000,5.333e-05) (316228,1.402e-05) (1000000,3.564e-06) (3162278,8.787e-07) (10000000,2.113e-07)
};
\addlegendentry{$\tilde{a}_{m},d=3$}
\addplot [forget plot,black,domain=1e5:2e7, samples=100, dotted,ultra thick]{4*x^(-1.5)*(ln(x))^(2*1.5)};
		\addplot[dashed,magenta,mark=o,mark size=2.5pt,mark options={solid}] coordinates {
 (100,1.549e-01) (316,5.213e-02) (1000,2.445e-02) (3162,1.144e-02) (10000,4.166e-03) (31623,1.456e-03) (100000,4.838e-04) (316228,1.525e-04) (1000000,4.663e-05) (3162278,1.351e-05) (10000000,3.806e-06) (31622777,1.040e-06)
};
\addlegendentry{$\tilde{a}_{m},d=4$}
\addplot [forget plot,black,domain=1e6:1e8, samples=100, dotted,ultra thick]{1.4*x^(-1.5)*(ln(x))^(3*1.5)};
		\addplot[dashed,cyan,mark=triangle,mark size=3pt,mark options={solid,rotate=180}] coordinates {
 (100,3.116e-01) (316,1.420e-01) (1000,6.811e-02) (3162,2.980e-02) (10000,1.431e-02) (31623,6.076e-03) (100000,2.338e-03) (316228,8.965e-04) (1000000,3.178e-04) (3162278,1.070e-04) (10000000,3.488e-05) (31622777,1.096e-05)
};
\addlegendentry{$\tilde{a}_{m},d=5$}
\addplot [black,domain=1e6:1e8, samples=100, dotted,ultra thick]{0.22*x^(-1.5)*(ln(x))^(4*1.5)};
\addlegendentry{$m^{-1.5}(\log m)^{(d-1)\cdot 1.5}$}
		\end{axis}
		\end{tikzpicture}
}
\hfill
\subfloat[sampling error]{\label{fig:num:fourier:H1.5:sampl}
		\begin{tikzpicture}[baseline]
		\begin{axis}[
			font=\footnotesize,
			enlarge x limits=true,
			enlarge y limits=true,
			height=0.4\textwidth,
			grid=major,
			width=0.47\textwidth,
            xmode=log,
            ymode=log,
			xlabel={$n$},
			ylabel={$L_2$ error},
			legend style={legend cell align=left, at={(1.0,1.5)}},
			legend columns = 2,
		]
		\addplot[red,mark=triangle,mark size=3pt,mark options={solid}] coordinates {
 (100,2.918e-01) (316,9.308e-02) (1000,2.531e-02) (3162,1.276e-02) (10000,3.758e-03) (31623,1.121e-03) (100000,3.153e-04) (316228,8.443e-05) (1000000,2.137e-05) (3162278,5.278e-06)
};
\addlegendentry{$\tilde{g}_n, d=2$}
		\addplot[dashed,red,no marks,ultra thick] coordinates {
 (100,2.882e-01) (316,9.006e-02) (1000,2.476e-02) (3162,1.256e-02) (10000,3.709e-03) (31623,1.109e-03) (100000,3.123e-04) (316228,8.361e-05) (1000000,2.117e-05) (3162278,5.233e-06)
};
\addlegendentry{$\tilde{a}_{\lfloor n/(4\log n) \rfloor}, d=2$}
\addplot [forget plot,black,domain=1e5:6e6, samples=100, dotted,ultra thick]{20*x^(-1.5)*(ln(x))^(2*1.5)};
		\addplot[blue,mark=square,mark size=2.2pt,mark options={solid}] coordinates {
(100,5.686e-01) (316,3.496e-01) (1000,1.329e-01) (3162,5.215e-02) (10000,2.131e-02) (31623,9.251e-03) (100000,3.211e-03) (316228,1.128e-03) (1000000,3.694e-04)
};
\addlegendentry{$\tilde{g}_n, d=3$}
		\addplot[dashed,blue,no marks,ultra thick] coordinates {
(100,5.626e-01) (316,3.408e-01) (1000,1.297e-01) (3162,5.137e-02) (10000,2.101e-02) (31623,9.145e-03) (100000,3.172e-03) (316228,1.116e-03) (1000000,3.661e-04)
};
\addlegendentry{$\tilde{a}_{\lfloor n/(4\log n) \rfloor}, d=3$}
\addplot [forget plot,black,domain=2e5:4e6, samples=100, dotted,ultra thick]{6*x^(-1.5)*(ln(x))^(3*1.5)};
		\addplot[magenta,mark=o,mark size=2.5pt,mark options={solid}] coordinates {
 (100,7.204e-01) (316,5.194e-01) (1000,3.013e-01) (3162,1.600e-01) (10000,6.847e-02) (31623,2.853e-02) (100000,1.539e-02) (316228,6.396e-03) (1000000,2.434e-03)
};
\addlegendentry{$\tilde{g}_n, d=4$}
		\addplot[dashed,magenta,no marks,ultra thick] coordinates {
 (100,7.004e-01) (316,5.071e-01) (1000,2.952e-01) (3162,1.573e-01) (10000,6.760e-02) (31623,2.817e-02) (100000,1.521e-02) (316228,6.331e-03) (1000000,2.411e-03)
};
\addlegendentry{$\tilde{a}_{\lfloor n/(4\log n) \rfloor}, d=4$}
\addplot [forget plot,black,domain=4e5:4e6, samples=100, dotted,ultra thick]{0.7*x^(-1.5)*(ln(x))^(4*1.5)};
		\addplot[cyan,mark=triangle,mark size=3pt,mark options={solid,rotate=180}] coordinates {
 (100,8.071e-01) (316,6.242e-01) (1000,5.042e-01) (3162,3.195e-01) (10000,1.658e-01) (31623,7.738e-02) (100000,3.871e-02) (316228,2.051e-02) (1000000,9.692e-03)
};
\addlegendentry{$\tilde{g}_n, d=5$}
		\addplot[dashed,cyan,no marks,ultra thick] coordinates {
 (100,7.876e-01) (316,6.194e-01) (1000,4.937e-01) (3162,3.148e-01) (10000,1.633e-01) (31623,7.644e-02) (100000,3.829e-02) (316228,2.030e-02) (1000000,9.605e-03)
};
\addlegendentry{$\tilde{a}_{\lfloor n/(4\log n) \rfloor}, d=5$}
\addlegendimage{empty legend}
\addlegendentry{}
\addplot [black,domain=4e5:6e6, samples=100, dotted,ultra thick]{0.05*x^(-1.5)*(ln(x))^(5*1.5)};
\addlegendentry{$n^{-1.5}(\log n)^{d\cdot 1.5}$}
		\end{axis}
		\end{tikzpicture}
}
\caption{Approximation errors and least squares sampling errors for test function $f\in H^{3/2-\varepsilon}_{\textnormal{mix}}(\tor^d)$.}
\end{figure}

\begin{sloppypar}
In Figure~\ref{fig:num:fourier:H1.5:trunc}, we visualize the relative approximation errors $\tilde{a}_m:=\|f-P_{m-1}f \|_{L_2(\tor^d)}$ for spatial dimensions $d=2,3,4,5$. These errors should decay like $m^{-1.5+\varepsilon}(\log m)^{(d-1)\cdot (1.5-\varepsilon)}$ for sufficiently large $m$, and we observe that the obtained approximation errors nearly decay as the theoretical results suggest.
Next, we apply Algorithm~\ref{algo1} with random nodes to the test function $f$ using the iterative method LSQR. The resulting sampling errors $\tilde{g}_n:=\|f-S^m_{\bX}f \|_{L_2(\tor^d)}$ are depicted in Figure~\ref{fig:num:fourier:H1.5:sampl}. In addition, the graphs $\sim n^{-1.5}(\log n)^{d\cdot 1.5}$ are shown as dotted lines which roughly correspond to the theoretical upper bounds $n^{-1.5+\varepsilon}(\log n)^{d\cdot (1.5-\varepsilon)}$. The errors seem to decay according to this bound for $d=2$ and slower for $d=3,4,5$. Again, we also plot the corresponding approximation errors $\tilde{a}_m$ with $m=\lfloor n/(4\log n)\rfloor$ as thick dashed lines, and we observe that these approximation errors $\tilde{a}_{\lfloor n/(4\log n) \rfloor}$ almost coincide with the sampling errors. Correspondingly, we still have preasymptotic behavior for $d=3,4,5$.
\end{sloppypar}

\subsubsection*{Test function $f$ from $H^{11/2-\varepsilon}_{\textnormal{mix}}(\tor^d)$}

As a third test function~$f$, we consider an $L_2$-normalized product of periodic one-dimensional B-Splines of order $6$, where each factor depends on a single variable and is a piecewise polynomial of degree $5$, and therefore, we have $f\in H^{11/2-\varepsilon}_{\textnormal{mix}}(\tor^d)$.
 
\begin{figure}[htb]
\subfloat[approximation error]{\label{fig:num:fourier:H5.5:trunc}
		\begin{tikzpicture}[baseline]
		\begin{axis}[
			font=\footnotesize,
			enlarge x limits=true,
			enlarge y limits=true,
			height=0.4\textwidth,
			grid=major,
			width=0.48\textwidth,
            xmode=log,
            ymode=log,
			xlabel={$m=|I|$},
			ylabel={$L_2$ approximation error},
			legend style={legend cell align=left, at={(1.0,1.35)}},
			legend columns = 2,
		]
		\addplot[dashed,red,mark=triangle,mark size=3pt,mark options={solid}] coordinates {
 (100,3.833e-03) (316,1.179e-05) (1000,8.560e-08)
};
\addlegendentry{$\tilde{a}_m$, $d=2$}
\addplot [forget plot,black,domain=1e2:2e3, samples=100, dotted,ultra thick]{5e5*x^(-5.5)*(ln(x))^(1*5.5)};
		\addplot[dashed,blue,mark=square,mark size=2.2pt,mark options={solid}] coordinates {
 (100,2.397e-01) (316,4.165e-02) (1000,2.114e-03) (3162,3.598e-05) (10000,6.222e-07)
};
\addlegendentry{$\tilde{a}_m$, $d=3$}
\addplot [forget plot,black,domain=2e3:2e4, samples=100, dotted,ultra thick]{1.5e6*x^(-5.5)*(ln(x))^(2*5.5)};
		\addplot[dashed,magenta,mark=o,mark size=2.5pt,mark options={solid}] coordinates {
 (100,5.472e-01) (316,2.883e-01) (1000,1.266e-01) (3162,1.649e-02) (10000,3.263e-03) (31623,1.376e-04) (100000,1.716e-06)
};
\addlegendentry{$\tilde{a}_m$, $d=4$}
\addplot [forget plot,black,domain=2e4:2e5, samples=100, dotted,ultra thick]{2e5*x^(-5.5)*(ln(x))^(3*5.5)};
		\addplot[dashed,cyan,mark=triangle,mark size=3pt,mark options={solid,rotate=180}] coordinates {
 (100,7.841e-01) (316,5.711e-01) (1000,3.940e-01) (3162,1.865e-01) (10000,5.702e-02) (31623,1.877e-02) (100000,1.754e-03) (316228,8.122e-05) (1000000,2.098e-06)
};
\addlegendentry{$\tilde{a}_m$, $d=5$}
\addplot [black,domain=2e5:2e6, samples=100, dotted,ultra thick]{2e3*x^(-5.5)*(ln(x))^(4*5.5)};
\addlegendentry{$m^{-5.5}(\log m)^{(d-1)\cdot 5.5}$}
		\end{axis}
		\end{tikzpicture}
}
\hfill
\subfloat[sampling error]{\label{fig:num:fourier:H5.5:sampl}
		\begin{tikzpicture}[baseline]
		\begin{axis}[
			font=\footnotesize,
			enlarge x limits=true,
			enlarge y limits=true,
			height=0.4\textwidth,
			grid=major,
			width=0.48\textwidth,
            xmode=log,
            ymode=log,
			xlabel={$n$},
			ylabel={$L_2$ error},
			legend style={legend cell align=left, at={(1.0,1.5)}},
			legend columns = 2,
		]
		\addplot[red,mark=triangle,mark size=3pt,mark options={solid}] coordinates {
 (100,6.606e-01) (316,3.209e-01) (1000,9.403e-02) (3162,4.025e-03) (10000,3.798e-05) (31623,4.355e-07)
};
\addlegendentry{$\tilde{g}_n, d=2$}
		\addplot[dashed,red,no marks,ultra thick] coordinates {
 (100,6.398e-01) (316,3.151e-01) (1000,9.297e-02) (3162,3.951e-03) (10000,3.741e-05) (31623,4.303e-07)
};
\addlegendentry{$\tilde{a}_{\lfloor n/(4\log n) \rfloor}, d=2$}
\addplot [forget plot,black,domain=5e3:5e4, samples=100, dotted,ultra thick]{8e7*x^(-5.5)*(ln(x))^(2*5.5)};
		\addplot[blue,mark=square,mark size=2.2pt,mark options={solid}] coordinates {
 (100,8.661e-01) (316,7.161e-01) (1000,4.377e-01) (3162,2.501e-01) (10000,4.183e-02) (31623,9.425e-03) (100000,2.674e-04) (316228,4.008e-06) (1000000,3.799e-08)
};
\addlegendentry{$\tilde{g}_n, d=3$}
		\addplot[dashed,blue,no marks,ultra thick] coordinates {
 (100,8.867e-01) (316,7.255e-01) (1000,4.433e-01) (3162,2.536e-01) (10000,4.236e-02) (31623,9.542e-03) (100000,2.704e-04) (316228,4.048e-06) (1000000,3.799e-08)
};
\addlegendentry{$\tilde{a}_{\lfloor n/(4\log n) \rfloor}, d=3$}
		\addplot[magenta,mark=o,mark size=2.5pt,mark options={solid}] coordinates {
 (100,9.457e-01) (316,8.821e-01) (1000,7.460e-01) (3162,5.507e-01) (10000,3.430e-01) (31623,1.333e-01) (100000,5.039e-02) (316228,8.678e-03) (1000000,8.133e-04)
};
\addlegendentry{$\tilde{g}_n, d=4$}
		\addplot[dashed,magenta,no marks,ultra thick] coordinates {
 (100,9.609e-01) (316,9.165e-01) (1000,7.617e-01) (3162,5.549e-01) (10000,3.470e-01) (31623,1.349e-01) (100000,5.098e-02) (316228,8.763e-03) (1000000,8.208e-04)
};
\addlegendentry{$\tilde{a}_{\lfloor n/(4\log n) \rfloor}, d=4$}
		\addplot[cyan,mark=triangle,mark size=3pt,mark options={solid,rotate=180}] coordinates {
 (100,9.774e-01) (316,9.484e-01) (1000,8.924e-01) (3162,7.874e-01) (10000,6.209e-01) (31623,4.247e-01) (100000,2.612e-01) (316228,1.198e-01) (1000000,2.361e-02)
};
\addlegendentry{$\tilde{g}_n, d=5$}
		\addplot[dashed,cyan,no marks,ultra thick] coordinates {
 (100,9.965e-01) (316,9.639e-01) (1000,9.215e-01) (3162,7.953e-01) (10000,6.322e-01) (31623,4.297e-01) (100000,2.638e-01) (316228,1.210e-01) (1000000,2.383e-02)
};
\addlegendentry{$\tilde{a}_{\lfloor n/(4\log n) \rfloor}, d=5$}
\addlegendimage{empty legend}
\addlegendentry{}
\addplot [black,domain=8e4:2e6, samples=100, dotted,ultra thick]{4e7*x^(-5.5)*(ln(x))^(3*5.5)};
\addlegendentry{$n^{-5.5}(\log n)^{d\cdot 5.5}$}
		\end{axis}
		\end{tikzpicture}
}
\caption{Approximation errors and least squares sampling errors for test function $f\in H^{11/2-\varepsilon}_{\textnormal{mix}}(\tor^d)$.}
\end{figure}

In Figure~\ref{fig:num:fourier:H5.5:trunc}, the relative approximation errors $\tilde{a}_m:=\|f-P_{m-1}f \|_{L_2(\tor^d)}$ are visualized for spatial dimensions $d=2,3,4,5$, which roughly decay like $m^{-5.5+\varepsilon}(\log m)^{(d-1)\cdot (5.5-\varepsilon)}$ for sufficiently large $m$.
The sampling errors $\tilde{g}_n:=\|f-S^m_{\bX}f \|_{L_2(\tor^d)}$ when applying Algorithm~\ref{algo1} with random nodes to the test function $f$ using the iterative method LSQR are depicted in Figure~\ref{fig:num:fourier:H5.5:sampl}. In addition, the graphs $\sim n^{-5.5}(\log n)^{d\cdot 5.5}$ are plotted as dotted lines which correspond to the theoretical upper bounds. The errors seem to decay roughly according to this bound for $d=2$ and $d=3$ or slightly slower. Again, the corresponding approximation errors $\tilde{a}_m$ with $m=n/(4\log n)$ are shown as thick dashed lines, and we observe that these approximation errors $\tilde{a}_{\lfloor n/(4\log n) \rfloor}$ almost coincide with the sampling errors. This means we still have preasymptotic behavior.

\subsection{Integration of functions from spaces with mixed smoothness}

Extensive numerical tests on integration using random point sets were performed by Oettershagen~\cite{Oe17}, cf.\ in particular \cite[Sec.~4.1]{Oe17} for tests concerning numerical integration in Sobolev spaces with mixed smoothness $H^{s}_{\textnormal{mix}}(\tor^d)$. These numerical tests, performed for $d\in\{2,4,8,16\}$ spatial dimensions and smoothness $s\in\{1,2,3\}$, suggest that the worst case cubature error may decay with a main rate of $n^{-s}$ with additional log factors. This is a remarkable behavior since plain Monte-Carlo with random points usually leads to an error decay of $n^{-1/2}$.

However, the corresponding theoretical results in \cite[Sec.~4.2]{Oe17} only give a main rate of $n^{-s+1/2}$. %
We highlight that our results obtained in this paper bridge this gap of $1/2$ in the main rate, since we show a worst-case error of $\sim n^{-s}(\log n)^{d\,s}$ in Corollary \ref{cor:fourier:int_error_wc}, i.e., our theoretical main rate corresponds to the observations by Oettershagen~\cite{Oe17}. 
Moreover, Algorithm~\ref{algo1} guarantees suitable error bounds in the preasymptotic setting, cf.\ Corollary~\ref{cor:fourier:int_error_wc} together with \eqref{eq:sigma_upper_bound_kuehn}.
More details on cubature rules based on least squares additionally refined using some variance reduction technique can be found in \cite{MiNo18}.

\subsection{Recovery and integration for the non-periodic case}

Now we consider the non-periodic situation. We use the Chebyshev measure $\tilde{\varrho}_D(\mathrm{d}\bx):=\prod_{t=1}^d(\pi\sqrt{1-x_t^2})^{-1} \, \mathrm{d}\bx$ as random sampling scheme on $D = [-1,1]^d$. We further define the non-periodic space $\tilde{H}^s_{\textnormal{mix}}([-1,1]^d)$ via the reproducing kernel 
\begin{equation*}
	\tilde{K}^1_{s,\ast}(x,y):=1+2\sum\limits_{h\in \N_0} \frac{\cos(h\arccos(x))\cos(h\arccos(y))}{w_{s,\ast}(h)^2}\,,\quad x,y\in \tor\,,
\end{equation*}
and its tensor product
\begin{equation*}
		\tilde{K}^d_{s,\ast}(\bx,\by):=	K^1_{s,\ast}(x_1,y_1)\otimes \cdots \otimes 	K^1_{s,\ast}(x_d,y_d)\,,\quad \bx,\by \in \tor^d\,.
\end{equation*}
The space $\tilde{H}^s_{\textnormal{mix}}([-1,1]^d)$ is embedded into $L_2(D,\tilde{\varrho}_D)$ if and only if $s>1/2$. We denote the $k$-th basis index of $\eta_k$, $k=1,\ldots,m-1$, by $\boldsymbol{h}_k:=\big(h_{k,t}\big)_{t=1}^d\in\N_0^d$, and we have $\eta_1 \equiv 1$ and 
$$\eta_k(\bx)=\prod_{t=1}^d \sqrt{2}^{\min\{h_{k,t},1\}} \cos(h_{k,t}\arccos x_t)$$
if $k>1$.
Moreover, for the vector $\boldsymbol{b}\in\C^{m-1}$ in~\eqref{eq:compute_int_weights:b}, we have
\begin{equation}
 b_k:=\int_D \eta_k\,\mathrm{d}\mu_D = \prod_{t=1}^d
 \begin{cases}
   2 & \text{for } h_{k,t}=0,\\
   \frac{2\sqrt{2}}{1-h_{k,t}^2} & \text{for } h_{k,t}\in 2\N,\\
   0 & \text{for } h_{k,t}\in 2\N-1,
 \end{cases}
\label{equ:bk:chebyshev}
\end{equation}
for $\mu_D\equiv 1$.

\begin{figure}[htb]
	\centering
	\begin{tikzpicture}[font=\scriptsize]
	\begin{axis}[width=0.6\textwidth,height=0.22\textwidth,enlarge x limits=false,enlarge y limits=true,xmin=-3,xmax=2,ymin=0,ymax=1.5,xtick={-1,0,1} %
	]
	\addplot[samples=100,black!70,domain=-3:-2.5] {0};
	\addplot[samples=100,black!70,domain=-2.5:-1] {1.459419514686549*(1/4*(5+2*x)};
    \addplot[samples=100,blue,domain=-1:-0.5,ultra thick] {1.459419514686549*1/4*(5+2*x)};
    \addplot[samples=100,blue,domain=-0.5:1,ultra thick] {1.459419514686549*1/4*(3-2*x)};
	\addplot[samples=100,black!70,domain=1:1.5] {1.459419514686549*1/4*(3-2*x)};
	\addplot[samples=100,black!70,domain=1.5:2] {0};
	\end{axis}
	\end{tikzpicture}
	\vspace{-0.6em}
	\caption{Dilated, scaled, and shifted B-Spline of order 2 considered in interval $[-1,1]$.}
	\label{fig:Bspline2:nonper}
\end{figure}
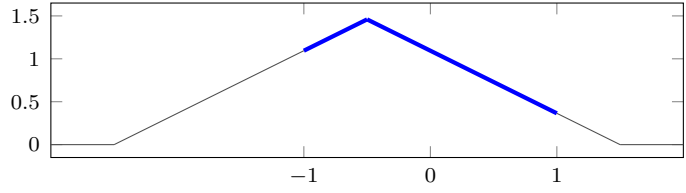
\noindent For the numerical experiments we consider the test function
$$
 f\colon \R\rightarrow\R,
 \quad
 f(\bx)
 :=
 \left(\frac{3\pi}{49\pi-48\sqrt{3}}\right)^{d/2}
 \,
 \prod_{i=1}^d
 \begin{cases}
   5+2x & \text{for } -5/2\leq x_i < -1, \\
   3-2x & \text{for } -1 \leq x_i < 3/2, \\
   0 & \text{otherwise},
  \end{cases}
$$
on $D:=[-1,1]^d$,
where the one-dimensional version is depicted in Figure~\ref{fig:Bspline2:nonper}.
For the Chebyshev coefficients of $f$, we have
$$
 \hat{f}_{\bk}
 =
 \left(\frac{3\pi}{49\pi-48\sqrt{3}}\right)^{d/2}
 \,
 \prod_{i=1}^d \begin{cases}
  4\cdot\dfrac{\sqrt{3}\,k_i\cos(2\pi k_i/3) + \sin(2\pi k_i/3)}{\sqrt{2}(-k_i + k_i^3)\pi} & \text{for } k_i\geq 2, \\
  -(3\sqrt{6} + 2\sqrt{2}\pi)/(6\pi) & \text{for } k_i=1, \\
  11/3 - 2\sqrt{3}/\pi & \text{for } k_i=0,
 \end{cases}
$$
and consequently, $f\in\tilde{H}^{3/2-\varepsilon}_{\textnormal{mix}}([-1,1]^d)$ for $\varepsilon>0$.

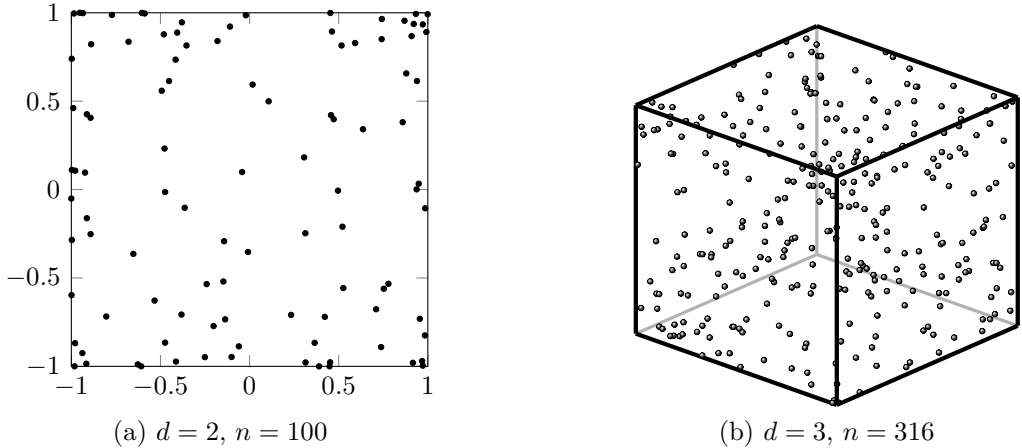
\begin{figure}[htb]
\centering
	\subfloat[$d=2$, $n=100$]{
		\begin{tikzpicture}
		\begin{axis}[
		plot box ratio = 1 1,
		xmin=-1,xmax=1,ymin=-1,ymax=1,
		font=\footnotesize,
		height=0.38\textwidth,
		width=0.38\textwidth,
		]
		\addplot[only marks,black,mark=*,mark size=1pt,mark options={solid}] coordinates {
			(-0.604,0.999) (-0.935,0.998) (-0.382,-0.707) (0.859,0.381) (0.879,0.657) (0.496,-0.006) (0.992,0.891) (0.984,-0.825) (-0.938,-0.925) (0.449,-1.000) (-0.060,-0.887) (-0.978,0.107) (0.949,0.033) (-0.009,-0.353) (0.753,-0.561) (-0.982,-1.000) (-0.679,0.836) (0.457,0.421) (-0.913,-0.162) (0.422,-0.720) (-0.354,0.815) (0.741,0.851) (0.462,0.894) (-0.364,-0.103) (-0.588,0.996) (-0.481,0.878) (0.968,-0.970) (-0.137,-0.734) (-0.180,0.840) (-0.143,-0.292) (0.985,-0.106) (-0.772,0.988) (0.933,0.994) (-0.998,0.740) (-0.477,0.232) (0.742,0.965) (-0.474,-0.014) (0.365,-0.867) (-0.147,-0.520) (-0.999,-0.597) (0.306,0.182) (0.314,-0.978) (0.452,-0.977) (-0.474,-0.866) (-0.021,0.986) (0.592,0.829) (0.939,0.614) (-0.985,0.996) (-0.953,1.000) (0.453,0.999) (0.738,-0.891) (0.999,0.992) (-0.923,0.096) (-0.413,-0.973) (0.526,-0.557) (-0.380,0.946) (0.909,0.868) (-0.608,-1.000) (-0.913,0.426) (0.516,0.815) (-0.804,-0.718) (0.916,-0.980) (0.710,-0.677) (-0.251,-0.948) (-1.000,-0.051) (-0.892,-0.252) (-0.652,-0.364) (0.921,0.937) (0.313,-0.247) (-0.533,-0.628) (-0.627,-0.989) (-0.202,-0.772) (0.955,-0.731) (0.972,-0.996) (0.017,0.594) (0.390,-1.000) (0.636,0.341) (-0.997,-0.285) (0.233,-0.709) (-0.989,0.461) (-0.241,-0.535) (-0.101,-0.947) (-0.997,0.111) (-0.979,-0.869) (0.972,0.934) (0.107,0.499) (-0.406,0.888) (-0.415,0.735) (-0.889,0.822) (-0.892,0.406) (-0.452,0.614) (-0.493,0.559) (0.779,-0.533) (-0.042,0.099) (-0.110,0.922) (-0.916,-0.985) (0.937,0.001) (0.521,-0.210) (0.869,0.955) (0.472,0.398)
		};
		\end{axis}
		\end{tikzpicture}
	}
\hspace{5em}
	\subfloat[$d=3$, $n=316$]{
		\begin{tikzpicture}
		\begin{axis}[axis background/.style={fill=white},
		every axis/.append style={font=\footnotesize},
		width=0.40\textwidth,
		height=0.40\textwidth,
		enlargelimits=false,
		axis lines = none,
		enlargelimits=false,
		clip=false,
		view={42}{25},
		grid=major,
		plot box ratio = 1 1 1,
		clip mode=individual,
		tickwidth=0pt,
		z buffer=sort,
		xmin=-1,xmax=1,
		ymin=-1,ymax=1,
		zmin=-1, zmax=1,
		]
		\addplot3[black!30,very thick,-] coordinates {(-1,1,1) (-1,1,-1)};
		\addplot3[black!30,very thick,-] coordinates {(-1,1,-1) (-1,-1,-1)};
		\addplot3[black!30,very thick,-] coordinates {(-1,1,-1) (1,1,-1)};
		\addplot3[black,ultra thick,-] coordinates {(-1,-1,1) (-1,-1,-1)};
		\addplot3[black,ultra thick,-] coordinates {(-1,-1,-1) (1,-1,-1)};
		\addplot3+[only marks, mark size=1pt, mark=ball, solid, ball color=black!75, mark options={black!75, draw=black}] coordinates{
			(0.999,-0.995,-0.819) (0.264,0.976,0.700) (0.877,0.972,-0.077) (-0.098,-0.431,0.259) (0.142,-0.685,-0.683) (-0.846,0.784,-0.965) (0.410,0.036,-0.593) (0.993,0.195,0.966) (0.107,0.548,-0.092) (-0.064,-0.958,-0.663) (0.797,0.921,0.735) (-0.958,0.984,0.755) (0.931,-0.954,-0.998) (-1.000,0.885,0.930) (0.036,0.960,0.806) (0.532,-0.999,-0.525) (0.205,0.074,0.506) (-0.068,0.455,-0.597) (-0.534,-0.353,0.976) (0.314,0.452,-0.633) (0.963,0.735,-0.863) (0.542,0.236,0.696) (0.986,0.403,-0.072) (-0.952,-0.642,0.464) (0.035,0.876,0.994) (-0.941,-0.228,-0.554) (0.142,0.402,-0.971) (0.559,0.866,0.793) (-0.895,-0.661,-0.410) (0.766,-0.307,0.264) (-0.312,0.170,-0.839) (0.949,-0.175,-0.971) (-0.570,-0.960,-0.886) (-0.982,0.695,0.856) (-0.772,-0.995,-0.585) (-0.966,-0.219,-0.357) (0.988,-0.993,0.456) (-0.999,0.997,0.840) (-0.936,0.070,0.742) (-0.999,0.960,-0.385) (-0.778,-0.359,-0.827) (-0.999,0.935,0.444) (0.433,-0.696,-1.000) (0.053,0.959,-0.630) (-0.717,0.946,0.422) (-0.820,-0.313,0.285) (-0.960,0.383,-0.522) (-0.402,0.211,0.999) (-1.000,0.946,-0.926) (0.317,-0.782,0.675) (-0.066,-1.000,-0.881) (-0.567,0.930,-0.204) (0.695,-0.739,0.864) (0.875,-0.718,-0.426) (-0.679,0.428,-0.971) (0.045,-0.576,0.178) (-0.166,0.854,-0.652) (-0.478,0.934,-0.915) (0.875,0.985,0.587) (0.998,-0.336,-0.044) (0.387,0.811,-0.840) (-0.842,-0.881,-0.950) (-0.926,0.383,-0.806) (0.249,0.585,-0.912) (-0.513,-0.964,-0.526) (-0.666,0.098,0.988) (0.157,0.939,0.101) (-0.843,-0.676,-0.712) (0.532,-0.740,-0.196) (-0.985,0.876,0.507) (-0.994,-0.014,0.149) (-0.622,-0.888,0.336) (-0.971,-0.749,0.909) (0.952,0.476,-0.366) (0.886,0.196,0.705) (-0.322,-0.608,-0.881) (-0.862,0.375,-0.983) (-0.217,0.627,-0.483) (-0.260,-0.636,0.840) (0.658,-0.272,-0.375) (-0.290,-0.419,0.996) (0.989,0.773,0.325) (0.962,-0.911,-0.487) (0.988,0.953,0.975) (-0.465,-0.614,0.950) (-0.612,-0.001,0.892) (0.937,0.778,-0.692) (-0.790,0.528,-0.624) (-0.917,-0.891,0.770) (0.064,0.579,0.591) (0.956,-0.919,-0.024) (-0.301,0.986,0.995) (-0.779,0.937,-0.075) (-0.970,0.047,1.000) (0.978,0.534,-0.030) (0.318,0.997,0.011) (-0.375,-0.923,0.748) (0.495,0.712,-0.344) (0.794,0.856,-0.146) (-0.500,0.773,-0.914) (0.569,-0.826,0.726) (-0.802,-0.222,0.204) (0.992,-0.999,0.450) (0.617,-0.522,0.805) (-0.866,-0.912,-0.770) (-0.179,-0.155,0.902) (0.406,0.225,0.956) (0.976,0.794,-0.608) (0.092,-0.249,0.652) (0.387,0.102,-0.483) (0.538,0.973,-0.655) (0.971,0.950,-0.860) (-0.887,0.140,0.493) (-0.292,0.952,0.112) (-0.556,-0.987,0.859) (0.572,0.825,-0.668) (0.942,-0.496,-0.851) (-0.029,-0.753,-0.215) (0.820,0.848,-0.076) (-0.519,-0.459,0.781) (0.156,-0.868,-0.917) (-0.342,-0.996,-0.987) (0.948,0.990,0.390) (0.525,0.537,-0.996) (0.858,0.906,-0.910) (-0.941,0.982,0.574) (-0.639,0.536,0.049) (0.745,0.967,0.133) (-0.778,-0.692,0.824) (-0.994,0.936,0.797) (-0.198,0.969,0.006) (0.980,-0.588,-0.646) (-0.680,-0.974,-0.930) (0.217,0.745,0.937) (-0.427,-0.965,0.996) (-0.996,0.658,0.984) (0.001,-0.163,0.420) (0.856,0.457,-0.960) (-0.687,-0.862,-0.973) (-0.990,0.099,-0.840) (-0.715,0.408,-0.995) (0.484,-0.999,0.649) (-0.511,-0.093,-0.899) (0.480,0.937,0.414) (-0.999,-0.937,0.790) (0.633,-0.180,0.207) (-0.335,0.481,-0.376) (-0.994,-0.601,0.847) (-0.291,0.988,-0.575) (0.628,0.573,-0.537) (0.970,-0.796,0.780) (0.999,0.171,0.012) (0.999,0.929,-0.682) (0.947,0.181,0.667) (0.912,0.712,-0.872) (0.811,0.946,0.969) (0.172,0.879,-0.766) (-0.252,-0.917,-0.966) (-0.921,-0.829,-0.973) (0.467,-0.886,0.749) (0.910,-0.996,0.840) (-0.650,-1.000,0.686) (0.877,0.298,0.725) (-0.975,0.336,0.993) (-0.372,0.968,-0.105) (-0.666,0.670,0.784) (-0.887,-0.585,-0.241) (0.400,-0.014,0.984) (-0.545,-0.689,-1.000) (0.992,0.668,0.998) (-0.004,-0.933,0.219) (0.681,0.698,-0.464) (-0.999,0.827,0.180) (0.325,0.696,-0.588) (0.238,-0.490,-0.878) (0.822,0.262,0.496) (0.468,-0.097,0.547) (0.928,-0.963,0.795) (0.953,0.859,-0.458) (-0.888,0.552,-0.707) (-0.822,-0.668,-0.468) (0.792,0.995,-0.513) (-0.982,-0.838,0.974) (-0.339,-0.588,0.720) (-0.588,-0.275,-0.617) (0.671,-0.736,0.327) (-0.947,-0.818,0.742) (0.795,0.601,0.885) (0.776,-0.843,0.953) (0.102,0.070,0.296) (-0.700,-0.597,-0.977) (-0.975,-0.117,-0.774) (0.385,0.973,0.917) (-0.997,-0.779,-0.266) (-0.981,-1.000,0.485) (-0.918,0.957,0.777) (-0.155,0.536,-0.712) (0.515,0.047,0.170) (0.900,0.069,-0.751) (-0.005,-0.948,-1.000) (-0.988,1.000,0.743) (-0.552,-0.429,-0.998) (0.430,0.383,-0.330) (0.783,-0.329,-0.701) (-0.615,-1.000,-0.801) (-0.007,0.142,0.488) (0.842,0.029,0.964) (0.294,-0.714,-1.000) (0.114,-0.990,0.699) (-0.507,0.879,0.808) (-0.274,-0.989,0.958) (0.798,-0.546,0.997) (-0.473,0.777,-0.930) (-0.117,0.996,0.060) (-0.403,0.268,0.848) (0.535,0.767,0.417) (0.211,0.656,0.990) (0.459,0.457,0.986) (-0.028,0.494,0.995) (-0.314,0.982,-0.978) (-0.071,0.110,-0.195) (-0.978,-0.314,-0.534) (-0.276,0.510,-0.913) (0.995,-0.186,-0.912) (0.772,-0.914,0.820) (0.659,-0.190,0.078) (0.741,0.149,0.985) (0.722,-0.324,0.999) (0.881,-0.314,-0.279) (0.750,0.710,-0.394) (0.998,0.070,-0.757) (-0.939,-0.030,-0.169) (-0.321,-0.961,-0.670) (-0.025,-0.705,-0.128) (-0.689,-0.164,-0.209) (-0.993,0.420,-0.685) (0.997,0.612,-0.913) (0.653,1.000,-0.251) (-0.894,0.977,0.022) (-0.894,-0.958,-0.524) (0.600,0.989,0.548) (-0.365,-0.872,-0.178) (-0.971,0.996,-0.201) (-0.276,0.300,0.796) (-0.901,-0.728,0.837) (0.807,-0.976,0.712) (0.976,-0.588,0.789) (-1.000,0.833,0.004) (0.454,-0.971,0.172) (0.532,-0.441,0.399) (-0.699,-0.915,-0.953) (0.881,-0.879,-0.717) (-0.603,-0.981,0.356) (-0.996,0.592,-0.575) (-0.957,-0.963,0.744) (0.805,0.359,-0.668) (0.616,-0.855,-0.992) (0.777,0.458,-0.596) (0.380,-0.004,0.176) (0.853,-0.626,0.995) (-0.994,0.862,0.560) (-0.293,0.842,0.632) (-0.708,0.981,0.184) (-0.060,-1.000,0.105) (0.931,0.427,0.925) (-0.636,-1.000,-0.330) (0.306,-0.945,-0.131) (-0.997,-0.944,-0.794) (-0.883,-0.405,0.723) (0.867,-0.730,0.146) (-0.818,-0.465,-0.433) (0.860,-0.708,0.590) (-0.487,-0.347,-0.942) (0.512,-0.246,-1.000) (-0.986,0.426,0.301) (0.671,0.785,0.479) (-0.002,-0.661,0.249) (0.984,0.882,0.960) (-0.613,0.686,0.414) (-0.378,0.854,0.080) (-0.937,-0.841,-0.685) (0.994,-0.979,-0.987) (-0.210,0.340,0.666) (-0.965,0.592,-0.870) (0.629,-0.939,-0.636) (0.021,-0.955,-0.686) (-0.514,0.303,0.147) (0.715,0.831,0.129) (0.870,0.835,0.787) (-0.699,-0.961,0.917) (0.108,0.664,0.785) (-0.202,0.955,0.995) (-0.880,-0.150,-0.844) (0.310,-0.975,0.976) (0.981,-0.726,-0.990) (-0.943,0.666,-0.728) (0.951,-0.992,-1.000) (0.158,-0.235,0.513) (-0.226,0.126,0.956) (-0.888,0.989,0.226) (-0.713,-0.792,-0.883) (-0.165,-0.934,0.007) (0.103,0.443,0.211) (-0.441,-0.943,-0.830) (0.757,0.999,-0.525) (0.932,-0.711,0.628) (0.671,0.328,0.166) (-0.290,0.839,-0.902) (0.321,0.861,-0.367) (0.429,-0.034,-0.319) (-0.921,0.731,0.239) (-0.012,0.965,-0.998) (-0.561,0.952,0.241) (0.244,-0.782,0.977) (0.984,0.147,-0.944) (-0.228,0.780,0.666)
		};
		\addplot3[black,ultra thick,-] coordinates {(-1,-1,1) (-1,1,1)};
		\addplot3[black,ultra thick,-] coordinates {(-1,1,1) (1,1,1)};
		\addplot3[black,ultra thick,-] coordinates {(1,1,1) (1,-1,1)};
		\addplot3[black,ultra thick,-] coordinates {(1,-1,1) (-1,-1,1)};
		\addplot3[black,ultra thick,-] coordinates {(1,-1,1) (1,-1,-1)};
		\addplot3[black,ultra thick,-] coordinates {(1,-1,-1) (1,1,-1)};
		\addplot3[black,ultra thick,-] coordinates {(1,1,-1) (1,1,1)};
		\end{axis}
		\end{tikzpicture}
	}
	\caption{Realizations of random nodes with respect to the Chebyshev measure $\tilde{\varrho}_D(\bx)$.}\label{fig:rand_nodes_chebyshev}
\end{figure}

\begin{figure}[htb]
\subfloat[sampling and approximation eror]{\label{fig:Bspline2:nonper:errors:sampl}
		\begin{tikzpicture}[baseline]
		\begin{axis}[
			font=\footnotesize,
			enlarge x limits=true,
			enlarge y limits=true,
			height=0.4\textwidth,
			grid=major,
			width=0.47\textwidth,
            xmode=log,
            ymode=log,
			xlabel={$n$},
			ylabel={$L_2$ error},
			legend style={legend cell align=left, at={(1.0,1.55)}},
			legend columns = 2,
		]
		\addplot[red,mark=triangle,mark size=3pt,mark options={solid}] coordinates {
 (100,1.885e-01) %
 (316,5.042e-02) %
 (1000,1.371e-02) %
 (3162,3.994e-03) %
 (10000,1.202e-03) %
 (31623,3.257e-04) %
 (100000,8.407e-05) %
 (316228,2.114e-05) %
};
\addlegendentry{$\tilde{g}_n, d=2$}
		\addplot[dashed,red,no marks,ultra thick] coordinates {
 (100,1.840e-01) (316,4.933e-02) (1000,1.349e-02) (3162,3.935e-03) (10000,1.187e-03) (31623,3.219e-04) (100000,8.321e-05) (316228,2.094e-05)
};
\addlegendentry{$\tilde{a}_{\lfloor n/(4\log n) \rfloor}, d=2$}
\addplot [forget plot,black,domain=5e3:4e5, samples=100, dotted,ultra thick]{3*x^(-1.5)*(ln(x))^(2*1.5)};
		\addplot[blue,mark=square,mark size=2.2pt,mark options={solid}] coordinates {
 (100,3.146e-01) %
 (316,1.590e-01) %
 (1000,5.400e-02) %
 (3162,2.046e-02) %
 (10000,7.250e-03) %
 (31623,2.386e-03) %
 (100000,7.643e-04) %
 (316228,2.315e-04) %
};
\addlegendentry{$\tilde{g}_n, d=3$}
		\addplot[dashed,blue,no marks,ultra thick] coordinates {
 (100,3.072e-01) (316,1.556e-01) (1000,5.301e-02) (3162,2.015e-02) (10000,7.161e-03) (31623,2.359e-03) (100000,7.567e-04) (316228,2.294e-04)
};
\addlegendentry{$\tilde{a}_{\lfloor n/(4\log n) \rfloor}, d=3$}
\addplot [forget plot,black,domain=1e4:4e5, samples=100, dotted,ultra thick]{0.75*x^(-1.5)*(ln(x))^(3*1.5)};
		\addplot[magenta,mark=o,mark size=2.5pt,mark options={solid}] coordinates {
 (100,3.959e-01) %
 (316,2.826e-01) %
 (1000,1.219e-01) %
 (3162,5.982e-02) %
 (10000,2.364e-02) %
 (31623,9.075e-03) %
 (100000,3.532e-03) %
 (316228,1.236e-03) %
};
\addlegendentry{$\tilde{g}_n, d=4$}
		\addplot[dashed,magenta,no marks,ultra thick] coordinates {
 (100,3.856e-01) (316,2.771e-01) (1000,1.195e-01) (3162,5.880e-02) (10000,2.334e-02) (31623,8.972e-03) (100000,3.497e-03) (316228,1.225e-03)
};
\addlegendentry{$\tilde{a}_{\lfloor n/(4\log n) \rfloor}, d=4$}
\addplot [forget plot,black,domain=2e4:4e5, samples=100, dotted,ultra thick]{0.08*x^(-1.5)*(ln(x))^(4*1.5)};
		\addplot[cyan,mark=triangle,mark size=3pt,mark options={solid,rotate=180}] coordinates {
 (100,5.176e-01) %
 (316,3.964e-01) %
 (1000,2.393e-01) %
 (3162,1.230e-01) %
 (10000,5.373e-02) %
 (31623,2.605e-02) %
 (100000,1.074e-02) %
 (316228,4.316e-03) %
};
\addlegendentry{$\tilde{g}_n, d=5$}
		\addplot[dashed,cyan,no marks,ultra thick] coordinates {
 (100,5.050e-01) (316,3.890e-01) (1000,2.347e-01) (3162,1.209e-01) (10000,5.299e-02) (31623,2.573e-02) (100000,1.063e-02) (316228,4.276e-03)
};
\addlegendentry{$\tilde{a}_{\lfloor n/(4\log n) \rfloor}, d=5$}
\addplot [forget plot,black,domain=5e3:4e5, samples=100, dotted,ultra thick]{7e-3*x^(-1.5)*(ln(x))^(5*1.5)};
\addlegendimage{empty legend}
\addlegendentry{}
\addlegendimage{black,dotted,ultra thick}
\addlegendentry{$\sim n^{-1.5} (\log n)^{d\cdot 1.5}$}
		\end{axis}
		\end{tikzpicture}
}
\hfill
\subfloat[integration error]{\label{fig:Bspline2:nonper:errors:int}
		\begin{tikzpicture}[baseline]
		\begin{axis}[
			font=\footnotesize,
			enlarge x limits=true,
			enlarge y limits=true,
			height=0.4\textwidth,
			grid=major,
			width=0.47\textwidth,
            xmode=log,
            ymode=log,
			xlabel={$n$},
			ylabel={integration error},
			legend style={legend cell align=left, at={(1.0,1.35)}},
			legend columns = 2,
		]
		\addplot[red,mark=triangle,mark size=3pt,mark options={solid},error bars/.cd, y dir=plus, y explicit, error bar style={solid}, error mark options={rotate=90,mark size=2,thick}] coordinates {
(1000,2.023e-03) -= (0,2.018e-03) += (0,5.542e-03) %
(3162,3.552e-04) -= (0,3.533e-04) += (0,6.312e-04) %
(10000,5.814e-05) -= (0,5.663e-05) += (0,1.450e-04) %
(31623,9.367e-06) -= (0,9.036e-06) += (0,2.033e-05) %
(100000,1.226e-06) -= (0,1.224e-06) += (0,2.543e-06) %
(316228,1.922e-07) -= (0,1.891e-07) += (0,4.727e-07) %
};
\addlegendentry{$d=2$}
\addplot [forget plot,black,domain=3e4:4e5, samples=100, dotted,ultra thick]{1.5*x^(-2)*(ln(x))^(2*2)};
		\addplot[blue,mark=square,mark size=2.2pt,mark options={solid},error bars/.cd, y dir=plus, y explicit, error bar style={solid}, error mark options={rotate=90,mark size=2,thick}] coordinates {
(1000,1.854e-02) -= (0,1.844e-02) += (0,4.534e-02) %
(3162,3.843e-03) -= (0,3.789e-03) += (0,1.046e-02) %
(10000,8.027e-04) -= (0,7.778e-04) += (0,2.430e-03) %
(31623,1.574e-04) -= (0,1.569e-04) += (0,3.758e-04) %
(100000,3.039e-05) -= (0,3.006e-05) += (0,5.760e-05) %
(316228,5.312e-06) -= (0,5.232e-06) += (0,1.139e-05) %
};
\addlegendentry{$d=3$}
\addplot [forget plot,black,domain=3e4:4e5, samples=100, dotted,ultra thick]{0.2*x^(-2)*(ln(x))^(3*2)};
		\addplot[magenta,mark=o,mark size=2.5pt,mark options={solid},error bars/.cd, y dir=plus, y explicit, error bar style={solid}, error mark options={rotate=90,mark size=2,thick}] coordinates {
(1000,4.286e-01) -= (0,2.755e-01) += (0,2.062e-01) %
(3162,4.287e-02) -= (0,4.248e-02) += (0,8.150e-02) %
(10000,8.071e-03) -= (0,7.769e-03) += (0,2.773e-02) %
(31623,1.513e-03) -= (0,1.510e-03) += (0,3.316e-03) %
(100000,3.119e-04) -= (0,3.118e-04) += (0,6.185e-04) %
(316228,7.721e-05) -= (0,7.502e-05) += (0,1.744e-04) %
};
\addlegendentry{$d=4$}
\addplot [forget plot,black,domain=3e4:4e5, samples=100, dotted,ultra thick]{0.023*x^(-2)*(ln(x))^(4*2)};
		\addplot[cyan,mark=triangle,mark size=3pt,mark options={solid,rotate=180},error bars/.cd, y dir=plus, y explicit, error bar style={solid}, error mark options={rotate=90,mark size=2,thick}] coordinates {
(1000,1.341e+00) -= (0,1.080e+00) += (0,1.449e+00) %
(3162,1.188e+00) -= (0,3.182e-01) += (0,3.419e-01) %
(10000,1.351e-01) -= (0,1.108e-01) += (0,9.192e-02) %
(31623,1.377e-02) -= (0,1.359e-02) += (0,3.556e-02) %
(100000,2.618e-03) -= (0,2.612e-03) += (0,5.817e-03) %
(316228,6.181e-04) -= (0,6.166e-04) += (0,1.677e-03) %
};
\addlegendentry{$d=5$}
\addplot [forget plot,black,domain=3e4:4e5, samples=100, dotted,ultra thick]{0.0015*x^(-2)*(ln(x))^(5*2)};
\addlegendimage{black,dotted,ultra thick}
\addlegendentry{$\sim n^{-2} (\log n)^{d\cdot 2}$}
\addplot [forget plot,black,domain=3e4:4e5, samples=100, solid,ultra thick]{0.005*x^(-1.5)*(ln(x))^(2*1.5)};
\addplot [forget plot,black,domain=3e4:4e5, samples=100, solid,ultra thick]{0.01*x^(-1.5)*(ln(x))^(5*1.5)};
\addlegendimage{black,solid,ultra thick}
\addlegendentry{$\sim n^{-1.5} (\log n)^{d\cdot 1.5}$}
		\end{axis}
		\end{tikzpicture}
}
\caption{Sampling errors and integration errors for non-periodic test function $f\in\tilde{H}^{3/2-\varepsilon}_{\textnormal{mix}}([-1,1]^d)$.}
\end{figure}
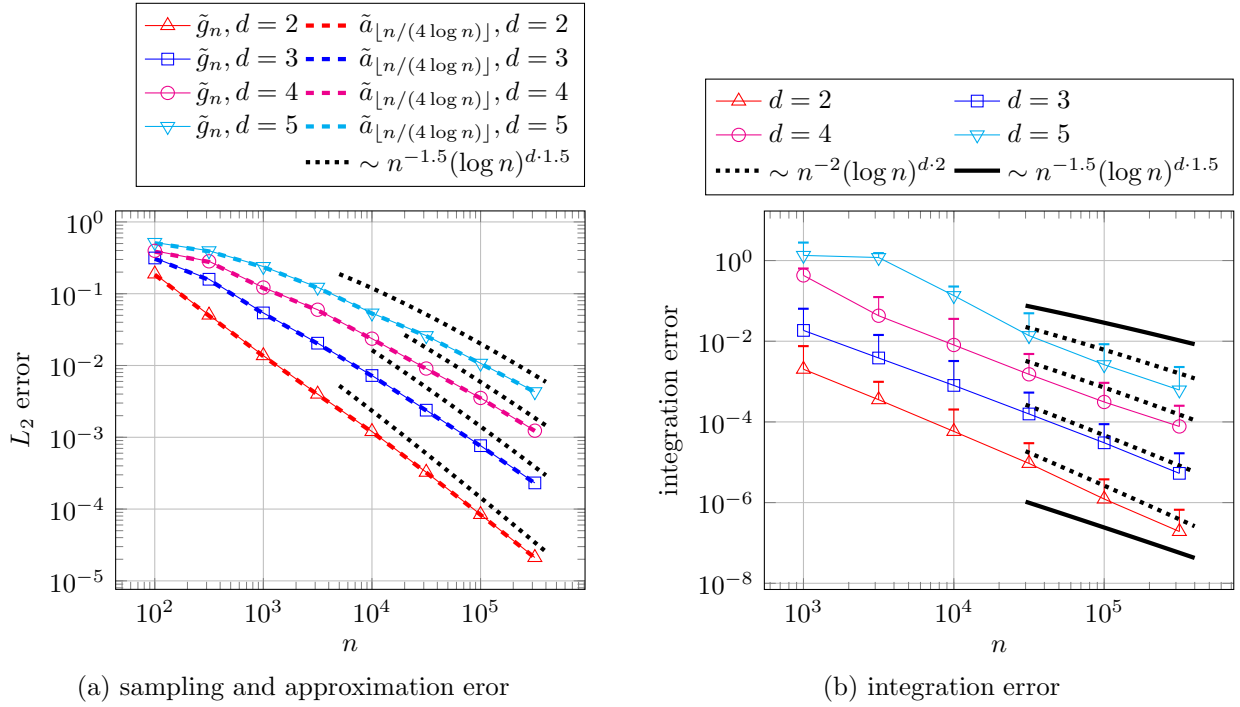

We use the parameters as in Section~\ref{sec:num:fourier} and apply Algorithm~\ref{algo1} on the test function~$f$ in the non-periodic setting, where we generate the random nodes with respect to the measure $\tilde{\varrho}_D(\mathrm{d}\bx):=\prod_{i=1}^d(\pi\sqrt{1-x_i^2})^{-1} \, \mathrm{d}\bx$.
In Figure~\ref{fig:rand_nodes_chebyshev}, we show realizations for these random nodes.
As before, we do not compute the least squares solution directly but use the iterative method LSQR on the matrix $\bL_m$, $m=\lfloor n/(4\log n)\rfloor$. The obtained sampling errors $\tilde{g}_n:=\|f-S^m_{\bX}f\|_{L_2([-1,1]^d,\varrho_D)}$ with $m=\lfloor n/(4\log n)\rfloor$ are plotted in Figure~\ref{fig:Bspline2:nonper:errors:sampl} as triangles as well as the corresponding approximation errors $\tilde{a}_m:=\|f-P_{m-1} f \|_{L_2([-1,1]^d,\varrho_D)}$ as thick dashed lines. We observe that the sampling and approximation errors almost coincide. Moreover, we plot the graphs $\sim n^{-1.5}(\log n)^{d\cdot 1.5}$ as dotted lines which correspond to the expected theoretical upper bounds $n^{-1.5+\varepsilon}(\log n)^{d\cdot (1.5-\varepsilon)}$. We observe that the obtained numerical errors nearly decay like these theoretical upper bound.

In addition, we use the numerically computed Chebyshev coefficients $c_k$ from Algorithm~\ref{algo1} to compute the approximation $Q_\bX f$ of $I(f)$ by 
$Q_\bX f=\int_D S_\bX^mf\,\mathrm{d}\mu_D=\sum_{k=1}^m c_k \, b_k$
where the complex numbers $b_k$ are calculated as stated in~\eqref{equ:bk:chebyshev}.
We repeatedly perform each test 100 times with different random nodes. The averages for the integration errors $|I(f) - Q_\bX f|$ of the 100 test runs are depicted as triangles in Figure~\ref{fig:Bspline2:nonper:errors:int} and the maxima as error bars. Moreover, we plot the graphs $\sim n^{-2}(\log n)^{d\cdot 2}$ as dotted lines, and we observe that the obtained integration errors approximately decay like these graphs.
For comparison, we also plot $n^{-1.5} (\log n)^{d\cdot 1.5}$ for $d=2$ and $d=5$ as thick solid lines which belong to the theoretical results $n^{-1.5+\varepsilon} (\log n)^{d\cdot (1.5-\varepsilon)}$ one obtains analogously to~\eqref{equ:fourier:wc_int_error} in Section~\ref{mixed1} for the non-periodic case. These thick solid lines decay distinctly slower.

In particular, we strongly expect that the theoretical preasymptotic results in~%
\eqref{eq:sigma_upper_bound_kuehn} and~\cite{Kr18} also hold for the Chebyshev case.

\paragraph{Acknowledgment.} The authors would like to thank Bastian Bohn, Michael Griebel, Karlheinz Gr\"ochenig, Aicke Hinrichs, Boris S.\ Kashin, David Krieg, Thomas K\"uhn, Laura Lippert, Erich Novak, Jens Oettershagen, Peter Oswald, Daniel Potts, Holger Rauhut, Ingo Steinwart, Martin Stoll, Vladimir N.\ Temlyakov and Mario Ullrich for several discussions on the topic.
L.K.\ gratefully acknowledges the funding by the Deutsche Forschungsgemeinschaft (DFG, German Research Foundation, project number 380648269).
T.U.\ would like to acknowledge support by the DFG Ul-403/2-1.
T.V.\ gratefully acknowledges support by the SAB 100378180.

\end{document}